\documentclass[12pt,a4paper]{amsart}

\usepackage{amsfonts,graphics,amsmath,amsthm,amsfonts,amscd,amssymb,amsmath,latexsym,multicol,
mathrsfs}
\usepackage{epsfig,url}
\usepackage{flafter}
\usepackage{fancyhdr}
\usepackage{hyperref}
\hypersetup{colorlinks=true, linkcolor=black}
 \usepackage[matrix, arrow]{xy}

\DeclareMathOperator{\Div}{Div}

\DeclareMathOperator{\Supp}{Supp}

\DeclareMathOperator{\vol}{vol}

\DeclareMathOperator{\Ivol}{Ivol}

 \numberwithin{equation}{subsection}
 \numberwithin{footnote}{subsection}

 \newtheorem{lem}[subsection]{Lemma}
 \newtheorem{prop}[subsection]{Proposition}
 \newtheorem{thm}[subsection]{Theorem}

{
\theoremstyle{upright}
 \newtheorem{defn}[subsection]{Definition}

}

 \newcommand{\N}{\mathbb N}
 \newcommand{\PP}{\mathbb P}
 
 \newcommand{\Q}{\mathbb Q}
 \newcommand{\R}{\mathbb R}
 \newcommand{\Z}{\mathbb Z}
  
 \newcommand{\bir}{\dashrightarrow}
 \newcommand{\rddown}[1]{\left\lfloor{#1}\right\rfloor} 

\title{\large B\MakeLowercase{oundedness and volume of generalised pairs}}
\author{\large C\MakeLowercase{aucher} B\MakeLowercase{irkar}}
\date{\today}
\begin{document}
\maketitle

\begin{abstract}
In this paper we investigate boundedness and volume of generalised pairs and 
give applications to usual algebraic varieties and pairs, especially to a class of pairs that we call 
stable minimal models. 
 
We prove a result about descent of nef divisors to bounded families. This is the key to proving  
various other results in this paper and elsewhere. 

As an application, 
fixing the dimension and a DCC set controlling coefficients,   
we will show that the set of volumes of all projective generalised lc 
pairs $(X,B+M)$, under the given data, satisfies the descending chain condition (DCC). 
Futhermore, we will show that in the klt case, the set of such pairs with ample $K_X+B+M$ 
and fixed volume forms a bounded family.

We will then apply the above to study projective lc pairs $(X,B)$ with abundant $K_X+B$ 
of arbitrary Kodaira dimension. In particular, we show that the set of Iitaka volumes of such pairs 
satisfies DCC under some natural boundedness assumptions on the fibres of the Iitaka fibration. 

We define strongly stable minimal models. Fixing appropriate invariants,
we show that such models form a bounded family. This and other results of this paper are crucial ingredients of the construction of moduli spaces of stable minimal models in a sequel work.   

\end{abstract}

\tableofcontents


\section{\bf Introduction}

We work over an algebraically closed field $k$ of characteristic zero.

\vspace{0.5cm}
{\textbf{\sffamily{Motivation and discussion of results.}}} 
The classification theory of higher dimensional algebraic varieties proceeds by finding special models in 
each birational equivalence class and then classifying such special models. In practice, such special 
models are good minimal models and Mori fibre spaces, that is, normal projective varieties $X$ with 
mild singularities such that either $K_X$ is semi-ample or $X$ admits a Mori-Fano fibration 
$X\to T$. Classification of such $X$ in practice means construction of their possible moduli spaces.
Usually the first step of this moduli theory is to fix certain invariants of $X$ and then show that the $X$ 
with these invariants form a bounded family. A similar philosophy is applicable to pairs.

A particular class of such $X$ consists of those with ample $K_X$. In dimension one, these correspond   
to curves of genus $\ge 2$ which form a bounded family for each fixed genus. It turns out that in higher dimension 
the right invariant to fix is the volume $(K_X)^{\dim X}$. More generally, pairs 
$(X,B)$ with mild singularities of fixed dimension such that $K_X+B$ is ample with fixed volume 
$(K_X+B)^{\dim X}$ and such that the coefficients of $B$ are in some fixed finite set, form a bounded family 
[\ref{HMX3}]. 

When $K_X$ is semi-ample but not ample, we get a contraction $f\colon X\to Z$ onto a normal variety.
The canonical bundle formula 
$$
K_X\sim_\Q f^*(K_Z+B_Z+M_Z)
$$
expresses $K_X$ in terms of $K_Z$, the discriminant divisor $B_Z$ and the moduli divisor $M_Z$. 
To understand $X$ one then needs to understand not only $Z$ but the structure $(Z,B_Z+M_Z)$ 
which is a \emph{generalised pair} with ample $K_Z+B_Z+M_Z$. Many problems about the geometry of the minimal model $X$ reduces to questions about the fibres of $f$ and about the generalised pair $(Z,B_Z+M_Z)$. 

A similar construction applies when $(X,B)$ is a good minimal model pair. More generally, any log Calabi-Yau fibration induces a generalised pair structure on its base. So understanding the geometry of generalised pairs is naturally very important. 

In this paper we study boundedness and volume properties of generalised pairs in a general context. The results proved here are crucial ingredients of the results in [\ref{B-moduli}] on the boundedness of stable minimal models and existence of their compact moduli spaces, and the results in [\ref{BH-variation}] as well as more recent work [\ref{Xiaowei}][\ref{Minzhe}].

Generalised pairs, introduced in [\ref{BZh}], have found numerous applications in higher dimensional algebraic geometry in recent years (including the results on Fano's in [\ref{B-compl}][\ref{B-BAB}]), and they are now central to mainstream research in the field; see [\ref{B-gen-pairs}] for a survey. Roughly speaking, a projective generalised pair $(X,B+M)$ 
consists of a projective pair $(X,B)$ and a nef divisor $M'$ on some birational model $X'\to X$; we will refer to $M'$ as the nef part of the generalised pair; the pushdown of $M'$ to $X$ is denoted $M$. Note that the notion of generalised pair is much more general than those derived from the canonical bundle formula above.  For the precise definition, see \ref{ss-gpp}. 
When $M'=0$, we get a pair in the traditional sense.  It is helpful to keep in mind the example above that arises on the base of a fibration (caution: we are now using $X$  rather than $Z$ to denote the underlying variety of the generalised pair).

The purpose of this paper is twofold. On the one hand, we want to investigate the geometry of 
generalised pairs. On the other hand, we want to apply the results on generalised pairs to 
understand the geometry of usual pairs. 
 
In the case of generalised pairs, we aim to study volumes and 
boundedness of such pairs.  
Let's fix numbers $d,p\in \N$ and $v\in \Q^{>0}$. Consider the set of  
projective generalised pairs $(X,B+M)$ with data $X'\to X,M'$ such that 
\begin{itemize}
\item $(X,B+M)$ is generalised lc of dimension $d$,

\item $pB$ is integral,

\item $pM'$ is Cartier, and  

\item $K_X+B+M$ is big. 
\end{itemize}
The first question is whether the set of the volumes  
$\vol(K_X+B+M)$ of such pairs satisfies the DCC. 
The second question concerns boundedness of such generalised pairs. Obviously, 
boundedness does not hold in such a generality. We need to impose further restrictions. 
Consider the set of generalised pairs as above with the additional property that 
\begin{itemize}
\item 
$K_X+B+M$ is ample with $\vol(K_X+B+M)=v$. 
\end{itemize}
The question then is whether such pairs form a 
bounded family.

These questions are important as they naturally appear 
in the context of varieties with semi-ample canonical divisor discussed above. 

For usual pairs, that is when $M'=0$, the answer to the above two questions are affirmative as confirmed in  [\ref{HMX2}][\ref{HMX3}]. 
Unfortunately, as often happens, the case of generalised pairs cannot 
simply be reduced to the case of 
usual pairs. The nef divisor $M'$ creates serious complications. 

In dimension 2,  [\ref{Filipazzi}] answered the first question affirmatively and answered the second question partially by taking advantage of special features of the geometry of surfaces. 

The main strategy of the proofs in [\ref{HMX2}][\ref{HMX3}] for usual pairs is to show first that the 
$(X,B)$ with bounded volume $\vol(K_X+B)$ are birational 
to models $(\overline{X}, \overline{B})$ which are log smooth and which belong to a bounded family. 
They then use boundedness of 
$(\overline{X}, \overline{B})$ to reduce the questions to log fibres of some fixed 
fibration.  

One can try to apply the same strategy in the context of generalised pairs by first finding a 
bounded birational model $(\overline{X}, \overline{B})$. Such models exist   
in any dimension, by the results of [\ref{BZh}]. However, unlike in the case of usual pairs, 
we also need to control the nef divisor $M'$. In other words, we need to choose $\overline{X}$ so that 
$M'$ descends to it, that is, assuming $X'\bir \overline{X}$ is a morphism, we want $M'$ to be the pullback 
of a divisor $\overline{M}$ on $\overline{X}$. Moreover, we need $\overline{M}$ to be bounded 
in some sense. 

In this paper we devise completely new techniques to ensure that in 
higher dimension the nef divisor $M'$ can be controlled, that is, it descends to a bounded model, in a quite general context (see Theorems \ref{t-descent-nef-divs-to-bnd-models}, \ref{t-descent-nef-divs-to-bnd-models-1}). 
This is one of the key results of this paper which opens the door to studying many problems 
about generalised pairs.

We then give affirmative answers to the first question on DCC of volumes (Theorem \ref{t-dcc-vol-gen-pairs}) and to  
the second question on boundedness in the generalised klt case (Theorem \ref{t-bnd-gen-pairs-vol=v}). 
It turns out unexpectedly that boundedness does not hold in the generalised lc case [\ref{BH-variation}].

We will then apply the above to study projective lc pairs $(X,B)$ with abundant $K_X+B$ 
of arbitrary Kodaira dimension. To be precise, we assume $K_X+B\sim_\Q 0/Z$ for some 
contraction $f\colon X\to Z$ where the Kodaira dimension $\kappa(K_X+B)=\dim Z$. 
We show that the set of Iitaka volumes of such pairs 
satisfies DCC under some natural boundedness assumptions on the fibres of $f$ (Theorem \ref{t-dcc-iitaka-volumes}) and this plays a key role in the moduli of minimal models [\ref{B-moduli}].

We then define a special class of stable minimal models, that is, lc strongly stable minimal models which consist of a projective lc pair $(X,B)$ 
with semi-ample $K_X+B$ together with a divisor $A\ge 0$ so that $K_X+B+A$ is ample 
and $A$ does not contain any non-klt centre of $(X,B)$. 
Fixing appropriate invariants we show that lc strongly stable minimal models form a bounded family 
(Theorems \ref{t-bnd-stable-mmodels-lc} and \ref{t-bnd-stable-mmodels-klt}). 
More precisely, we fix the dimension, the Iitaka volume of $K_X+B$ which defines a contraction 
$f\colon X\to Z$, the volume of $A|_F$ for general fibres $F$ of $f$, and the function $\vol(K_X+B+tA)$ for $t\in [0,1]$. 
This result is used in [\ref{B-moduli}] to prove boundedness of semi-log canonical stable minimal models and to construct their moduli spaces.   

In the rest of this introduction we state our main results more precisely. Although it is possible to state variants of these results with a simpler notation but due to necessity of applications we will state them in their most general forms and this makes the notation a little involved.

\vspace{0.5cm}
{\textbf{\sffamily{DCC of volume of generalised pairs.}}} 
Our first main result concerns DCC of volumes.
We introduce some notation. Let $d\in \N$ and $\Phi\subset \R^{\ge 0}$.
Let $\mathcal{G}_{glc}(d,\Phi)$ be the set of projective generalised pairs $(X,B+M)$ 
with data $X'\overset{\phi}\to X$ and $M'$ where 
\begin{itemize}
\item $(X,B+M)$ is generalised lc of dimension $d$, 
\item the coefficients of $B$ are in $\Phi$, 
\item $M'=\sum \mu_iM_i'$ where $M_i'$ are nef Cartier and $0<\mu_i\in \Phi$, and 
\item $K_X+B+M$ is big.\\
\end{itemize}

\begin{thm}\label{t-dcc-vol-gen-pairs}
Let $d\in \N$ and let $\Phi\subset \R^{\ge 0}$ be a DCC set. 
Then 
$$
\{\vol(K_X+B+M) \mid (X,B+M)\in \mathcal{G}_{glc}(d,\Phi) \}
$$
satisfies the DCC.
\end{thm}

This was proved in [\ref{HMX2}, Theorem 1.3][\ref{HMX1}] for usual pairs, that is, when the nef part $M'=0$ 
(the dimension 2 case was proved in [\ref{Alexeev-K^2}]). 
For generalised pairs of dimension 2, it was proved in [\ref{Filipazzi}] 
 assuming all $\mu_i=\frac{1}{p}$ for a fixed number $p\in\N$. 

The theorem is a consequence of Theorem \ref{t-descent-nef-divs-to-bnd-models} below.

\vspace{0.5cm}
{\textbf{\sffamily{Descent of nef divisors to bounded models.}}}
The next main result is about descent of nef divisors. It essentially says that the nef parts of 
generalised pairs with volume bounded from above descend to bounded birational models. This is perhaps the most important result of this paper in the sense that all the other main results are based on this and also in the sense that its proof is the hardest and contains the most new ideas. The proof builds on boundedness of crepant models established in [\ref{B-lcyf}].

For $v\in \R^{>0}$, let 
$$
\mathcal{G}_{glc}(d,\Phi,v)\subset \mathcal{G}_{glc}(d,\Phi)
$$
consist of those 
$
 (X,B+M)
$
that have 
$
\vol(K_X+B+M)=v.
$ 
Similarly, let 
$$
\mathcal{G}_{glc}(d,\Phi,<\!\!v)\subset \mathcal{G}_{glc}(d,\Phi)
$$ 
consist of those  
$
 (X,B+M) 
$
such that the volume 
$
\vol(K_X+B+M)<v.
$\\

\begin{thm}\label{t-descent-nef-divs-to-bnd-models}
Let $d\in \N$, $\Phi\subset \R^{\ge 0}$ be a DCC set, and $v\in \R^{>0}$.
Then there is a bounded set of couples $\mathcal{P}$ depending only on $d,\Phi,v$ satisfying the following. Assume that 
$$
(X,B+M) \in \mathcal{G}_{glc}(d,\Phi,<\!\!v)
$$
with data $X'\overset{\phi}\to X$ and $M'=\sum \mu_iM_i'$.
Then there exist a log smooth couple $(\overline{X},\overline{\Sigma})\in \mathcal{P}$ and a birational map 
$\overline{X}\bir X$ such that 
\begin{itemize}
\item $\overline{\Sigma}$ contains the exceptional divisors of $\overline{X}\bir X$ and the support of the birational 
transform of $B$, 

\item each $M_i'$ descends to $\overline{X}$, say as $\overline{M}_{i}$, and 

\item $\overline{H}-\sum \mu_{i}\overline{M}_{i}$ is pseudo-effective
 for some very ample divisor $\overline{H}\le \overline{\Sigma}$.
\end{itemize}
\end{thm}


In this paper, a couple consists of a projective variety and a reduced divisor on it. 

We will prove another result (Theorem \ref{t-descent-nef-divs-to-bnd-models-1}) 
on descent of nef divisors that may be more useful in certain situations. 

\vspace{0.5cm}
{{\textbf{\sffamily{Boundedness of generalised pairs.}}}
Another application of Theorem \ref{t-descent-nef-divs-to-bnd-models} is boundedness of 
generalised pairs under mild conditions.
Let 
$$
 \mathcal{F}_{glc}(d,\Phi,v)\subset \mathcal{G}_{glc}(d,\Phi,v)
$$
consist of those $(X,B+M)$ with ample $K_X+B+M$. 
Define 
$\mathcal{F}_{gklt}(d,\Phi,v)$ 
similarly by requiring the pairs $(X,B+M)$ to be generalised klt.

\begin{thm}\label{t-bnd-gen-pairs-vol=v}
Let $d\in \N$, $\Phi\subset \R^{\ge 0}$ be a DCC set, and $v\in \R^{>0}$. 
Then the set
$
\mathcal{F}_{gklt}(d,\Phi,v)
$ 
forms a bounded family.
\end{thm}  

If $\Phi\subset \Q^{\ge 0}$ and $v\in \Q^{>0}$, then the proof of the theorem shows that there is a fixed $l\in\N$ such that $l(K_X+B+M)$ is very ample.


It turns out that boundedness does not hold in the generalised lc case, that is, the theorem does not hold for $\mathcal{F}_{glc}(d,\Phi,v)$ [\ref{BH-variation}].
In this case, instead of boundedness we prove the following crucial result 
which is a key ingredient of the proof of Theorem \ref{t-bnd-gen-pairs-vol=v} and also allows us 
to treat generalised lc pairs that appear in canonical bundle formulae, hence 
in particular treat semi-ample pairs discussed above. 

Given a projective generalised pair 
$$
(X,B+M)\in \mathcal{G}_{glc}(d,\Phi)
$$ 
with data $X'\to X$ and $M'=\sum \mu_i M_i'$, we let 
$$
C(X,B+M)=\{\mu_i \mid M_i'\not\equiv 0\}.
$$

\begin{thm}\label{t-disc-for-F-glc}
Let $d\in \N$, $\Phi\subset \R^{\ge 0}$ be a DCC set, and $v\in \R^{>0}$. 
Then the set of generalised log discrepancies 
$$
\{a(D,X,B+M)\le 1\mid (X,B+M)\in \mathcal{F}_{glc}(d,\Phi,v), ~~~\mbox{$D$ prime divisor over $X$}\}
$$
is finite. Moreover, 
$$
\bigcup _{(X,B+M)\in \mathcal{F}_{glc}(d,\Phi,v)} C(X,B+M)
$$
is finite.
\end{thm}

Here $a(D,X,B+M)$ denotes the generalised log discrepancy of $D$ with respect to $(X,B+M)$.
In particular, the first part of the theorem implies that the set of coefficients of all $B$ for 
$$
(X,B+M)\in \mathcal{F}_{glc}(d,\Phi,v)
$$ 
is finite.  

\vspace{0.5cm}
{\textbf{\sffamily{DCC of Iitaka volumes.}}}
We would now like to go beyond the general type case and consider pairs 
of any Kodaira dimension. For the definition of Iitaka volume, see \ref{ss-Iitaka-volumes}.  
Let $d\in \N$ and let $\Phi\subset \Q^{\ge 0}$ be a DCC set. Consider the set 
 of projective klt pairs $(X,B)$ of dimension $d$ where the 
coefficients of $B$ are in $\Phi$. It is conjectured in [\ref{Z-Li}] that the set of Iitaka volumes 
of all such pairs satisfies the DCC. A natural approach is to use the canonical bundle formula to reduce to the DCC for the base of the Iitaka fibration associated to $(X,B)$ and apply   
 Theorem \ref{t-dcc-vol-gen-pairs}. 
However, this is currently out of reach in such generality even when $\dim Z=1$. 
We need to make some kind of boundedness assumption on the fibres of the Iitaka fibration. 

First we fix some notation which is inspired 
by the point of view and the results of [\ref{B-pol-var}] and by moduli of stable 
minimal models [\ref{B-moduli}].

\begin{defn}
\emph{
Let $d\in \N$, $\Phi\subset \Q^{\ge 0}$, and $u\in \Q^{>0}$.
Let $\mathcal{I}_{lc}(d,\Phi,u)$ be the set of projective pairs $(X,B)$ such that 
\begin{itemize}
\item $(X,B)$ is lc of dimension $d$,
\item the coefficients of $B$ are in $\Phi$,
\item $f\colon X\to Z$ is a contraction with $K_X+B\sim_\Q 0/Z$,
\item $\kappa(K_X+B)=\dim Z$, and 
\item there is an integral divisor $A\ge 0$ on $X$ such that over some non-empty 
open subset of $Z$ we have: $(X,B+tA)$ is lc for some $t>0$ and $A$ is ample, and 
\item $\vol(A|_F)=u$ for the general fibres $F$ of $f$.\
\end{itemize}
Define $\mathcal{I}_{lc}(d,\Phi,<\!\!u)$ similarly by replacing the condition $\vol(A|_F)=u$ with 
$\vol(A|_F)<u$. And define $\mathcal{I}_{klt}(d,\Phi,u)$ and $\mathcal{I}_{klt}(d,\Phi,<\!\!u)$ similarly 
by replacing the lc condition of $(X,B)$ with klt.
}
\end{defn}

When $\Phi$ is DCC, the divisor $A$ ensures that if $(X,B)$ belongs to  $\mathcal{I}_{lc}(d,\Phi,u)$, $\mathcal{I}_{klt}(d,\Phi,u)$ or $\mathcal{I}_{klt}(d,\Phi,<\!\!u)$, then  
the log general fibres $(F,B_F)$ belong to a bounded family. 
This follows from [\ref{B-pol-var}, Corollaries 1.6 and 1.8].

\begin{thm}\label{t-dcc-iitaka-volumes}
Let $d\in \N$, $\Phi\subset \Q^{\ge 0}$ be a DCC set, and $u\in \Q^{>0}$.
Then the sets 
$$
\{ \Ivol(K_X+B) \mid (X,B)\in \mathcal{I}_{klt}(d,\Phi,<\!\!u)\}
$$
and
$$
\{ \Ivol(K_X+B) \mid (X,B)\in \mathcal{I}_{lc}(d,\Phi,u)\}
$$
satisfy the DCC.
\end{thm}  

The theorem plays an important role in the proof of boundedness of stable minimal models in [\ref{B-moduli}].

We will see that the assumptions of the theorem ensure that we get the generalised pair $(Z,B_Z+M_Z)$ 
from the adjunction 
$$
K_X+B\sim_\Q f^*(K_Z+B_Z+M_Z)
$$ 
where the coefficient of $B_Z$ belong to a fixed DCC set and that 
$pM_{Z'}$ is Cartier for some fixed $p$ on some high resolution $Z'\to Z$, so we can apply 
Theorem \ref{t-dcc-vol-gen-pairs}. Moreover, fixing the Iitaka volume in the klt case, 
Theorem \ref{t-bnd-gen-pairs-vol=v} 
shows that the log canonical model of $(X,B)$ which is given by the Proj of the 
log canonical ring of $(X,B)$ belongs to a bounded family.

The special case of the theorem when $(X,B)$ is $\epsilon$-lc for fixed $\epsilon>0$ 
and $B$ is big over $Z$ was proved 
in Li [\ref{Z-Li}] using quite different arguments. In this case we can derive $A$ naturally from $-K_X$ where  
$\vol(A|_F)$ is bounded from 
above as $F$ belongs to a bounded family by BAB [\ref{B-BAB}].

Also, Jiao [\ref{Jiao}] independently proved DCC of Iitaka volumes for pairs that are similar to those in  
$\mathcal{I}_{klt}(d,\Phi,<\!\!u)$, using methods that are quite different from ours.

\vspace{0.5cm}
{\textbf{\sffamily{Boundedness of strongly stable minimal models.}}}
We will recall the definition of stable minimal models in the lc case from [\ref{B-moduli}]. We mainly treat strongly stable minimal models which were first defined in this paper. The class of stable minimal models includes both the so-called KSBA-stable pairs of 
general type and stable Calabi-Yau pairs (i.e. polarised Calabi-Yau pairs) 
as special cases. They are defined so that we can address moduli of 
minimal models of arbitrary Kodaira dimension. 

In the general type case, say when $(X,B)$ is a projective lc pair with ample $K_X+B$, to ensure 
boundedness of $(X,B)$ it is enough to fix the dimension of $X$, let coefficients of $B$ 
be in a fixed DCC set, and let the volume of $K_X+B$ be fixed. In the stable Calabi-Yau case, 
say when $(X,B)$ is a projective lc log Calabi-Yau pair, to ensure boundedness we 
need to fix the dimension, let the coefficients of $B$ be in a fixed DCC set, and in addition 
have an ample integral divisor $A\ge 0$ on $X$ with fixed volume so that $(X,B+tA)$ is 
lc for some $t>0$. In the intermediate Kodaira dimension case, we need more delicate conditions that 
are roughly a mixture of the conditions of the general type and the Calabi-Yau cases.

\begin{defn}\label{d-stabl-mmodels}
\emph{
(1) An \emph{lc stable minimal model} $(X,B),A$ consists of a projective pair $(X,B)$ and an  
$\R$-divisor $A\ge 0$ such that   
\begin{itemize}
\item $(X,B)$ is lc,
\item $K_X+B$ is semi-ample,
\item $K_X+B+tA$ is ample for some $t>0$, and 
\item $(X,B+tA)$ is lc for some $t>0$.\
\end{itemize}
We say $(X,B),A$ is an \emph{lc strongly stable minimal model} if in addition $K_X+B+A$ is ample.}
\end{defn}
(2) 
Let $d\in \N$, $\Phi\subset \Q^{\ge 0}$, $u\in \Q^{>0}$, and $\sigma\in \Q[t]$ be a polynomial.
An \emph{lc $(d,\Phi,u,\sigma)$-stable minimal model} is an lc stable minimal model $(X,B),A$ such that    
\begin{itemize}
\item $\dim X=d$,
\item the coefficients of $B$ and $A$ are in $\Phi$,
\item $\vol(A|_F)=u$ where $F$ is a general fibre of the contraction $f\colon X\to Z$ 
defined by $K_X+B$, and 
\item we have 
$$
\vol(K_X+B+tA)=\sigma(t), \ \ \forall \ 0\le t\ll 1.
$$
\end{itemize}
 Let $\mathcal{S}_{lc}(d,\Phi,u,\sigma)$ be the set of all $(d,\Phi,u,\sigma)$-stable minimal models.  Let $\mathcal{S}^s_{lc}(d,\Phi,u,\sigma)$ be the set of all the $(d,\Phi,u,\sigma)$-stable minimal models which are strongly stable.

\begin{thm}\label{t-bnd-stable-mmodels-lc}
Let $d\in \N$, $\Phi\subset \Q^{\ge 0}$ be a DCC set, $u\in \Q^{>0}$, and $\sigma\in \Q[t]$ 
be a polynomial. Then $\mathcal{S}^s_{lc}(d,\Phi,u,\sigma)$ is a bounded family.
\end{thm}

The more general result of boundedness of semi-log canonical stable minimal models (without the strongly assumption) is one of the main results of [\ref{B-moduli}] which relies on the above theorem. For the definition of boundedness and for more results in this direction, see Section 8.



\vspace{0.5cm}
{\textbf{\sffamily{Acknowledgements.}}}
This work was done at the University of Cambridge with support of the Royal Society. It was revised at Tsinghua University. A workshop 
was organised in Shanghai on the results of this paper. Thanks to 
Jingjun Han and the participants of that workshop for their comments. Thanks to 
Xiaowei Jiang for his corrections. Thanks to the referee for the comments and suggestions that improved the exposition considerably.

\section{\bf Preliminaries}

All the varieties in this paper are quasi-projective over a fixed algebraically closed field $k$ of characteristic zero
unless stated otherwise. The set of natural numbers $\N$ is the set of positive integers, so it does not contain $0$.

\subsection{Contractions}
In this paper a \emph{contraction} refers to a projective morphism $f\colon X\to Y$ of varieties 
such that $f_*\mathcal{O}_X=\mathcal{O}_Y$ ($f$ is not necessarily birational). In particular, $f$ has connected fibres and 
if $X\to Z\to Y$ is the Stein factorisation of $f$, then $Z\to Y$ is an isomorphism.

\subsection{Divisors}
Let $X$ be a normal variety and let $M$ be an $\R$-divisor on $X$. 
We denote the coefficient of a prime divisor $D$ in $M$ by $\mu_DM$.   
    
Now let $f\colon X\to Z$ be a morphism to a normal variety. We say $M$ is \emph{horizontal} over $Z$ 
if the induced map $\Supp M\to Z$ is dominant, otherwise we say $M$ is \emph{vertical} over $Z$. If $N$ is an $\R$-Cartier 
divisor on $Z$, we sometimes denote $f^*N$ by $N|_X$. 

Again let $f\colon X\to Z$ be a morphism to a normal variety and let $M$ and $L$ be $\R$-Cartier divisors on $X$. 
We say $M\sim L$ over $Z$ (resp. $M\sim_\Q L$ over $Z$)(resp. $M\sim_\R L$ over $Z$) if there is a Cartier  
(resp. $\Q$-Cartier)(resp. $\R$-Cartier) divisor $N$ on $Z$ such that $M-L\sim f^*N$  
(resp. $M-L\sim_\Q f^*N$)(resp. $M-L\sim_\R f^*N$). For a point $z\in Z$, we say $M\sim L$ over $z$ if   
$M\sim L$ over $Z$ perhaps after shrinking $Z$ around $z$. The properties $M\sim_\Q L$ and $M\sim_\R L$ over $z$ 
are similarly defined.

For a birational map $X\bir X'$ (resp. $X\bir X''$)(resp. $X\bir X'''$)(resp. $X\bir Y$) 
whose inverse does not contract divisors and for 
an $\R$-divisor $M$ on $X$ we usually denote the pushdown of $M$ to $X'$ (resp. $X''$)(resp. $X'''$)(resp. $Y$) 
by $M'$ (resp. $M''$)(resp. $M'''$)(resp. $M_Y$).

\subsection{b-divisors}\label{ss-b-divisor}

We recall some definitions regarding b-divisors. 
A \emph{b-$\R$-divisor} $\bf M$ over a normal variety $X$ consists of an $\R$-divisor ${\bf M}_Y$ on each birational model $Y\to X$ so that if $Y\to X$ and $Y'\to X$ are two birational models with corresponding divisors ${\bf M}_Y, {\bf M}_{Y'}$ and so that the induced map $Y'\bir Y$ is a morphism, then ${\bf M}_Y$ is the pushdown of ${\bf M}_{Y'}$. 
A \emph{b-$\R$-divisor} $\bf M$ is \emph{b-$\R$-Cartier} if there is a birational model $Y\to X$ so that ${\bf M}_Y$ is $\R$-Cartier and so that for any birational model $Y'\to Y$, ${\bf M}_{Y'}$ is the pullback of ${\bf M}_Y$. 

To define generalised pairs below we will use b-divisors but with different notation. We just need to keep in mind that a b-$\R$-Cartier b-$\R$-divisor over $X$ is determined by the choice of a birational model $Y\to X$ and an $\R$-Cartier $\R$-divisor $M$ on $Y$ up to the following equivalence: 
 another birational model $Y'\to X$ and $\R$-Cartier $\R$-divisor
$M'$ defines the same b-$\R$-Cartier  b-$\R$-divisor if there is a common resolution $W\to Y$ and $W\to Y'$ on which the pullbacks of $M$ and $M'$ coincide.

\subsection{Pairs}
In this paper a \emph{sub-pair} $(X,B)$ consists of a normal quasi-projective variety $X$ and an $\R$-divisor 
$B$ such that $K_X+B$ is $\R$-Cartier. 
If the coefficients of $B$ are at most $1$ we say $B$ is a 
\emph{sub-boundary}, and if in addition $B\ge 0$, 
we say $B$ is a \emph{boundary}. A sub-pair $(X,B)$ is called a \emph{pair} if $B\ge 0$ (we allow coefficients 
of $B$ to be larger than $1$ for practical reasons).

Let $\phi\colon W\to X$ be a log resolution of a sub-pair $(X,B)$. Let $K_W+B_W$ be the 
pulback of $K_X+B$. The \emph{log discrepancy} of a prime divisor $D$ on $W$ with respect to $(X,B)$ 
is $1-\mu_DB_W$ and it is denoted by $a(D,X,B)$.
We say $(X,B)$ is \emph{sub-lc} (resp. \emph{sub-klt})(resp. \emph{sub-$\epsilon$-lc}) 
if $a(D,X,B)$ is $\ge 0$ (resp. $>0$)(resp. $\ge \epsilon$) for every $D$. When $(X,B)$ 
is a pair we remove the sub and say the pair is lc, etc. Note that if $(X,B)$ is an lc pair, then 
the coefficients of $B$ necessarily belong to $[0,1]$. Also if $(X,B)$ is $\epsilon$-lc, then 
automatically $\epsilon\le 1$ because $a(D,X,B)=1$ for most $D$. 

Let $(X,B)$ be a sub-pair. A \emph{non-klt place} of $(X,B)$ is a prime divisor $D$ on 
birational models of $X$ such that $a(D,X,B)\le 0$. A \emph{non-klt centre} is the image on 
$X$ of a non-klt place. When $(X,B)$ is lc, a non-klt centre is also called an 
\emph{lc centre}.

\subsection{Minimal model program (MMP)}\label{ss-MMP} 
We will use standard results of the minimal model program (cf. [\ref{kollar-mori}][\ref{BCHM}]). 
Assume $(X,B)$ is a pair and $X\to Z$ is a projective morphism. 
 Assume $H$ is an ample$/Z$ $\R$-divisor and that $K_X+B+H$ is nef$/Z$. Suppose $(X,B)$ is klt or 
that it is $\Q$-factorial dlt. We can run an MMP$/Z$ on $K_X+B$ with scaling of $H$. If 
$(X,B)$ is klt and if either $K_X+B$ or $B$ is big$/Z$, then the MMP terminates [\ref{BCHM}]. If $(X,B)$ 
is $\Q$-factorial dlt and if $K_X+B$ is not pseudo-effective$/Z$, then the MMP terminates [\ref{BCHM}]. 
If $(X,B)$ is $\Q$-factorial dlt and if $K_X+B$ is pseudo-effective$/Z$, then in general we do not know whether the MMP terminates but we know that in some step of the MMP we reach a model $Y$ on which $K_Y+B_Y$, 
the pushdown of $K_X+B$, is a limit of movable$/Z$ $\R$-divisors: indeed, if the MMP terminates, then 
the claim is obvious; otherwise after finitely many steps the MMP produces an infinite sequence $X_i\bir X_{i+1}$ 
of flips and a decreasing sequence $\alpha_i$ of numbers in $(0,1]$ such that 
$K_{X_i}+B_i+\alpha_iH_i$ is nef$/Z$; by [\ref{BCHM}][\ref{B-lc-flips}, Theorem 1.9], $\lim\alpha_i=0$; 
in particular, if $Y:=X_1$, then $K_Y+B_Y$ is the limit of the movable$/Z$ $\R$-divisors 
$K_Y+B_Y+\alpha_i H_Y$.

\subsection{Generalised pairs}\label{ss-gpp}
For the basic theory of generalised pairs, see [\ref{BZh}, Section 4].
Below we recall some of the main notions and discuss some basic properties.\

(1)
A \emph{generalised pair} consists of 
\begin{itemize}
\item a normal variety $X$ equipped with a projective
morphism $X\to Z$, 

\item an $\R$-divisor $B\ge 0$ on $X$, and 

\item a b-$\R$-Cartier  b-divisor over $X$ represented 
by some birational contraction $X' \overset{\phi}\to X$ and $\R$-Cartier divisor
$M'$ on $X'$
\end{itemize}
such that $M'$ is nef$/Z$ and $K_{X}+B+M$ is $\R$-Cartier, where $M := \phi_*M'$. 

We usually refer to the pair by saying $(X,B+M)$ is a  generalised pair with 
data $X'\overset{\phi}\to X\to Z$ and $M'$, and refer to $M'$ as the \emph{nef part}. Since a b-$\R$-Cartier b-divisor is defined birationally, 
in practice we will often replace $X'$ with a resolution and replace $M'$ with its pullback.
When $Z$ is not relevant we usually drop it
 and do not mention it: in this case one can just assume $X\to Z$ is the identity. 
When $Z$ is a point we also drop it but say the pair is projective. 

Now we define generalised singularities.
Replacing $X'$ we can assume $\phi$ is a log resolution of $(X,B)$. We can write 
$$
K_{X'}+B'+M'=\phi^*(K_{X}+B+M)
$$
for some uniquely determined $B'$. For a prime divisor $D$ on $X'$, the \emph{generalised log discrepancy} 
$a(D,X,B+M)$ is defined to be $1-\mu_DB'$. 

We say $(X,B+M)$ is 
\emph{generalised lc} (resp. \emph{generalised klt})(resp. \emph{generalised $\epsilon$-lc}) 
if for each $D$ the generalised log discrepancy $a(D,X,B+M)$ is $\ge 0$ (resp. $>0$)(resp. $\ge \epsilon$).
A \emph{generalised non-klt place} of $(X,B+M)$ is a prime divisor 
$D$ on birational models of $X$ with $a(D,X,B+M)\le 0$, and 
a \emph{generalised non-klt centre} of $(X,B+M)$ is the image of a generalised non-klt place.
The \emph{generalised non-klt locus}  of the generalised pair is the union of all 
the generalised non-klt centres.  

We will also use similar definitions when $B$ is not necessarily effective in which case we have a 
generalised sub-pair. 

(2)
Let $(X,B+M)$ be a generalised pair as in (1).
We say $(X,B+M)$ is \emph{generalised dlt} if it is generalised lc and if 
$\eta$ is the generic point of any generalised non-klt centre of 
$(X,B+M)$, then $(X,B)$ is log smooth near $\eta$ and $M'=\phi^*M$  holds over a neighbourhood of $\eta$.  
Note that when $M'=0$, then $(X,B)$ is {generalised dlt} iff it is dlt in the usual sense. 

The generalised dlt property is preserved under the MMP [\ref{B-compl}, 2.13(2)]. 

(3)
Let $(X,B+M)$ be a generalised pair as in (1) and let $\psi\colon X''\to X$ be a projective birational 
morphism from a normal variety. Replacing $\phi$ we can assume $\phi$ factors through 
$\psi$. We then let $B''$ and $M''$ be the pushdowns of 
$B'$ and $M'$ on $X''$ respectively. In particular,  
$$
K_{X''}+B''+M''=\psi^*(K_{X}+B+M).
$$
If $B''\ge 0$, then $(X'',B''+M'')$ is also a generalised pair 
with data $X'\to X''\to Z$ and $M'$. 

Assume that we can write $B''=\Delta''+G''$ where $(X'',\Delta''+M'')$ is $\Q$-factorial generalised dlt, 
$G''\ge 0$ is supported in $\rddown{\Delta''}$, and 
every exceptional prime divisor of $\psi$ is a component of $\rddown{\Delta''}$. Then we say 
$(X'',\Delta''+M'')$ is a \emph{$\Q$-factorial generalised dlt model} of $(X,B+M)$. Such models exist by 
[\ref{B-non-klt}, Lemma 2.8] (also see [\ref{HMX2}, Proposition 3.3.1] and [\ref{BZh}, Lemma 4.5]). 
If $(X,B+M)$ is generalised lc, then $G''=0$.

(4) The next lemma compares generalised log discrepancies of a generalised pair with log discrepancies of its underlying pair under mild assumptions.

\begin{lem}\label{l-comapring-log-disc}
Assume $(X,B+M)$ is a generalised pair with data $X'\overset{\phi}\to X$ and $M'$ and assume $K_X+B$ is $\R$-Cartier. 
Write $\phi^*M=M'+\sum e_iE_i$ where $E_i$ run through the exceptional divisors of $\phi$. Then for each $i$, we have 
$$
a(E_i,X,B)=a(E_i,X,B+M)+e_i.
$$
\end{lem}
\begin{proof}
Write 
$$
K_{X'}+C'=\phi^*(K_{X}+B)
$$
and 
$$
K_{X'}+B'+M'=\phi^*(K_{X}+B+M).
$$
Clearly $B'=C'+\sum e_iE_i$. Thus 
$$
a(E_i,X,B+M)=1-\mu_{E_i}B'=1-\mu_{E_i}C'-e_i=a(E_i,X,B)-e_i.
$$
\end{proof}

\subsection{Bounded families of generalised pairs}\label{ss-bnd-gen-pairs}
Let $d\in \N$ and $\Phi\subset \R^{\ge 0}$.
Assume that $0$ is not an accumulation point of $\Phi$, e.g. when $\Phi$ is DCC. 
We say that a subset $\mathcal{E}\subset \mathcal{G}_{glc}(d,\Phi)$ forms a bounded family if 
there is $r\in \N$ such that for each $(X,B+M)\in \mathcal{E}$ there is a very ample 
divisor $A$ on $X$ with 
$$
A^d\le r \ \mbox{and} \ (K_X+B+M)\cdot A^{d-1}\le r.
$$ 

\subsection{Generalised lc thresholds}

Let $(X,B+M)$ be a generalised pair with data $X' \overset{\phi}\to X$ and $M'$.
Assume that $D$ on $X$ is an effective $\R$-divisor and that $N'$ on $X'$ is an $\R$-divisor which is
nef$/X$ and that $D+N$ is $\R$-Cartier where $N=\phi_*N'$.
The \emph{generalised lc threshold} of $D+N$ with respect to $(X,B+M)$ 
is defined as
$$
\sup \{s\in\R \mid \mbox{$(X,B+sD+M+sN)$ is generalised lc}\}
$$
where the pair in the definition has boundary part $B+sD$ and nef part $M'+sN'$.

The next lemma says that the generalised lc threshold is bounded from below away from zero under some 
boundedness assumptions.

\begin{lem}\label{l-bnd-glct-e-lc}
Let $d,r\in \N$ and $\epsilon\in \R^{>0}$ . Then there exists $t\in \R^{>0}$ depending only on $d,r,\epsilon$ 
satisfying the following. Assume that 
\begin{itemize}
\item $(X,B+M)$ is a projective generalised $\epsilon$-lc pair of dimension $d$ with 
data $X'\overset{\phi}\to X$ and $M'$, 

\item $A$ is a very ample divisor on $X$ with $A^d\le r$, 

\item $N'$ is a nef $\R$-divisor on $X'$ and $D$ is an effective $\R$-divisor on $X$,

\item $D+N$ is $\R$-Cartier where $N=\phi_* N'$, and

\item $A-(B+M+N+D)$ is pseudo-effective.
\end{itemize}
Then 
$$
(X,B+tD+M+tN)
$$
 is generalised klt with nef part $M'+tN'$.
\end{lem}
\begin{proof}
Since $N'$ is nef, by the negativity lemma, $\phi^*(D+N)=D'+N'$ where $D'\ge 0$. 
Pick an ample $\Q$-divisor $G$ on $X$ and let $G'=\phi^*G$. 
Write $G'\sim_\Q H'+P'$ where $H'$ is ample and $P'\ge 0$. 
Pick a general $\R$-divisor 
$$
0\le T' \sim_\R N'+H'.
$$ 
 Since 
$$
D'+T'+P'\sim_\R D'+N'+H'+P'\sim_\R 0/X,
$$
we have 
$$
D'\le D'+T'+P'=\phi^*\phi_*(D'+T'+P').
$$
Moreover, picking $G$ small enough and replacing $A$ with $2A$ we can assume 
$$
A-B-M-\phi_*(D'+T'+P')\sim_\R A-(B+M+D+N+G)
$$
is pseudo-effective. 
Thus replacing $N'$ with $0$ and replacing $D$ with $\phi_*(D'+T'+P')$ we can assume $N'=0$.

Applying a similar argument reduces the problem to the case $M'=0$. Indeed, 
write
$$
K_{X'}+B'+M'=\phi^*(K_X+B+M).
$$
Pick a general $\R$-divisor 
$$
0\le S' \sim_\R M'+H'
$$ 
and let $S=\phi_* S'$ and $P=\phi_*P'$ where $H',P'$ are as above. 
Then we can assume 
$$
(X,B+S+P)
$$ 
is $\frac{\epsilon}{2}$-lc because 
$$
K_{X'}+B'+S'+P'\sim_\R K_{X'}+B'+M'+H'+P'\sim_\R 0/X
$$
which means
$$
K_{X'}+B'+S'+P'= \phi^*(K_X+B+S+P)
$$
so $(X,B+S+P)$ is $\frac{\epsilon}{2}$-lc assuming $S',P'$ are sufficiently small. 
If 
$$
(X,B+S+P+tD)
$$ 
is klt for some $t>0$, then $(X,B+tD+M)$ is generalised klt. 
Therefore, replacing $\epsilon$ with $\frac{\epsilon}{2}$ and replacing $B$ with $B+S+P$, we can assume $M'=0$.

Now by [\ref{B-BAB}, Theorem 1.8], there is $t>0$ depending only on $d,r,\epsilon$ such that 
$$
(X,B+tD)
$$
is klt. 
\end{proof}

\subsection{MMP for generalised pairs}\label{ss-MMP-gpairs}

Given a generalised lc pair $(X,B+M)$ with data $X'\to X\to Z$ and $M'$, we can run the MMP on 
$K_X+B+M$ over $Z$ with scaling of an ample divisor $H$ if $(X,C)$ is klt for some boundary $C$. 
Such $C$ exists when $(X,B+M)$ is $\Q$-factorial generalised dlt 
or when it is generalised klt. The MMP terminates in the generalised klt case when 
$K_X+B+M$ or $B+M$ is big over $Z$ [\ref{BZh}, Lemma 4.4]. When $(X,B+M)$ is $\Q$-factorial generalised dlt, the MMP terminates if $K_X+B+M$ is not pseudo-effective$/Z$ [\ref{BZh}, Lemma 4.4]. When $(X,B+M)$ is $\Q$-factorial generalised dlt and $K_X+B+M$ is  pseudo-effective$/Z$, we do not know in general if the MMP terminates but in the course of the MMP we reach a model $Y$ on which, $K_Y+B_Y+M_Y$, the pushdown of $K_X+B+M$, is a limit of movable$/Z$ $\R$-divisors: indeed, after finitely many steps the MMP produces an infinite sequence $X_i\bir X_{i+1}$ 
of flips and a sequence $\alpha_i$ of numbers in $(0,1]$ such that 
$K_{X_i}+B_i+M_i+\alpha_iH_i$ is nef$/Z$; if $\alpha:=\lim\alpha_i>0$, then we can find a boundary $\Delta\sim_\R B+M+\alpha H$ so that $(X,\Delta)$ is klt, so the MMP terminates by [\ref{BCHM}]; thus we can assume $\alpha=0$; in particular, if $Y:=X_1$, then $K_Y+B_Y+M_Y$ is the limit of the movable$/Z$ $\R$-divisors 
$K_Y+B_Y+M_Y+\alpha_i H_Y$.

\begin{lem}\label{l-mmp-gen-scaling}
Let $(X,B+M)$ be a generalised klt pair with data $X'\to X\to Z$ and $M'$. 
Assume that $K_X+B$ is big over $Z$, $K_X+B+M$ is nef over $Z$, and $M$ is $\R$-Cartier.  
Then we can run the MMP on $K_X+B$ over $Z$ with scaling of $M$ and it terminates. 
\end{lem}
\begin{proof}
Since $K_X+B$ is big over $Z$, we can write $K_X+B\sim_\R D+A$ where $D\ge 0$ and $A\ge 0$ is ample 
over $Z$. Assume $t>0$ is sufficiently small. Then 
$$
(1+t)(K_X+B+M)\sim_\R K_X+B+tD+tA+(1+t)M
$$
is nef over $Z$ and 
$$
(X,B+tD+tA+(1+t)M)
$$
is generalised klt with nef part $tA'+(1+t)M'$ where $A'$ on $X'$ is the pullback of $A$. 

 It is enough to run an MMP on $K_X+B$ 
over $Z$ with scaling of $tD+tA+(1+t)M$ as it induces an MMP on $K_X+B$ with scaling of $M$. 
Perturbing the nef and big$/Z$ divisor $tA'+(1+t)M'$ we can find 
$$
0\le L\sim_\R tD+tA+(1+t)M
$$ 
so that $(X,B+L)$ is klt. So it is enough to run an MMP on $K_X+B$ over $Z$ with scaling of $L$, which terminates by [\ref{BCHM}].
\end{proof}

\subsection{Toroidal pairs}\label{ss-toroidal-pairs}

We recall a few definitions and facts about toroidal pairs. 

(1)
A pair $(Y,\Delta)$ is \emph{toroidal} if for each closed point $y\in Y$ there is a toric pair $(Y',\Delta')$ 
and a closed point $y'\in Y'$ such that $(Y,\Delta)$ and $(Y',\Delta')$ are formally isomorphic near $y,y'$. In particular, $\Delta$ is reduced, 
$(Y,\Delta)$ is lc, $Y$ is smooth outside $\Delta$, and 
$K_Y+\Delta$ is Cartier (cf. [\ref{B-toroidalisation}, Lemma 3.11]). For the general theory of toroidal pairs, see [\ref{KKMS}], though the terminology is slightly different.

Let $(Y,\Delta)$ be a toroidal pair. A prime divisor $D$ over $Y$ is said to be \emph{toroidal with respect to} $(Y,\Delta)$ if $D$ is an lc place of $(Y,\Delta)$, that is, $a(D,Y,\Delta)=0$.

(2)
A toroidal pair $(Y,\Delta=\sum_1^n S_i)$ is \emph{strict} if the $S_i$ are normal. 
In this case, $\bigcap_{i\in I} S_i$ is normal and $\bigcap_{i\in I} S_i\setminus \bigcup_{i\notin I} S_i $ 
is smooth for each subset $I\subseteq \{1,\dots,n\}$ [\ref{KKMS}, page 57]. 
The irreducible components of such $\bigcap_{i\in I} S_i$ are the strata of $(Y,\Delta)$, hence are normal.

Let $(Y,\Delta)$ be a strictly toroidal pair. Then for each closed point $y\in Y$, perhaps after shrinking $Y$ near $y$, there is a toric pair $(Y',\Delta')$ and a closed point $y'\in Y'$ together with an \'etale morphism $\pi\colon Y\to Y'$ mapping $y$ to $y'$ so that $\pi^{-1}\Delta'=\Delta$ [\ref{KKMS}, page 195]. 

A birational contraction $X\to Y$ from a normal variety is \emph{toroidal with respect to} $(Y,\Delta)$ (or sometimes we say it is a \emph{toroidal blowup}) if 
for each closed point $y\in Y$, perhaps after shrinking $Y$ near $y$, we can choose $\pi$ above so that $X=Y\times_{Y'}X'$ for some toric birational contraction $X'\to Y'$. In particular, all the exceptional divisors of $X\to Y$ are toroidal with respect to $(Y,\Delta)$. Moreover, if $K_X+\Delta_X$ is the pullback of $K_Y+\Delta$, then $(X,\Delta_X)$ is a strictly toroidal pair. 

If $(Y,\Delta)$ is a log smooth toroidal pair, then it is easy to see that the blowup of any stratum is automatically a toroidal blowup.

(3) 
Let $(Y,\Delta)$ be a strictly toroidal pair. 
The non-klt centres of this pair are exactly its strata. Assume 
that $0\le C\le \Delta$, $K_Y+C$ is $\R$-Cartier, and $V$ is a non-klt centre of $(Y,C)$. Then there is an 
adjunction formula $K_V+C_V=(K_Y+C)|_V$: indeed, $V$ is also a non-klt centre of $(Y,\Delta)$, hence 
it is an irreducible component of $\bigcap_{i\in I} S_i$ 
for some $I$; since $\bigcap_{i\in I} S_i$ is normal, shrinking $Y$ near $V$ we can assume $V=\bigcap_{i\in I} S_i$; 
also $\sum_{i\in I}S_i\le \rddown{C}$; picking $i\in I$, we have divisorial adjunctions
$K_{S_i}+C_{S_i}=(K_Y+C)|_{S_i}$ and $K_{S_i}+\Delta_{S_i}=(K_Y+\Delta)|_{S_i}$. Now $({S_i},\Delta_{S_i})$ 
is strictly toroidal, so repeating the process (or applying induction) proves the claim.
 
In the previous paragraph, if the coefficients of $C$ belong to a 
DCC set $\Phi$, then the coefficients of $C_V$ belong to a DCC set $\Psi$ depending only on $\Phi$ and $\dim Y$
[\ref{Shokurov-log-flips}, Corollary 3.10][\ref{BZh}, Proposition 4.9].

(4) 
Let $(Y,\Delta)$ be a $\Q$-factorial strictly toroidal pair and $D$ an exceptional prime divisor over $Y$ that is toroidal with respect to $(Y,\Delta)$. We will show that there is an extremal birational contraction $X\to Y$ extracting $D$ and no other divisors. Moreover, letting $K_X+\Delta_X$ be the pullback of $K_Y+\Delta$, $(X,\Delta_X)$ is a $\Q$-factorial strictly toroidal pair.

Proof: Since $D$ is toroidal with respect to $(Y,\Delta)$, we have $a(D,Y,\Delta)=0$. Moreover, since $(Y,\Delta)$ is $\Q$-factorial strictly toroidal, $Y$ has klt singularities. Thus there is a birational contraction $X\to Y$ extracting $D$ but no other divisors with $X$ being $\Q$-factorial [\ref{BCHM}, Corollary 1.4.3]. Moreover, we can assume that $X\to Y$ is extremal and that $-D$ is ample over $Y$. Such a contraction $X\to Y$ is uniquely determined by $(Y,\Delta)$ and $D$. 

It remains to show that $(X,\Delta_X)$ is a strictly toroidal pair. This is a local problem over $Y$, so we can fix a closed point $y\in Y$ and shrink $Y$ around it if necessary. In particular, we can assume that we have a toric model $(Y',\Delta')$ and an \'etale morphism $\pi\colon Y\to Y'$ such that $K_Y+\Delta=\pi^{*}(K_{Y'}+\Delta')$. The prime divisor $D$ uniquely determines a corresponding prime divisor $D'$ over $Y'$ that is toroidal with respect to $(Y',\Delta')$. There is a toric birational contraction $X'\to Y'$ extracting $D'$ but no other divisors and so that $-D'$ is ample over $Y'$ (here we do not need to assume $X',Y'$ to be $\Q$-factorial). 

Let $K_{X'}+\Delta_{X'}$ be the pullback of $K_{Y'}+\Delta'$. Then $(X',\Delta_{X'})$ is also toric, in particular, $D'\subset X'$ is normal. Moreover, base change via $Y\to Y'$ gives a birational contraction $\tilde{X}\to Y$ and an \'etale morphism $\tilde{X}\to X'$ so that the pullback of $D'$ to $\tilde{X}$, say $\tilde{D}$, is normal (hence has disjoint irreducible components) and it is anti-ample over $Y$. Only one irreducible component of $\tilde{D}$ intersects the fibre of $\tilde{X}\to Y$ over $y$, so 
shrinking $Y$ around $y$ we can assume that $\tilde{D}$ is irreducible. Moreover, $\tilde{D}$ is the birational transform of $D$.    
But then $X$ and $\tilde{X}$ are isomorphic in codimension one and since both $D$ and $\tilde{D}$ are anti-ample over $Y$, we see that $X$ and $\tilde{X}$ are actually isomorphic. Moreover, $K_X+\Delta_X$ is the pullback of $K_{X'}+\Delta_{X'}$. Therefore, $(X,\Delta_X)$ is strictly toroidal.  

\subsection{Volume of divisors}

Recall that the volume of an $\R$-divisor $L$ on a normal projective variety $X$ of dimension $d$ is defined 
as 
$$
\vol(L)=\limsup_{m\to \infty} \frac{h^0({mL})}{m^d/d!}. 
$$

\begin{lem}\label{l-vol-ample-div-increase}
Let $\phi\colon W\to X$ be a birational morphism between normal projective varieties. Assume $H$ is an 
ample $\R$-divisor on $X$ and that $G$ is an $\R$-divisor on $W$ that is either 
\begin{itemize}
\item effective and not exceptional over $X$, or
 
\item nef but not numerically trivial.   
\end{itemize}
Then $\vol(\phi^*H+G)>\vol(H)$.
\end{lem}
\begin{proof}
Assume $\vol(\phi^*H+G)=\vol(H)$. First suppose $G$ is nef but not numerically trivial. 
Then 
$$
\vol(\phi^*H+G)=(\phi^*H+G)^{d}\ge \vol(H)+(\phi^*H)^{d-1}\cdot G>\vol(H),
$$
where $d=\dim X$, a contradiction. 

Now assume $G$ is effective but not exceptional. Replacing $W$ with a resolution and replacing $G$ with its birational transform, we can assume that $W$ is smooth. Let $D$ be a component of $G$ not contracted over $X$. 
We will argue as in the proof of [\ref{HMX3}, Lemma 2.2.2].
If $g$ is the coefficient of $D$ in $G$, then for $t\in [0,g]$ we have 
$$
\vol(H)\le v(t):=\vol(\phi^*H+tD)\le \vol(\phi^*H+G),
$$
so $v(t)$ is a constant function on $[0,g]$. But then by [\ref{Laz-Mustata}, 4.25(iii)] we have 
$$
0=v'(0)=d\vol_{W|D}(\phi^*H)\ge (\phi^*H)^{d-1}\cdot D\neq 0
$$
where $v'(t)$ denotes the differentiation of $v(t)$, and $\vol_{W|D}$ denotes restricted volume. 
This is not possible, a contradiction.
\end{proof}

\subsection{Iitaka volumes}\label{ss-Iitaka-volumes}
Let $L$ be a $\Q$-divisor on a normal projective variety $X$ with Kodaira dimension $\kappa(L)$. 
Following [\ref{Z-Li}], define the \emph{Iitaka volume} of $L$ to be 
$$
\Ivol(L):=\limsup_{m \mapsto \infty}  \frac{ h^0(mL)}{m^{\kappa(L)}/\kappa(L)!}
$$
when $\kappa(L)\ge 0$, and define it to be $0$ when $\kappa(L)= -\infty$. 
When $\kappa(L)= 0$, by convention we let 
$\kappa(L)!=1$, so in this case $\Ivol(L)=1$. 

If $f\colon X\to Z$ is a contraction and $L\sim_\Q f^*A$ for some big $\Q$-divisor $A$, then 
$\Ivol(L)=\vol(A)$.

\subsection{Rational maps}
The following lemma is similar to [\ref{HMX3}, Lemma 2.2.2].

\begin{lem}\label{l-rat-map-to-lc-model}
Assume that 
\begin{itemize}
\item $(X,B+M)$ is a projective generalised lc pair with data $X'\overset{\phi}\to X$ and $M'$, 

\item $K_X+B+M$ is ample, 

\item $(Y,B_Y+M_Y)$ is a projective generalised lc pair with data $X'\overset{\psi}\to Y$ and $M'$, 

\item $\rho\colon Y\bir X$ is the birational map induced by $\phi,\psi$, 

\item that 
$$
\vol(K_Y+B_Y+M_Y)=\vol(K_X+B+M),
$$

\item and 
$$
a(D,X,B+M)\ge a(D,Y,B_Y+M_Y)
$$ 
for any prime divisor $D$ over $X$.
\end{itemize}
Then $\rho^{-1}$ does not contract any divisor and $(X,B+M)$ is the generalised lc model of $(Y,B_Y+M_Y)$.
\end{lem}
\begin{proof}
Generalised lc models are analogues of lc models for usual pairs. Here 
$(X,B+M)$ being the generalised lc model of $(Y,B_Y+M_Y)$ means that $X$ is the ample model 
of $K_Y+B_Y+M_Y$ and that $K_X+B+M$ is the pushdown of $K_Y+B_Y+M_Y$.

By the last assumption, 
$$
P':=\psi^*(K_Y+B_Y+M_Y)-\phi^*(K_X+B+M)
$$
is effective. Applying Lemma \ref{l-vol-ample-div-increase}, we see that 
$P'$ is exceptional over $X$, hence $(X,B+M)$ is indeed the generalised lc model of $(Y,B_Y+M_Y)$, 
so $\rho^{-1}$ does not contract any divisor.
\end{proof}


\subsection{Adjunction for fibre spaces}

Let $(X,B)$ be a projective sub-pair and let $f\colon X\to Z$ be a contraction with $\dim Z>0$ 
such that $(X,B)$ is sub-lc near the 
generic fibre of $f$, and $K_X+B\sim_\R 0/Z$. 
Below we recall a construction based on [\ref{kaw-subadjuntion}] which we refer 
to as \emph{adjunction for fibre spaces} or \emph{canonical bundle formula}. The klt case is done in [\ref{ambro-adj}] and the 
lc case is done in [\ref{FG-lc-trivial}] by reducing it to the klt case. See also [\ref{FM}] for another 
variant.

For each prime divisor $D$ on $Z$ we let 
$t_D$ be the lc threshold of $f^*D$ with respect to $(X,B)$ over the generic point of $D$, that is, 
$t_D$ is the largest number so that $(X,B+t_Df^*D)$ is sub-lc over the generic point of $D$. 
Next let $b_D=1-t_D$, and then define $B_Z=\sum b_DD$ where the sum runs over all the 
prime divisors on $Z$. 

By assumption, $K_X+B\sim_\R f^*L_Z$ for some $\R$-Cartier 
$\R$-divisor $L_Z$ on $Z$. Letting $M_Z=L_Z-(K_Z+B_Z)$ we get 
$$
K_X+B\sim_\R f^*(K_Z+B_Z+M_Z).
$$
We call $B_Z$ the \emph{discriminant divisor} and ${M_Z}$ the \emph{moduli divisor} of adjunction. 
Obviously $B_Z$ is uniquely determined but $M_Z$  is determined only up to $\R$-linear equivalence 
because it depends on the choice of $L_Z$.

Now let $\phi\colon X'\to X$ and $\psi\colon Z'\to Z$ be birational morphisms from normal 
projective varieties and assume the induced map $f'\colon X'\bir Z'$ is a morphism. 
Let $K_{X'}+B'=\phi^*(K_X+B)$. Similar to above we can define a discriminant divisor $B_{Z'}$, and 
taking $L_{Z'}=\psi^*L_Z$ gives a moduli divisor $M_{Z'}$ so that 
$$
K_{X'}+B'\sim_\R f'^*(K_{Z'}+B_{Z'}+M_{Z'})
$$ 
and $B_Z=\psi_*B_{Z'}$ and $M_Z=\psi_*M_{Z'}$.

\begin{thm}[{[\ref{ambro-adj}][\ref{FG-lc-trivial}]}]
With the above notation and assumptions, suppose that $(X,B)$ is lc over the generic point 
of $Z$ and that $B$ is a $\Q$-divisor. If $Z'\to Z$ is a high resolution, then 
$M_{Z'}$ is nef and for any birational morphism $Z''\to Z'$ from a normal projective 
variety, $M_{Z''}$ is the pullback of $M_{Z'}$.
\end{thm}

Under the assumptions of the theorem, $(Z,B_Z+M_Z)$ is a generalised sub-pair with nef part $M_{Z'}$. In the sequel when we discuss this kind of generalised sub-pair, we will refer to $M_{Z'}$ both as the moduli divisor of adjunction and as the nef part of $(Z,B_Z+M_Z)$. Note that when $(X,B)$ is a pair, then $(Z,B_Z+M_Z)$ is a generalised pair.


\subsection{Bounded families}\label{ss-bnd-couples}
In this paper a \emph{couple} $(X,D)$ consists of a normal projective variety $X$ and a  divisor 
$D$ on $X$ whose non-zero coefficients are all equal to $1$, i.e. $D$ is a reduced divisor. 
The reason we call $(X,D)$ a couple rather than a pair is that we are concerned with 
$D$ rather than $K_X+D$ and we do not want to assume $K_X+D$ to be $\Q$-Cartier 
or with nice singularities. Two couples $(X,D)$ and $(X',D')$ are isomorphic if 
there is an isomorphism $X\to X'$ mapping $D$ onto $D'$.

We say that a set $\mathcal{P}$ of couples of dimension $\le d$ is \emph{bounded} 
if there is $r\in \N$ such that for each $(X,D)\in \mathcal{P}$ we can find a very 
ample divisor $A$ on $X$ so that $A^{\dim X}\le r$ and $D\cdot A^{\dim X-1}\le r$. This is  
equivalent to saying that there exist 
finitely many projective morphisms $V^i\to T^i$ of varieties and reduced divisors $C^i$ on $V^i$ 
such that for each $(X,D)\in \mathcal{P}$ there exist an $i$, a closed point $t\in T^i$, and an 
isomorphism $\phi\colon V^i_t\to X$ such that $(V^i_t,C^i_t)$ is a couple and 
$\phi_*C_t^i\ge D$.

A set $\mathcal{Q}$ of projective lc pairs $(X,B)$ is said to be bounded if the $(X,\Supp B)$ 
form a bounded family of couples.

\subsection{Big log divisors}

\begin{lem}\label{lem-big-log-div}
Let $d,p\in \N$, $\Phi\subset \R^{\ge 0}$ be a DCC set, and $v\in \R^{>0}$. 
Then there exists $m\in \N$ depending only on $d,p,\Phi,v$ satisfying the following. 
Assume that $(X,B+M)$ is a projective generalised pair with data $X'\to X$ and $M'$, and $A\ge 0$ is a 
divisor on $X$ so that 
\begin{itemize}
\item $(X,B+M)$ is generalised lc,
\item the coefficients of $B$ belong to $\Phi$ and $pM'$ is Cartier, 
\item $|A|$ is base point free, and 
\item $K_X+B+M$ is big and $\vol(K_X+B+M+A)<v$.
\end{itemize}
Then we can find 
$$
0\le L\sim m(K_X+B+M)
$$
such that $L\ge A$.
\end{lem}
\begin{proof}
First note that if we can find $0\le L'\sim m(K_X+B+M)$ so that $L'\ge A'$ for some $A'\sim A$, then we can also find $0\le L\sim m(K_X+B+M)$ with $L\ge A$. We will use this later in the proof. 

Applying [\ref{BZh}, Theorem 1.3], there exists $m\in \N$ depending only on $d,p,\Phi$ such that $|m(K_X+B+M)|$ defines a birational map. So there is a resolution $\phi\colon W\to X$ so that 
$$
\phi^*m(K_X+B+M)\sim P+F
$$
where $P$ is the free part and $F$ is the fixed part of the linear system $|\phi^*m(K_X+B+M)|$. 
In particular, $P$ defines a morphism $W\to \PP^n$ to some projective space so that the map is birational onto its image. Since $\vol(K_X+B+M)<v$, the image belongs to a bounded family, hence its normalisation, say $V$, also belongs to a bounded family (cf. proof of [\ref{HMX1}, Lemma 2.4.2]). Moreover, $P$ is the pullback of an ample divisor $H$ on $V$ and there exists a bounded number $l$ so that $lH$ is very ample. 

Denote the morphism $W\to V$ by $\psi$. Then since $A$ is nef and $H$ is ample,
$$
\phi^*A\cdot (\psi^*H)^{\dim X-1}\le \vol(\phi^*A+\psi^*H)\le \vol(m(K_X+B+M+A))<m^dv.
$$
Thus $(\psi_*\phi^*A)\cdot H^{\dim X-1}<m^dv$. We can change $A$ linearly so that $\phi^*A$ is reduced and it has no exceptional component over $V$. Then $(V,\psi_*\phi^*A)$ belongs to a bounded family and there is a bounded number $a\in \N$ such that  
we can find $G\in |aH|$ with $G\ge \psi_*\phi^*A$. This implies that $\psi^*G\ge \phi^*A$ as $\phi^*A$  has no exceptional component over $V$. Then in turn we have $\phi_*\psi^*G\ge A$, hence we can find 
$$
A\le  L\sim am(K_X+B+M).
$$
Now replace $m$ with $am$.
\end{proof}

\subsection{Nef divisors}

\begin{lem}\label{lem-nef-div}
Let $X$ be a smooth projective variety and $M$ be a nef Cartier divisor and $A$ an ample Cartier divisor on $X$. Then there exists $n\in \N$ depending only on $d:=\dim X$ such that 
$$
|n(K_{X}+3d{A}+l{M})|
$$ 
is base point free for every $l\in \N$.
\end{lem}
\begin{proof}
This follows immediately from the effective base point free theorem [\ref{kollar-ebpf}] by noting that $K_{X}+3d{A}+l{M}$ is nef and Cartier, and that $3dA+lM$ is ample. 
\end{proof}

\subsection{Blowup model}

\begin{lem}\label{lem-blowup-model}
Let $(X,B+M)$ be a projective generalised lc pair with data $X'\to X$ and $M'$, $(\overline{X}, \overline{\Sigma})$ be a projective log smooth pair with $\overline{\Sigma}$ reduced, and $X\bir \overline{X}$ be a birational map. Assume that $\overline{\Sigma}$ contains the support of the birational transform of $B$ and the exceptional divisors of $\overline{X}\bir X$ and that $M'$ descends to $\overline{X}$ say as $\overline{M}$. Then there is a sequence $Y\to \overline{X}$ of smooth blowups toroidal with respect to $(\overline{X}, \overline{\Sigma})$ such that if $B_Y$ is the sum of the birational transform of $B$ and the exceptional divisors of $Y\bir X$ and $M_Y$ is the pullback of $\overline{M}$, then 
$$
\vol(K_X+B+M)=\vol(K_Y+B_Y+M_Y).
$$   
\end{lem}
\begin{proof}
Consider the prime exceptional$/\overline{X}$ divisors $D$ on $X$ which are toroidal with respect to $(\overline{X},\overline{\Sigma})$, i.e. $a(D,\overline{X},\overline{\Sigma})=0$. There is a birational contraction $Y\to \overline{X}$ which is 
a sequence of smooth blowups toroidal with respect to $(\overline{X},\overline{\Sigma})$ such that the birational transform of each of those divisors $D$ appears on $Y$: indeed, we just need to blowup the centres of such $D$ repeatedly, starting with $\overline{X}$, until they all appear [\ref{kollar-mori}, Lemma 2.45] (also see [\ref{B-BAB}, 2.16]). 

We can assume $\rho\colon X'\bir Y$ is a morphism.
Let $B'$ be the sum of the birational transform of $B$ and the reduced exceptional divisor of $X'\to X$ and $B_{Y}$ be its pushdown to $Y$. Note that $B_{Y}\le \Sigma_{Y}$ where $K_{Y}+\Sigma_{Y}$ is the pullback of $K_{\overline{X}}+\overline{\Sigma}$, and $B_Y$ is the sum of the birational transform of $B$ and the exceptional divisors of $Y\bir X$.
 
Now clearly
$$
\vol(K_{X}+B+M)=\vol(K_{X'}+B'+M')\le \vol(K_{Y}+B_{Y}+M_{Y})
$$
where $M_{Y}$ is the pushdown of $M'$ which coincides with the pullback of $\overline{M}$.

We claim that the inequality of the volumes is actually an equality. 
It is enough to show that $K_{X'}+B'\ge \rho^*(K_{Y}+B_{Y})$ which follows if we show that 
for any prime divisor $C$ on $X'$ we have $\mu_CB'\ge 1-a(C,Y,B_{Y})$. 
This is obvious if $C$ is not exceptional over $Y$ or if it is exceptional over $X$. So we assume $C$ is exceptional over $Y$ but not exceptional over $X$. In particular, $C$ is exceptional over $\overline{X}$ but $C$ is not toroidal with respect to $(\overline{X},\overline{\Sigma})$ by our choice of $Y$. Thus $C$ is not toroidal with respect to $({Y},\Sigma_{Y})$, hence 
$$
a(C,Y,B_{Y})\ge a(C,Y,\Sigma_{Y})\ge 1
$$ 
which implies $\mu_CB'\ge 1-a(C,Y,B_{Y})$ as required.
\end{proof}


\section{\bf Descent of nef divisors}

\subsection{Descent of divisors}\label{ss-descend-divs}
(1) Let $X,Y,Z$ be normal varieties, $X\to Z$ and $Y\to Z$ be contractions, and $X\bir Y/Z$ be a birational map. Let $M$ be an $\R$-Cartier $\R$-divisor on $X$. 
We say that $M$ descends to $Y$ if there exist an $\R$-Cartier $\R$-divisor $L$ on $Y$ and  
a commutative diagram of contractions of normal varieties  
$$
\xymatrix{
& W\ar[ld]_\phi\ar[rd]^\psi &\\
X \ar[rd] & & Y\ar[ld]\\
&Z&
}
$$
compatible with the given $X\bir Y$ such that $\phi^*M=\psi^*L$. We also refer to this situation by saying that 
 \emph{$M$ descends to $Y$ as $L$}.

(2) In some special situations nef divisors automatically descend. Assume $X\to Y/Z$ is a birational contraction of normal varieties 
and that $X$ is of Fano type over $Y$, i.e. there is a big$/Y$ boundary $C$ such that $(X,C)$ is klt and $K_X+C\equiv 0/Y$. Assume that $(X,B)$ is lc for some $B$ and that $M$ is a Cartier divisor 
which is nef over $Y$. We claim that if $K_X+B+uM$ is anti-nef over $Y$ for some $u>2d$, 
then $M$ descends to $Y$ as a Cartier divisor, where $d=\dim X$. 
First note that we can replace $Z$ with $Y$, hence assume $Y=Z$.
Second note that $-(K_X+B+uM)$ is semi-ample over $Y$ as $X$ is of Fano type over $Y$, so adding a 
general divisor 
$$
0\le L\sim_\R -(K_X+B+uM)/Y
$$ 
to $B$ it is enough to treat only the case $K_X+B+uM\equiv 0/Y$. Third replacing 
$B$ with $tB+(1-t)C$ and $uM$ with $tuM$ for some $t<1$ close to $1$, we can assume that $(X,B)$ is klt.  

Now if $X\to Y$ is an isomorphism the claim is obvious. If not, then since $X$ is of Fano type over $Y$, 
there is an extremal contraction $X\to U/Y$. Assume $M$ is ample over $U$. Then $K_X+B$ is anti-ample 
over $U$. By boundedness of the length of extremal rays [\ref{kawamata-bnd-ext-ray}], 
there is a curve $\Gamma$ generating the extremal ray corresponding to $X\to U$ such that 
$$
2d<u\le uM\cdot \Gamma=-(K_X+B)\cdot \Gamma\le 2d,
$$ 
a contradiction. Thus $M\equiv 0/U$, hence $M$ descends to $U$ as a Cartier divisor by the cone theorem 
[\ref{kollar-mori}, Theorem 3.7], so we can replace $X$ with $U$ and repeat the process until we reach $Y$. 

(3) The following lemma is a useful criterion for descent of nef divisors.

\begin{lem}\label{l-t>3dp=descend}
Let $d,p\in \N$. Assume $(X,B+M)$ is a generalised pair of dimension $d$ 
with data $X'\overset{\phi}\to X$ and $M'$ where $pM'$ is Cartier. 
Assume that $(X,B)$ is klt and that $\lambda$ is the generalised lc threshold of $M$ with respect to $(X,B)$. 
Then the following are equivalent: 
\begin{itemize}
\item $\lambda\ge 3dp$, 

\item  $\lambda=\infty$,

\item $M'$ descends to $X$ and $pM$ is Cartier. 
\end{itemize} 
\end{lem}
\begin{proof}
We can assume $\phi$ is a log resolution of $(X,B)$. 
Let $\Gamma'$ be the birational transform of $B$ plus the reduced exceptional divisor of $\phi$. 
Let $n=3dp$ and run an MMP on $K_{X'}+\Gamma'+n M'$ over $X$ with scaling of some ample divisor. 
Then by (2) above, $M'$ descends to every model appearing in the MMP. We reach a model $X''$ 
on which $K_{X''}+\Gamma''+nM''$ is a limit of movable $\R$-divisors, relatively over $X$. 

Assume that $\lambda \ge n$, i.e. $(X,B+n M)$ is generalised lc. Then we can write 
$$
K_{X'}+\Gamma'+nM'=\phi^*(K_X+B+nM)+E'
$$
where $E'\ge 0$ is exceptional over $X$. Thus by the general negativity lemma [\ref{B-lc-flips}, Lemma 3.3], 
the MMP terminates with a $\Q$-factorial dlt model $(X'',\Gamma''+nM'')$ of 
$(X,B+nM)$. In particular, $K_{X''}+\Gamma''+nM''\sim_\R 0/X$. Moreover, $X''$ is of Fano type 
over $X$ because $(X,B)$ is klt.
Then $M''$ descends to $X$ by (2) above, hence $M'$ descends to $X$ and that $pM$ is Cartier.

Now assume that $M'$ descends to $X$. Then obviously $\lambda=\infty$. 

Finally if $\lambda=\infty$, then clearly $\lambda\ge n$.
\end{proof}

\subsection{Generalised lc modifications}

Let $(X,B+M)$ be a generalised pair with data $X'\to X$ and $M'$. 
Let $C$ be the divisor obtained from $B$ by replacing every coefficient 
$>1$ with $1$. Assume that $(X'',B''+M'')$ is a generalised pair 
with data $X'\to X''$ and $M'$ equipped with a birational contraction $X''\to X$. We say that 
$(X'',B''+M'')$ is the \emph{generalised lc modification} of $(X,B+M)$ if the following hold:
\begin{itemize}
\item $B''=C^\sim+E$ where $C^\sim$ is the birational 
transform of $C$ and $E$ is the reduced exceptional divisor of $X''\to X$,

\item $(X'',B''+M'')$ is generalised lc, and 

\item $K_{X''}+B''+M''$ is ample over $X$. 
\end{itemize}

The generalised lc modification is unique if it exists. 
It is also obvious that if $(X,B+M)$ is generalised lc then it is the generalised lc modification of itself.

\subsection{Bounded coefficients on generalised lc modifications}

In this subsection we show that under certain assumptions the pullback of some divisors to a generalised 
lc modification have bounded coefficients. This is important for the subsequent results in this section.

\begin{prop}\label{p-gen-lc-model}
Let $d,p,r\in \N$. Then there exists $c\in \N$ depending only on $d,p,r$ satisfying the following. 
Assume that 
\begin{itemize}
\item $(X,B+M)$ is a projective generalised pair of dimension $d$ with data $X'\overset{\phi}\to X$, $M'$, 

\item $(X,B)$ is $\Q$-factorial klt,

\item $pB$ is integral and $pM'$ is Cartier,

\item $A$ is a very ample divisor on $X$ with $A^d\le r$, and 

\item $A-(B+M)$ is pseudo-effective.
\end{itemize}
Then there is a generalised lc modification $(X'',B''+3dpM'')$ of $(X,B+3dpM)$ such that 
\begin{itemize}
\item denoting $X''\to X$ by $\pi$ we have 
$$
\pi^*M=M''+\sum e_iE_i''
$$ 
where $E_i''$ are the exceptional divisors of $\pi$ and $\sum e_i\le c$,

\item $M'$ descends to $X''$ as $M''$, 

\item $pM''$ is Cartier, and 

\item $({X''},B''-t(\sum e_iE_i'')+(n-t)M'')$ is generalised klt for any small $t>0$.
\end{itemize} 
\end{prop}

\begin{proof}
Note that we are not assuming $(X,B+M)$ to be generalised lc.\\

\emph{Step 1.}
We can assume that $\phi$ is a log resolution of $(X,B)$. 
Let $\Delta'$ be the sum of the birational transform $B^\sim$ and the reduced exceptional divisor of $\phi$. 
Let $\Gamma'=\Delta'+\frac{1}{2}H'$ where $H'$ is a general element of $|6d\phi^*A|$. 
To ease notation put $n=3dp$. 

Run an MMP on $K_{X'}+\Gamma'+nM'$ with scaling of some ample divisor $L'$. Since 
$$
\Gamma'=\Delta'+\frac{1}{2}H'\sim_\Q \Delta'+3d\phi^*A,
$$ 
any extremal ray in the course of the MMP intersects the pullback of $A$ trivially, 
by boundedness of the length of extremal rays [\ref{kawamata-bnd-ext-ray}]. Thus the MMP 
is automatically an MMP over $X$. In addition, since $n=3dp$ and $pM'$ is Cartier, 
the MMP is $M'$-trivial, that is, in each step $M'$ intersects the contracted extremal ray trivially. 
In particular, by \ref{ss-descend-divs}(2), $pM'$ descends as a Cartier divisor 
to every model appearing in the MMP.\\

\emph{Step 2.}
The divisor $K_{X'}+\Gamma'+nM'$ is big [\ref{B-compl}, Lemma 2.46], so in the course of the MMP 
we reach a model $X''$ on which $K_{X''}+\Gamma''+nM''$ is a limit of movable $\R$-divisors, by \ref{ss-MMP-gpairs}.
In particular, 
for any exceptional$/X$ prime divisor $S''$ on $X''$, $(K_{X''}+\Gamma''+nM'')|_{S''}$ is pseudo-effective. 
Thus by the general negativity lemma [\ref{B-lc-flips}, Lemma 3.3], we can write 
$$
K_{X''}+\Gamma''+nM''+G''=\pi^*(K_X+\Gamma+nM)
$$
where $\pi$ denotes $X''\to X$, $\Gamma=\pi_*\Gamma''$, and $G''$ is effective and exceptional. 

Write 
$$
\pi^*M=M''+\sum e_iE_i''
$$ 
where $E_i''$ are the exceptional divisors of $\pi$. Since $M''$ is a nef divisor, $e_i\ge 0$.\\

\emph{Step 3.} 
If $\pi$ is an isomorphism, then the proposition holds. So assume $\pi$ is not an isomorphism in which case it has exceptional divisors as $X$ is $\Q$-factorial. 
We claim that $e_i>0$ for every $i$. Assume not, say $e_1=0$. Write $\phi^*M=M'+\sum f_jF_j$ where $F_j$ 
runs through the exceptional divisors of $\phi$. Say $F_1$ is the birational transform of $E_1$, then $f_1=0$. 
Since $(X,B)$ is klt and $\Gamma=B+\frac{1}{2}H$ where $H=\phi_*H'$ is a general member of $|6dA|$, 
we see that $(X,\Gamma)$ is klt, so we have 
$$
0<a(E_1,X,\Gamma)=a(F_1,X,\Gamma).
$$ 
On the other hand, by Lemma \ref{l-comapring-log-disc} and the condition $f_1=0$, 
$$
a(F_1,X,\Gamma)=a(F_1,X,\Gamma+nM)+nf_1=a(F_1,X,\Gamma+nM).
$$
In addition, since $G''\ge 0$ and since $E_1$ is a component of $\rddown{\Gamma''}$,
$$
a(F_1,X,\Gamma+nM)=a(E_1,X,\Gamma+nM)\le a(E_1,X'',\Gamma''+nM'')=0.
$$
To summarise we have proved
$$
0<a(E_1,X,\Gamma)=a(E_1,X,\Gamma+nM)\le a(E_1,X'',\Gamma''+nM'')=0
$$
which is a contradiction.\\ 
   
\emph{Step 4.}   
In this step we show that the MMP of Step 1 terminates with a good minimal model and then we replace 
$X''$ with the corresponding ample model. 
Assume $t>0$ is a sufficiently small real number. Then $\Gamma''-t(\sum e_iE_i'')\ge 0$. 
Moreover, the MMP from $X''$ onwards is also an MMP on 
$$
K_{X''}+\Gamma''-t(\sum e_iE_i'')+(n-t)M''= K_{X''}+\Gamma''+nM''-t\pi^*M
$$   
over $X$ with scaling of $L''$ where $L''$ is the pushdown of $L'$ in Step 1. Since all $e_i>0$, 
$$
\rddown{\Gamma''}=\Supp \sum e_iE_i''
$$ 
and since $M'$ descends to $X''$ as $M''$, 
$$
({X''},\Gamma''-t(\sum e_iE_i'')+(n-t)M'')
$$ 
is generalised klt, hence the MMP terminates over $X$: indeed, 
we can find a boundary $\Theta''$ 
such that over $X$ we have 
$$
\Theta''\equiv \Gamma''-t(\sum e_iE_i'')+(n-t)M''\equiv \Gamma''+nM''
$$
and $(X'',\Theta'')$ is klt, so the termination follows from [\ref{BCHM}].
Therefore, the MMP of Step 1 terminates globally because as we observed the MMP is over $X$. 
Moreover, the pushdown of $K_{X''}+\Gamma''+nM''$ to the minimal model is semi-ample over $X$. 
Replacing $X''$ with the end product of the MMP, we can assume $K_{X''}+\Gamma''+nM''$ is globally nef 
and that it is semi-ample over $X$. 

Let $X'''$ be the ample model of $K_{X''}+\Gamma''+nM''$ over $X$ which exists in view of the previous paragraph. 
Since $K_{X''}+\Theta''\equiv 0/X'''$ and $(X'',\Theta'')$ is klt, $X''$ is of Fano type over $X'''$. 
Then by \ref{ss-descend-divs}(2), $pM''$ descends to $X'''$ as a nef Cartier divisor. 
Replacing $X''$ with $X'''$, we can assume $K_{X''}+\Gamma''+nM''$ is ample over $X$. 
By our choice of $H'$ and the definition of $\Gamma'$, in fact $K_{X''}+\Gamma''+nM''$ 
is globally ample. However, now $({X''},\Gamma''+nM'')$ may not be generalised dlt but it is 
generalised lc which is enough for our purposes. 

By construction, $(X'',B''+nM'')$ is the generalised lc modification of $(X,B+nM)$ where $B'':=\Delta''$ 
is the sum of $B^\sim$ and the reduced exceptional divisor of $X''\to X$ which we again denote by $\pi$. 
As before we will write 
$$
\pi^*M=M''+\sum e_iE_i''
$$
and we note that 
$$
({X''},B''-t(\sum e_iE_i'')+(n-t)M'')
$$ 
is generalised klt for any small $t>0$.\\

\emph{Step 5.}
In this step we show that the volume of the restriction of $K_{X''}+\Gamma''+nM''$ to each 
component of $\rddown{\Gamma''}$ is bounded from below away from zero. 
Pick a component $S''$ of $\rddown{\Gamma''}=\Supp \sum e_iE_i''$ and let $T$ be its 
normalisation. Then by [\ref{B-compl}, 3.1(1),(2)], we have a generalised adjunction formula 
$$
K_T+\Gamma_T''+nM_T''\sim_\R (K_{X''}+\Gamma''+nM'')|_T
$$
where the coefficients of $\Gamma_T''$ belong to a DCC set depending only on $p$, and $nM_{T}''$ is 
Cartier where $nM_{T}''$ is the nef part of $(T,\Gamma_T''+nM_T'')$. 

Now by [\ref{BZh}, Theorem 1.3], 
there is $\alpha>0$ depending only on $d,p$ such that  
$$
\alpha<\vol(K_T+\Gamma_T''+nM_T'')=\vol((K_{X''}+\Gamma''+nM'')|_{S''})
$$
$$
=(K_{X''}+\Gamma''+nM'')^{d-1}\cdot S''.
$$\  

\emph{Step 6.}
Since $A$ is very ample and $A^d\le r$, $X$ belongs to a bounded family of varieties, 
so there is $q\in \N$ depending only on $d,r$ such that $qA-K_X$ is ample. Moreover,  
since $A-(B+M)$ is pseudo-effective, both $A-B$ and $A-M$ are pseudo-effective, hence 
letting $l=q+1+3d+n$, we see that 
$$
lA-(K_X+\Gamma+nM)=lA-(K_X+B+\frac{1}{2}H+nM)
$$
$$
=qA-K_X+A-B+3dA-\frac{1}{2}H+nA-nM
$$
is pseudo-effective. Therefore, 
$$
\vol(K_X+\Gamma+nM+A)\le \vol((l+1)A)\le (l+1)^dr.
$$
In view of Step 2, this implies that 
$$
\vol(K_{X''}+\Gamma''+nM''+\pi^*A)\le (l+1)^dr.
$$

Since both $K_{X''}+\Gamma''+nM''$ and $\pi^*A$ are nef, 
$$
(K_{X''}+\Gamma''+nM'')^{d-1}\cdot \pi^*A \le \vol(K_{X''}+\Gamma''+nM''+\pi^*A)\le (l+1)^dr.
$$ 
Thus since $A-M$ is pseudo-effective, we get 
$$
(K_{X''}+\Gamma''+nM'')^{d-1}\cdot \pi^*M\le (K_{X''}+\Gamma''+nM'')^{d-1}\cdot \pi^*A\le (l+1)^dr, 
$$ 
hence by nefness of $M''$ and the choice of $\alpha$ we have 
$$
\alpha(\sum e_i)\le(K_{X''}+\Gamma''+nM'')^{d-1}\cdot (\sum e_iE_i'') 
$$
$$
\le (K_{X''}+\Gamma''+nM'')^{d-1}\cdot \pi^*M \le (l+1)^dr. 
$$ 
Therefore, $\sum e_i$ is bounded from above by some fixed number $c\in \N$ depending only on 
$d,l,r,\alpha$, hence depending only on $d,p,r$. 
\end{proof}

\subsection{Bounded crepant models}

Given a generalised pair $(X,B+M)$ with data $X'\overset{\phi}\to X$ and $M'$ 
and a birational contraction $\psi\colon Y\to X$ 
from a normal variety we can write 
$$
K_{Y}+B_Y+M_Y=\psi^*(K_X+B+M)
$$
in a natural way: we can assume $X'\bir Y$ is a morphism; we first write 
$$
K_{X'}+B'+M'=\phi^*(K_X+B+M)
$$
and then simply push down $B'+M'$ to $Y$ to get $B_Y+M_Y$. 
When $B_Y\ge 0$ we say that $(Y,B_Y+M_Y)$ is a \emph{crepant model} of $(X,B+M)$. Note that 
$(Y,B_Y+M_Y)$ is a generalised pair with data $X'\to Y$ and $M'$.

When $(X,B+M)$ is projective generalised lc and $X$ varies in a bounded family, in general, 
crepant models do not form a bounded family even for fixed $X$, e.g. when $(X,B)$ is any projective toric 
pair of dimension $\ge 2$ and $M'=0$. The crucial point of the 
next result is that we can sometimes ensure that certain crepant models are bounded.

In contrast when $(X,B+M)$ is projective generalised $\epsilon$-lc ($\epsilon>0$) and $X$ varies in a bounded family 
and the ``degree" of $B+M$ is bounded from above, 
the crepant models are bounded by [\ref{B-lcyf}, Theorem 2.3]. This is an important ingredient of the 
proof of the next result.

\begin{prop}\label{p-bnd-model-by-g-lct}
Let $d,p,r\in \N$. Then there exists $s\in \N$ depending only on $d,p,r$ satisfying the following. 
Assume that 
\begin{itemize}
\item $(X,B+M)$ is a projective generalised pair of dimension $d$ with data $X'\overset{\phi}\to X$, $M'$, 

\item $(X,B)$ is $\Q$-factorial klt,

\item $pB$ is integral and $pM'$ is Cartier,

\item $A$ is a very ample divisor on $X$ with $A^d\le r$,

\item $A-(B+M)$ is pseudo-effective, and 

\item $\lambda<\infty$ where $\lambda$ is the generalised lc threshold of $M$ with respect to $(X,B)$.
\end{itemize}
Then there exist a $\Q$-factorial crepant model $(Y,B_Y+\lambda M_Y)$ of $(X,B+\lambda M)$ and 
a very ample divisor $A_Y$ on $Y$ such that 
\begin{itemize}
\item the reduced exceptional divisor of $Y\to X$ is equal to $\rddown{B_Y}$,

\item the pair 
$$
(Y,B_Y-t\rddown{B_Y}+\lambda M_Y)
$$ 
is generalised klt for any $t\in (0,1]$, and 

\item $A_Y^d\le s$ and $A_Y-(B_Y+M_Y)$ is pseudo-effective.
\end{itemize} 
\end{prop}

\begin{proof}
\emph{Step 1.}
In this step we consider the generalised lc modification $(X'',B''+3dpM'')$  of $(X,B+3dpM)$.
By Proposition \ref{p-gen-lc-model}, there is $c\in \N$ 
depending only on $d,p,r$ such that, letting $n=3dp$, there is a generalised lc modification 
$(X'',B''+nM'')$  of $(X,B+nM)$ so that 
\begin{itemize}
\item denoting the given morphism $X''\to X$ by $\pi$ we have 
$$
\pi^*M=M''+\sum e_iE_i''
$$ 
where $E_i''$ are the exceptional divisors of $\pi$ and $\sum e_i\le c$,

\item $M'$ descends to $X''$ as $M''$, 

\item $pM''$ is Cartier, and 

\item $({X''},B''-t(\sum e_iE_i'')+(n-t)M'')$ is generalised klt for any small $t>0$.
\end{itemize} 
 
Note that $\pi$ is not an isomorphism otherwise $M'$ descends to $X$ which means $\lambda=\infty$ which is 
not the case by assumption.\\

\emph{Step 2.}
In this step we replace $X''$ with a $\Q$-factorialisation and then 
show that we can run an MMP$/X$ on $K_{X''}+B''$ with scaling of $nM''$.
We know that 
$$
({X''},B''-t(\sum e_iE_i'')+(n-t)M'')
$$ 
is generalised klt for small $t>0$ and  
$$
K_{X''}+B''-t(\sum e_iE_i'')+(n-t)M''
$$
$$
=K_X+B''+nM''-t\pi^*M
$$
$$
\equiv K_X+B''+nM''/X.
$$
In particular, there is a small $\Q$-factorialisation 
of $X''$. Replace $X''$ with this $\Q$-factorialisation. The nefness of $K_{X''}+B''+nM''$ over $X$ 
is preserved and the property $\sum e_i\le c$ is also preserved since we do not get any new exceptional divisor.  

By Lemma \ref{l-mmp-gen-scaling}, we can run the MMP$/X$ on $K_{X''}+B''-t(\sum e_iE_i'')$ with scaling of $(n-t)M''$. 
The reduced exceptional divisor of $X''\to X$ is equal to $\rddown{B''}$. Assuming $t$ 
is small enough, each exceptional divisor appears in $B''-t(\sum e_iE_i'')$ with coefficient 
close to $1$. But then since $(X,B)$ is klt and since $B''$ is the sum of the birational 
transform $B^\sim$ and the reduced exceptional divisor of $\pi$, we can assume that  
$$
K_{X''}+B''-t(\sum e_iE_i'')=\pi^*(K_X+B)+L''
$$
where $L''\ge 0$ is exceptional whose support is equal to the reduced exceptional divisor of $\pi$. 
Thus by the negativity lemma, the MMP contracts $L''$ and it ends with $X$ as $X$ is $\Q$-factorial.

Adding $t\pi^*M$ to $K_{X''}+B''-t(\sum e_iE_i'')$ we can see that 
the above MMP is also an MMP$/X$ on $K_{X''}+B''+tM''$ with scaling of $(n-t)M''$ which in turn implies that the MMP 
is also an MMP$/X$ on $K_{X''}+B''$ with scaling of $nM''$.\\ 

\emph{Step 3.}
In this step we define $(Y,B_Y+\lambda M_Y)$. 
Let 
$$
X''=X_1''\bir X_2'' \bir \cdots X_{l-1}''\to X_l''=X
$$ 
be the steps of the MMP on $K_{X''}+B''$ of the previous step. 
By definition of MMP with scaling, in each step we have a number $\alpha_i$ which is the smallest number so that 
$K_{X_i''}+B_i''+\alpha_i nM_i''$ is nef over $X$. The last step $X_{l-1}''\to X_l''=X$ of the MMP 
is an extremal divisorial contraction. Moreover, 
$$
K_{X_{l-1}''}+B_{l-1}''+\alpha_{l-1}n M_{l-1}''\equiv 0/X
$$
and the coefficient of the exceptional divisor of $X_{l-1}''\to X$ 
in $B_{l-1}''$ is equal to $1$. Then $(X,B+\alpha_{l-1}nM)$ is generalised lc but not generalised 
klt. Therefore, we see that $\alpha_{l-1}n=\lambda$. 

Let $i$ be the smallest number so that $\alpha_i n\le \lambda <\alpha_{i-1} n$. Then 
$\alpha_j n=\alpha_{l-1}n=\lambda$, for $i\le j\le l-1$, 
hence  
$$
K_{X_i''}+B_i''+\lambda M_i''\equiv 0/X
$$ 
because in the steps $X_i''\bir X_{l}''=X$ the numbers $\alpha_jn$ are all equal to $\lambda$.
In particular, this means that $({X_i''},B_i''+\lambda M_i'')$ is a crepant model of $(X,B+\lambda M)$.

Let 
$$
(Y,B_Y+\lambda M_Y):=(X_i'',B_i''+\lambda M_i'')
$$
and let $\rho$ denote $Y\to X$.\\ 

\emph{Step 4.}
In this step we show that $(Y,\Supp B_Y)$ belongs to a bounded family of pairs.
By assumption, $A^d\le r$, $A-B$ is pseudo-effective, $pB$ is integral, and 
$(X,B)$ is klt. Then $(X,B)$ belongs to a 
bounded set of pairs, hence $(X,B)$ is $\epsilon$-lc for some $\epsilon>0$ depending only on $d,p,r$. 
Thus by Lemma \ref{l-bnd-glct-e-lc}, $\lambda$ is bounded from below away from zero depending only on $d,p,r$.

By construction, $\rho^*M=M_Y+\sum g_jG_j$ where $G_j$ denote the 
exceptional divisors of $\rho$. By Step 1, $\sum g_j\le \sum e_i\le c$. 
Therefore, there is a real number $u>0$ depending only on $d,p,r$ such that $\lambda-u>0$ and $B_Y-u\sum g_jG_j\ge 0$. 

Now 
$$
K_Y+B_Y-u\sum g_jG_j+(\lambda-u)M_Y=K_Y+B_Y+\lambda M_Y-u\rho^*M
$$
$$
=\rho^*(K_X+B+(\lambda-u)M).
$$    
Thus 
$$
(Y,B_Y-u\sum g_jG_j+(\lambda-u)M_Y)
$$ 
is a crepant model of $(X,B+(\lambda-u)M)$. 

Since $\lambda<\infty$, we have $\lambda<n$ by Lemma \ref{l-t>3dp=descend}.  
Let $b=\frac{\lambda-u}{\lambda}$. We can write 
$$
K_X+B+(\lambda-u)M=(1-b)(K_X+B)+b(K_X+B+\lambda M)
$$
where 
$$
1-b=\frac{u}{\lambda}\ge \frac{u}{n}.
$$ 
Thus since $(X,B)$ is $\epsilon$-lc, we deduce that   
$(X,B+(\lambda-u)M)$ is $\delta$-lc where $\delta=\frac{\epsilon u}{n}$. 
But then by [\ref{B-lcyf}, Theorem 2.3], the $(Y,B_Y)$ form a bounded family 
of pairs. Thus we can find a very ample divisor $A_Y$ on $Y$ such that $A_Y^d$ is bounded, say by $s\in\N$. 
Moreover, we can ensure that $A_Y-\rho^*A$ is pseudo-effective, so perhaps after replacing 
$A_Y$ with a bounded multiple we can assume $A_Y-(B_Y+M_Y)$ is pseudo-effective.\\

\emph{Step 5.}
By construction, the reduced exceptional divisor of $Y\to X$ is equal to $\rddown{B_Y}$. 
On the other hand, by Step 2, 
$$
(X'',B''-t\sum e_iE_i''+(n-t) M'')
$$ 
is generalised klt for any small $t>0$, and $M'$ descends to $X''$ as $M''$. Thus the generalised non-klt locus of
$$
(X'',B''+(n-t) M'')
$$ 
is equal to $\Supp \sum e_iE_i''=\rddown{B''}$. We then deduce that the generalised non-klt locus of 
$(X'',B''+\lambda M'')$ is equal to $\rddown{B''}$ because $M'$ descends to $X''$ as $M''$. 

By construction, $Y=X_i''$ where $i$ is the smallest number so that $\alpha_i n\le \lambda <\alpha_{i-1} n$.
This implies that $X''\bir Y=X_i''$ is an MMP on $K_{X''}+B''+\lambda M''$ over $X$. So
we see that any generalised non-klt centre of $(Y,B_Y+\lambda M_Y)$ is the birational transform of a 
generalised non-klt centre of $(X'',B''+\lambda M'')$, hence such centres are contained in $\rddown{B_Y}$. Therefore, 
$$
(Y,B_Y-t\rddown{B_Y}+\lambda M_Y)
$$ 
is generalised klt for any $t\in (0,1]$.  
\end{proof}

\subsection{Finiteness of generalised lc thresholds on bounded families}

\begin{lem}\label{l-finiteness-glct-bnd-family}
Let $d,p,r\in \N$. Then there is a finite set $\Lambda$ of rational numbers depending only on $d,p,r$ 
satisfying the following. 
Assume that 
\begin{itemize}
\item $(X,B+M)$ is a projective generalised pair of dimension $d$ with data $X'\overset{\phi}\to X,M'$, 

\item $(X,B)$ is $\Q$-factorial klt,

\item $pB$ is integral and $pM'$ is Cartier,

\item $A$ is a very ample divisor on $X$ with $A^d\le r$,

\item $A-(B+M)$ is pseudo-effective, and 

\item $\lambda<\infty$ where $\lambda$ is the generalised lc threshold of $M$ with respect to $(X,B)$.
\end{itemize}
Then $\lambda \in \Lambda$.
\end{lem}
\begin{proof}
By Proposition \ref{p-bnd-model-by-g-lct}, 
there exist a number $s\in \N$ depending only on $d,p,r$ and a 
$\Q$-factorial crepant model $(Y,B_Y+\lambda M_Y)$ of $(X,B+\lambda M)$ and 
a very ample divisor $A_Y$ on $Y$ such that 
\begin{itemize}
\item the reduced exceptional divisor of $\pi\colon Y\to X$ is equal to $\rddown{B_Y}$,

\item the pair 
$$
(Y,B_Y-t\rddown{B_Y}+\lambda M_Y)
$$ 
is generalised klt for any small $t>0$, and 

\item $A_Y^d\le s$ and $A_Y-(B_Y+M_Y)$ is pseudo-effective.
\end{itemize} 
It turns out that $\pi$ is not an isomorphism: indeed, 
since $(X,B)$ is klt, for each $t>0$, $B-t\rddown{B}=B$; so 
$$
(X,B-t\rddown{B}+\lambda M)=(X,B+\lambda M)
$$ 
is not generalised klt; thus by the second property, $\pi$ is not an isomorphism. 

Assume that $K_Y+B_Y$ is nef over $X$. Then by the negativity lemma, 
$$
K_Y+B_Y+G_Y=\pi^*(K_X+B)
$$ 
for some $G_Y\ge 0$. But this is not possible because $K_X+B$ is klt while $(Y,B_Y+G_Y)$ is not. 
Thus $K_Y+B_Y$ is not nef over $X$. Then there is a $K_Y+B_Y$-negative extremal ray over $X$. 
Since $Y$ is klt, this ray is generated 
by an extremal curve $C$ with 
$$
0<-(K_Y+\gamma B_Y)\cdot C\le 2d
$$ 
for any $\gamma$ 
sufficiently close to $1$ (see the second paragraph of [\ref{B-mmodel-II}, \S 3]).
Thus 
$$
0<-(K_Y+B_Y)\cdot C\le 2d.
$$  
Since $(Y,B_Y)$ belongs to a bounded family of pairs, 
 there are finitely many possibilities for the number $-(K_Y+B_Y)\cdot C$ as the Cartier index of 
 $K_Y+B_Y$ is bounded [\ref{B-compl}, Lemma 2.24][\ref{B-lcyf}, Lemma 3.16]. 
 
On the other hand,  since $Y$ is klt and it belongs to a bounded family, it is 
$\epsilon$-lc for some $\epsilon>0$ depending only on $d,s$, hence depending only on $d,p,r$. 
Thus by Lemma \ref{l-bnd-glct-e-lc}, 
there is a rational number $u>0$ depending only on $d,s,\epsilon,$ hence depending only on $d,p,r$ 
such that $(Y,uM_Y)$ is generalised klt. 
Similarly we can choose $u$ so that $(X,B+uM)$ is generalised klt.

By the previous paragraph and by boundedness of the length of extremal rays,  
$K_Y+uM_Y+2dA_Y$ is nef. Moreover, 
$$
(K_Y+uM_Y+2dA_Y)\cdot A_Y^{d-1}
$$
is bounded from above as $A_Y-M_Y$ is pseudo-effective. Thus 
the Cartier index of $K_Y+uM_Y+2dA_Y$ is bounded from above by [\ref{B-compl}, Lemma 2.25] 
keeping in mind that $l(K_Y+uM_Y+2dA_Y)$ is integral for some $l\in\N$ depending only on $u,p$.
This in turn implies that the Cartier index of $M_Y$ is bounded from above. 

Now with $C$ as above, we have 
$$
(K_Y+B_Y+\lambda M_Y)\cdot C=0,
$$
so 
$$
\lambda=\frac{-(K_Y+B_Y)\cdot C}{M_Y\cdot C}.
$$
Since $\lambda>u$ and since $-(K_Y+B_Y)\cdot C\le 2d$, 
we see that $M_Y\cdot C$ is bounded from above, hence it belongs to a fixed finite set as 
the Cartier index of $M_Y$ is bounded. 
But then $\lambda$ can take only finitely many possibilities.
\end{proof}

\subsection{Construction of towers of crepant models}\label{ss-tower-crep-models}

Let $d,p,r\in \N$. Assume that 
\begin{itemize}
\item $(X,B+M)$ is a projective generalised pair of dimension $d$ with data $X'\overset{\phi}\to X, M'$, 

\item $(X,B)$ is $\Q$-factorial klt,

\item $pB$ is integral and $pM'$ is Cartier,

\item $A$ is a very ample divisor on $X$ with $A^d\le r$, and 

\item $A-(B+M)$ is pseudo-effective.
\end{itemize}

We will construct a finite sequence 
$X_i,B_i,M_i,A_i,\lambda_i,r_i$ of varieties, divisors and numbers where $i=0,\dots, l$ for some $l$. The sequence 
$r_0,\dots, r_l$ is a sequence of natural numbers depending only on $d,p,r$ but the other data depend on $d,p,r,X,B,M,A$; 
a priori $l$ also depends on the latter data. 

Let $X_0=X, B_0=B, M_0=M, A_0=A$, $r_0=r$. Let $\lambda_0$ be the generalised lc threshold of $M_0$ 
with respect to $(X_0,B_0)$. The sequence to be constructed satisfies the following properties for $i>0$, 
after replacing $X'$ with a high resolution: 
\begin{enumerate}
\item $(X_i,B_i+M_i)$ is a projective generalised pair with data $X'\to X_i$ and $M'$, 

\item $pB_i$ is integral,

\item $A_i$ is a very ample divisor on $X_i$ with $A_i^d\le r_i$,

\item  $A_i-(B_i+M_i)$ is pseudo-effective, 

\item $\lambda_i$ is the generalised lc threshold of $M_i$ with respect to $(X_i,B_i)$,

\item $(X_{i},\Gamma_i+\lambda_{i-1}M_{i})$ is a $\Q$-factorial crepant model of $(X_{i-1},B_{i-1}+\lambda_{i-1}M_{i-1})$ where  
$\rddown{\Gamma_i}$ is the reduced exceptional divisor of $X_{i}\to X_{i-1}$, 

\item $B_i=\Gamma_i-\frac{1}{p}\rddown{\Gamma_i}$, 

\item $(X_{i},B_i+\lambda_{i-1}M_{i})$ is generalised klt, 

\item $\lambda_i>\lambda_{i-1}$, and 

\item $M'$ descends to $X_l$. 

\end{enumerate}

 Suppose that we have already inductively constructed $X_{i},B_i,M_{i},A_{i},\lambda_{i},r_{i}$ satisfying the properties (1)-(9). If $M'$ descends to $X_{i}$, we let $l=i$ and stop. 
Assume then that $M'$ does not descend to $X_{i}$.  We will construct $X_{i+1},B_{i+1},M_{i+1},A_{i+1},\lambda_{i+1},r_{i+1}$ as follows. 

First note that $(X_{i},B_i)$ is klt: if $i=0$, then this holds by assumption; but if $i>0$, then it follows 
from (8). Thus applying Lemma \ref{l-t>3dp=descend} to $(X_{i},B_i+\lambda_{i}M_{i})$ we see that  
$\lambda_{i}<3dp$ otherwise the lemma implies $M'$ descends to $X_{i}$, a contradiction. 

Now applying Proposition \ref{p-bnd-model-by-g-lct} with data $d,p,r_i$ to $(X_{i},B_i+M_{i})$, we see that: 
\begin{itemize}
\item there is $r_{i+1}\in \N$ depending only on $d,p,r_i$, 

\item there is a $\Q$-factorial crepant model $(X_{i+1},\Gamma_{i+1}+\lambda_{i} M_{i+1})$ of $(X_{i},B_i+\lambda_{i} M_{i})$ 
where  $\rddown{\Gamma_{i+1}}$ is the reduced exceptional divisor of $X_{i+1}\to X_{i}$, 

\item there is a very ample divisor $A_{i+1}$ on $X_{i+1}$ with $A_{i+1}^d\le r_{i+1}$, 
 
\item $A_{i+1}-(B_{i+1}+M_{i+1})$ is pseudo-effective where $B_{i+1}:=\Gamma_{i+1}-\frac{1}{p}\rddown{\Gamma_{i+1}}$, and

\item $(X_{i+1},B_{i+1}+\lambda_{i}M_{i+1})$ is generalised klt.
\end{itemize}

In particular, $pB_{i+1}$ is integral because $p\Gamma_{i+1}$ is integral as $\Gamma_{i+1}$ is the sum of the 
birational transform of $B_i$ plus the reduced exceptional divisor of $X_{i+1}\to X_{i}$.
 
We let $\lambda_{i+1}$ be the generalised lc threshold of $M_{i+1}$ with respect to $(X_{i+1},B_{i+1})$. 
Since $(X_{i+1},B_{i+1}+\lambda_{i}M_{i+1})$ is generalised klt, $\lambda_{i+1}>\lambda_{i}$. 
Therefore, $X_{i+1}$, $B_{i+1}$, $M_{i+1}$, $A_{i+1}$, $\lambda_{i+1}$, $r_{i+1}$, 
satisfy the properties (1)-(9) listed above with $i+1$ in place of $i$.

We show that the construction stops after finitely many steps, so property (10) holds for some $l$. 
By the ACC for generalised lc thresholds [\ref{BZh}, Theorem 1.5], 
the numbers $\lambda_0<\lambda_1<\dots,$ cannot form an infinite  
sequence, hence there is a minimal $l$ such that $\lambda_l=\infty$, i.e. $M'$ descends to $X_l$, and the 
construction terminates with $X_l$.  

By construction, $r_{i}$ depends only on $d,p,r_{i-1}$, and in turn $r_{i-1}$ depends only on $d,p,r_{i-2}$, and 
so on. Thus the sequence $r_0,r_1,\dots,$ depends only on $d,p,r$. But 
$l$ and the $X_i,B_i,M_i,A_i,\lambda_i$ depend on $d,p,r,X,B,M,A$. 
We will show in the next subsection that in fact $l$ is bounded from above depending only on $d,p,r$.

\subsection{Boundedness of length of towers of crepant models}

\begin{prop}\label{p-bnd-length-tower-crep-models}
Let $d,p,r\in \N$ with $p\ge 2$. Then there is $m\in \N$ depending only on $d,p,r$ satisfying the following. Assume that 
\begin{itemize}
\item $(X,B+M)$ is a projective generalised pair of dimension $d$ with data $X'\overset{\phi}\to X,M'$, 

\item $(X,B)$ is $\Q$-factorial klt,

\item $pB$ is integral and $pM'$ is Cartier,

\item $A$ is a very ample divisor on $X$ with $A^d\le r$, and 

\item $A-(B+M)$ is pseudo-effective.
\end{itemize}
Let $X_i,B_i,M_i,A_i,\lambda_i,r_i$, where $0\le i\le l$, be the sequence 
constructed in \ref{ss-tower-crep-models} for $d,p,r,X,B,M,A$. 
Then $l\le m$. Moreover, 
\begin{enumerate}
\item[(i)] $\Supp B_l$ contains the support of the birational transform of $B$ and the 
exceptional divisor of $X_l\to X$, 

\item[(ii)] $({X}_l, B_l)$ belongs to a bounded set of pairs depending only on $d,p,r$, and  

\item[(iii)] $M'$ descends to ${X}_l$. 
\end{enumerate}     
\end{prop}
\begin{proof}
If the proposition is not true, then there is a sequence $X^j,B^j,M^j,A^j$ as in the proposition, $j\in \N$, 
so that if $X_i^j,B^j_i,M_i^j,A_i^j,\lambda_i^j,r_i$ is the sequence constructed for $X^j,B^j,M^j,A^j$ where
$0\le i\le l^j$, then 
the $l^j$ form an infinite strictly increasing sequence of numbers. We will derive a contradiction.

By Lemma \ref{l-finiteness-glct-bnd-family}, for each $i$, there is a finite set of 
rational numbers $\Lambda_i$ depending only on $d,p,r_i$, hence depending only on $d,p,r,i$, 
such that the generalised lc thresholds 
$\lambda_i^j\in \Lambda_i$ whenever $\lambda_i^j\neq \infty$, that is, when $i<l^j$. Therefore, 
there is an infinite set $J_0\subseteq \N$ such that $\alpha_0:=\lambda_0^j\neq \infty$ 
is independent of $j$ for $j\in J_0$. 
Similarly, there is an infinite subset $J_1\subseteq J_0$ such that $\alpha_1:=\lambda_1^j\neq \infty$ is 
independent of $j$ for $j\in J_1$. In particular, since by construction, $\lambda_0^j<\lambda^j_{1}$ 
for every $j\in J_1$, we see that $\alpha_0<\alpha_1$.

Continuing the above process gives a decreasing sequence 
$$
\dots \subseteq  J_i\subseteq J_{i-1}\dots \subseteq J_0\subseteq \N
$$
of infinite subsets and thresholds $\alpha_i:=\lambda_i^j\neq \infty$ 
independent of $j\in J_i$ so that we get an infinite strictly increasing sequence 
$
\alpha_0<\alpha_1<\dots
$
of generalised lc thresholds. This contradicts the ACC for generalised lc thresholds [\ref{BZh}, Theorem 1.5]. This proves the existence of $m$.

Now we show (i)-(iii) hold. Consider $X,B,M,A$ as in the proposition and the associated sequence $X_i,B_i,M_i,A_i,\lambda_i, r_i$ where $0\le i\le l$. By the above arguments, $l\le m$ is bounded from above. Also the numbers $r_0,\dots,r_l$ depend only on $d,p,r$. By the properties of the sequence listed in \ref{ss-tower-crep-models},
 $(X_l,B_l+M_l)$ is a projective generalised pair of dimension $d$ with data $X'\to X_l$ and $M'$, 
 $pB_l$ is integral and $pM'$ is Cartier,
 $A_l$ is a very ample divisor on $X_l$ with $A_l^d\le r_l$, and $A_l-(B_l+M_l)$ is pseudo-effective. Therefore, $A_l^d\le r_l$ and $A_l^{d-1}\cdot B_l\le r_l$, hence $(X_l,B_l)$ belongs to a bounded family of pairs. Moreover, since $p\ge 2$, $\Supp B_l$ contains the reduced exceptional divisor of $X_l\to X$ as well as the support of the birational transform of $\Supp B$, by construction of the sequence. Finally, $M'$ descends to $X_l$ again by construction.   
\end{proof}

\subsection{Proof of Theorem \ref{t-descent-nef-divs-to-bnd-models}}

We are now ready to prove our main result on the descent of nef divisors to bounded models. 

\begin{proof}[Proof of Theorem \ref{t-descent-nef-divs-to-bnd-models}]
\emph{Step 1.}
Pick 
$$
(X,B+M)\in \mathcal{G}_{lc}(d,\Phi,<\!\!v)
$$ 
with data $X'\overset{\phi}\to X$, $M'=\sum \mu_iM_i'$ where $\mu_i>0$.
In this step we reduce the problem to the situation in which $(X,B+M)$ is generalised klt, 
$pB$ is integral, $p\mu_i$ are integral, and $K_{X}+B+(1-u)M$ is big for some $p\in \N$ and 
$u\in \Q^{>0}$ depending only on $d,\Phi$.

We can assume that $\phi$ is a log resolution of $(X,B)$. Let 
$\Gamma'$ be the sum of the birational transform of $B$ plus the reduced exceptional divisor of $\phi$. 
Then 
$$
K_{X'}+\Gamma'+M'=\phi^*(K_X+B+M)+G'
$$ 
where $G'$ is effective and exceptional over $X$. By [\ref{BZh}, Theorem 8.1], 
$K_{X'}+\alpha \Gamma'+\alpha M'$ is big for some real number $\alpha\in (0,1)$ depending only on $d,\Phi$. 
Set $\beta=\frac{\alpha+1}{2} \in (\alpha,1)$. 
Since $\Phi$ is DCC, there exists $2\le p\in \N$ depending only on $\Phi,\beta$, hence depending only on $d,\Phi$, such that for each $0<\mu \in \Phi$ we have $\frac{1}{p}<\mu-\beta\mu$ so there is a smallest natural number 
$q$ such that $\beta \mu\le \frac{q}{p}<\mu$. Denote $\bar{\mu}:=\frac{q}{p}$. 

Define a boundary $\Delta'\le \Gamma'$ 
by replacing each non-zero coefficient $\mu$ of $\Gamma'$ with $\bar{\mu}$, and define $N':=\sum \bar{\mu_i}M_i'$. 
This way we get a generalised klt pair $(X',\Delta'+N')$ where  
$p\Delta'$ and $pN'$ are integral. 
Moreover, fixing $u\in (0,1-\frac{\alpha}{\beta})$, we have that 
$K_{X'}+\Delta'+(1-u)N'$ is big because  $\Delta'-\alpha \Gamma'\ge 0$ and because 
$$
(1-u)N'-\alpha M'=\sum ((1-u)\bar{\mu_i}-\alpha\mu_i)M_i'
$$
is nef as 
$$
(1-u)\bar{\mu_i}-\alpha\mu_i\ge (1-u)\beta \mu_i-\alpha\mu_i\ge \frac{\alpha}{\beta}\beta\mu_i-\alpha\mu_i=0.
$$
Replacing $(X,B+M)$ with $(X',\Delta'+N')$ 
we can assume that $(X,B+M)$ is generalised $\frac{1}{p}$-lc, $pB$ is integral and $p\mu_i$ are integral, and that 
$K_{X}+B+(1-u)M$ is big where $p,u$ are chosen depending only on $d,\Phi$.\\

\emph{Step 2.}
In this step we find a bounded birational model of $(X,B+M)$.
By [\ref{BZh}, Theorem 1.3], there is a natural number $m$ depending only on $d,p$ 
such that $p|m$ and $|m(K_X+B+M)|$ defines a birational map. In particular, there is an integral divisor 
$$
0\le L\sim m(K_X+B+M)
$$
with $\vol(L)<m^dv$. 

By [\ref{BZh}, Lemma 4.4], $(X,B+M)$ has a minimal model as $(X,B+M)$ is generalised klt and 
$K_X+B+M$ is big. Replacing $(X,B+M)$ with the minimal model and replacing $L$ with its pushdown we can assume that $K_X+B+M$ is nef and big.

 Applying [\ref{B-compl}, Proposition 4.4] to $X,B,L$, we deduce that 
there exist a natural number $c$ and a bounded set of couples $\mathcal{Q}$ depending only on $d,p,v$ such that 
there is a projective log smooth $(X'',{\Sigma}'')\in \mathcal{Q}$ 
and a birational map ${X}''\bir X$ such that 
\begin{itemize}
\item  ${{\Sigma}}''$  contains the exceptional 
divisor of  ${X}''\bir X$ and the birational transform of $\Supp (B+{L})$;

\item replacing $X'$ we can assume $\psi\colon X'\bir {X}''$ is a morphism, and 

\item each coefficient of $\psi_*\phi^*L$ is at most $c$. 

\end{itemize}

In particular, there is a very ample divisor $A''$ on $X''$ with $A''^d$ bounded from above, say by $r$, 
and such that $A''-\Sigma''$ and $A''-\psi_*\phi^*L$ are pseudo-effective.\\

\emph{Step 3.}
In this step we introduce divisors $\Theta'',M''$ on $X''$ and reduce the problem to the situation in which 
$A''-(\Theta''+M'')$ is pseudo-effective. 
Let $\Theta'$ on $X'$ be the birational transform of $B$ plus $(1-\frac{1}{p})$ times the reduced exceptional divisor of 
$\phi$. Let ${\Theta}'', {M}''$ be the pushdowns of $\Theta',M'$. 
Then ${A}''-{\Theta}''$ is pseudo-effective as ${\Theta}''\le {\Sigma}''$. Note that $p\Theta''$ is integral.

On the other hand, since $K_X+B+(1-u)M$ is big, we can write 
$$
\frac{1}{m}L\sim_\Q K_X+B+M=K_X+B+(1-u)M+uM \sim_\Q uM+R
$$ 
where $R\ge 0$. 
Thus 
$$
\psi_*\phi^*\frac{1}{m}L\sim_\Q \psi_*\phi^*uM+\psi_*\phi^*R
$$
which implies that 
$$
\frac{1}{m}A''-\psi_*\phi^*uM\sim_\Q \frac{1}{m}A''-\psi_*\phi^*\frac{1}{m}L+\psi_*\phi^*R
$$
is pseudo-effective. Therefore, replacing $A''$ with a bounded multiple we can assume that ${A}''-\psi_*\phi^*M$ 
is pseudo-effective which in turn implies $A''-M''$ is pseudo-effective as $\psi_*\phi^*M\ge M''$ because 
$\phi^*M\ge M'$ since $M'$ is nef. Since 
$A''-\Theta''$ is pseudo-effective by the previous paragraph, replacing ${A}''$ with $2{A}''$ we can assume that 
${A}''-({\Theta}''+{M}'')$ is pseudo-effective.\\

\emph{Step 4.}
In this step we finish the proof by 
applying Proposition \ref{p-bnd-length-tower-crep-models}.  By construction, 
\begin{itemize}
\item $({X}'', {\Theta}''+{M}'')$ is a projective generalised pair of dimension $d$ with data $X'\to X'',M'$, 

\item $({X}'', {\Theta}'')$ is log smooth klt,

\item $p\Theta''$ is integral and $pM'$ is Cartier,

\item $A''$ is a very ample divisor on $X''$ with $A''^d\le r$, and

\item $A''-(\Theta''+M'')$ is pseudo-effective.
\end{itemize}

Applying Proposition \ref{p-bnd-length-tower-crep-models} to $d,p,r,X'',\Theta'',M'',A''$, 
we deduce that there is a $\Q$-factorial generalised klt pair $(\overline{X}, \overline{\Theta}+\overline{M})$ 
with data $\rho\colon X'\to \overline{X}$ and $M'$ such that 
\begin{itemize}
\item $\pi\colon \overline{X}\bir X''$ is a morphism,

\item $\overline{\Sigma}:=\Supp \overline{\Theta}$ contains the support of the birational transform of $\Theta''$ and the reduced exceptional divisor of $\pi$, 

\item $(\overline{X}, \overline{\Sigma})$ belongs to a bounded set of couples depending only on $d,p,r$, and 

\item $M'$ descends to $\overline{X}$. 
 
\end{itemize}     
  
Since  $M'$ descends to $\overline{X}$, $M'=\rho^*\rho_*M'$.
Since $M'=\sum \mu_i M_i'$ and $\mu_i>0$ by definition, and since each $M_i'$ is nef, we get 
$M'_i=\rho^*\rho_*M_i'$ for each $i$, hence each $M_i'$ descends to $\overline{X}$.
 
By construction, $\overline{\Sigma}$ contains the exceptional divisors of $\overline{X}\bir X$ and 
the support of the birational transform of $B$: indeed, any prime exceptional divisor of $\overline{X}\bir X$ 
is either an exceptional divisor of $\pi$ or a component of the birational transform of $\Theta''$ so it is a 
component of $\overline{\Sigma}$; 
moreover, any component of the birational transform of $B$ is either exceptional over $X''$ or it is 
a component of the birational transform of $\Theta''$ so again it is a component of $\overline{\Sigma}$. 
 
There exist a bounded set of couples $\mathcal{P}$ depending only on $d,p,r$ such that  
$(\overline{X}, \overline{\Sigma})$ has a log resolution $Y\to \overline{X}$ so that if $\Sigma_Y$ is 
the sum of the birational transform of $\overline{\Sigma}$ and the reduced exceptional divisor of 
$Y\to \overline{X}$, then $(Y,\Sigma_Y)$ belongs to $\mathcal{P}$. Moreover, we can find a very ample divisor $H_Y$ so that $(Y,\Sigma_Y+H_Y)$ is log smooth and bounded and $H_Y-M_Y$ is pseudo-effective where $M_Y$ is the pullback of $\overline{M}=\sum \mu_i \overline{M}_i$ (here we can use pseudo-effectivity of $A''-M''$). Now replacing $(\overline{X}, \overline{\Sigma})$ with $(Y,\Sigma_Y+H_Y)$ and replacing $\mathcal{P}$ accordingly we are done.
\end{proof}

We derive the following result on descent of nef divisors which is in some sense more general than 
\ref{t-descent-nef-divs-to-bnd-models} since we do not put any restriction on singularities.  
It essentially says that if $X$ varies in a bounded family and if $pM'$ is a nef Cartier divisor on 
some birational model whose image on $X$ has bounded ``degree", then we can descend $M'$ to a bounded 
model of $X$. 

\begin{thm}\label{t-descent-nef-divs-to-bnd-models-1}
Let $d,r\in \N$ and $\delta\in \R^{>0}$. 
Then there is a bounded set of couples $\mathcal{P}$ depending only on $d,r,\delta$ 
satisfying the following. Assume that 
\begin{itemize}
\item $\phi\colon X'\to X$ is a birational morphism between normal projective varieties of dimension $d$, 

\item $B\ge 0$ is an $\R$-divisor on $X$ whose non-zero coefficients are $\ge \delta$,

\item $M'=\sum \mu_iM_i'$ where $M'$ are nef and Cartier on $X'$ and $\mu_i\ge \delta$, 

\item $A$ is a very ample divisor on $X$ with $A^d\le r$, and  

\item $A-(B+M)$ is pseudo-effective where $M=\phi_*M'$.
\end{itemize}
Then there exist a log smooth couple $(\overline{X},\overline{\Sigma})\in \mathcal{P}$ and a birational morphism  
$\overline{X}\to X$ such that 
\begin{itemize}
\item $\overline{\Sigma}$ contains the exceptional divisors of $\overline{X}\to X$ and the support of the birational 
transform of $B$, and 

\item each $M_i'$ descends to $\overline{X}$.
\end{itemize}
\end{thm}
\begin{proof}
Since $A^d\le r$, $X$ belongs to a bounded family of varieties. Replacing $A$ we can assume 
$A-(K_X+B+M)$ is pseudo-effective. Replace $\delta$ with $\frac{1}{p}<\delta$ 
for some $p\in \N$. Decreasing the coefficients of $B$ and the $\mu_i$ we can assume that 
$pB$ is integral, the coefficients of $B$ are $\le 1$, and that $p\mu_i$ are all integral. 

Replacing $X'$ we can assume $\phi$ is a log resolution of $(X,B)$. Let $B'$ be the sum of the 
birational transform of $B$ and the reduced exceptional divisor of $\phi$. Let $N'=M'+3d A'$ 
where $A'=\phi^* A$. Then $K_{X'}+B'+N'$ is big, $pB'$ is integral, and $pN'$ is Cartier. 
Moreover, 
$$
\vol(K_{X'}+B'+N')\le \vol(K_X+B+M+3dA)
$$
$$
\le \vol((1+3d)A)\le (1+3d)^dr
$$
because $A-(K_X+B+M)$ is pseudo-effective. 

Considering $(X',B'+N')$ as a generalised pair with 
data $X'\to X'$ and $N'$, 
$$
(X',B'+N')\in \mathcal{G}_{lc}(d,\Phi,<\!\!v)
$$
where  $\Phi=\{\frac{i}{p}\mid i\in \N\}$ and $v=(1+3d)^dr+1$. 
Thus by Theorem \ref{t-descent-nef-divs-to-bnd-models}, 
there exists a bounded set of couples $\mathcal{P}$ depending only on $d,p,v$, hence depending only $d,r,\delta$, 
such that there exist a log smooth couple $(\overline{X},\overline{\Sigma})\in \mathcal{P}$ and a birational map 
$\overline{X}\bir X'$ such that 
\begin{itemize}
\item $\overline{\Sigma}$ contains the exceptional divisors of $\overline{X}\bir X'$ and the support of the birational 
transform of $B'$, and 

\item $N'$ descends to $\overline{X}$.
\end{itemize}

Thus there is a common resolution $X''$ of $X'$ and $\overline{X}$ such that if 
$N''=M''+3dA''$ denotes the pullback of $N'=M'+3dA'$, then $N''$ is the pullback of a divisor on $\overline{X}$. 
In particular,  any curve $C$ contracted by $X''\to \overline{X}$ is also contracted by $X''\to X$  
because from $N''\cdot C=0$ we get $A''\cdot C=0$. 
This implies that the induced map $\overline{X}\bir X$ is actually a morphism. 

Now $\overline{\Sigma}$ contains the exceptional divisors of $\overline{X}\to X$ and the support of the birational 
transform of $B$: indeed any prime exceptional divisor of $\overline{X}\to X$ is either exceptional over $X'$ 
or else is a component of the birational transform of $B'$; moreover, the support of the birational 
transform of $B$ is contained in the support of the birational transform of $B'$. 

On the other hand,  
since $N'=\sum \mu_iM_i'+3dA'$ descends to $\overline{X}$ and since the $M_i'$ and $A'$ are all nef, 
each $M_i'$ descends to $\overline{X}$. 
\end{proof}


\section{\bf DCC of volume of generalised pairs}

In this section we prove Theorem \ref{t-dcc-vol-gen-pairs}. 
The proof consists of two main parts. One part treats the theorem for pairs birational to a fixed variety. 
This is Proposition \ref{p-dcc-vol-gen-pair-fixed-base}, a variant of which was first proved for generalised pairs 
in [\ref{Filipazzi}, Theorem 1.10] where the proof is nearly identical to the proof of 
[\ref{HMX1}, Proposition 5.1] for usual pairs. Our proof here also generally follows  
[\ref{HMX1}] but it is somewhat simpler than both [\ref{HMX1}][\ref{Filipazzi}]. 
The second part reduces the general statement to the case in part one which is again modelled on [\ref{HMX1}] 
but the nef divisors involved need extra care. 
For this part also we generally follow [\ref{Filipazzi}] although some of the details are different.

\subsection{DCC of volumes of pairs with fixed birational model}

\begin{prop}\label{p-dcc-vol-gen-pair-fixed-base}
Let $\Phi\subset \R^{\ge 0}$ be a DCC set. 
Assume that $(Y,\Delta)$ is a projective $\Q$-factorial strictly toroidal pair and $N_1,\dots,N_q$ are 
pseudo-effective $\R$-Cartier $\R$-divisors on $Y$.
Let $\mathcal{F}$ be the set of $(X,B+M)$ equipped with a birational 
morphism $\phi\colon X\to Y$ where 
\begin{itemize}
\item $(X,B)$ is a projective lc pair, 

\item the coefficients of $B$ are in $\Phi$,  

\item $\phi_*B\le \Delta$, and 

\item $M=\phi^*\sum_1^q \mu_jN_j$ where $\mu_j\in \Phi$. 
\end{itemize}
Then 
$$
\{\vol(K_X+B+M) \mid (X,B+M)\in \mathcal{F}\}
$$
satisfies the DCC.
\end{prop}
\begin{proof}
Note that $(X,B+M)$ is not necessarily a generalised pair because the $N_i$ may not be nef. 
However, in practice when we apply the proposition, the $N_i$ would be nef.\\

\emph{Step 1.}
Since 
$\Phi$ is DCC, its closure $\overline{\Phi}\subset [0,\infty)$ is also DCC. Replacing $\Phi$ with 
$\overline{\Phi}\cup \{1\}$ we can assume $\Phi$ contains its limiting points and $1\in \Phi$.
If the proposition does not hold, then there is a sequence $(X_i,B_i+M_i)$, 
$\phi_i\colon X_i\to Y$, $M_i=\phi^*\sum \mu_{i,j}N_{j}$ as in the proposition 
such that the volumes 
$$
v_i=\vol(K_{X_i}+B_i+M_i)
$$ 
form a strictly decreasing sequence of numbers. \\ 

\emph{Step 2.}
Replacing $(X_i,B_i)$ with a $\Q$-factorial dlt model we can assume $(X_i,B_i)$ is $\Q$-factorial dlt.  
Running an MMP on $K_{X_i}+B_i$ over $Y$ with scaling of some ample divisor, we reach a model on which 
the pushdown of $K_{X_i}+B_i$ is a limit of movable$/Y$ $\R$-divisors. Replacing $X_i$ with that model 
we can assume that $K_{X_i}+B_i$ is a limit of movable$/Y$ $\R$-divisors. By the general negativity lemma 
[\ref{B-lc-flips}, Lemma 3.3], 
$$
K_{X_i}+B_i+G_i=\phi_i^*(K_Y+C_i)
$$
for some $G_i\ge 0$ where $C_i={\phi_i}_*B_i$. Since $C_i\le \Delta$, for every component $D$ of $B_i+G_i$ 
the log discrepancies satisfy  
$$
a(D,Y,\Delta)\le a(D,Y,C_i)=a(D,X_i,B_i+G_i)<1.
$$
Since $(Y,\Delta)$ is strictly toroidal, it is lc and $K_Y+\Delta$ is Cartier (see \ref{ss-toroidal-pairs} or [\ref{B-toroidalisation}, Lemma 3.11]), hence $a(D,Y,\Delta)=0$, so $D$ is a toroidal divisor with respect to $(Y,\Delta)$ by definition.\\ 

\emph{Step 3.}
For each $i$, let $\mathbf{M}_i$ be the b-divisor with trace $\mathbf{M}_{i,{X_i}}=B_i$ and whose coefficient 
in any exceptional prime divisor over $X_i$ is $1$. 
The set of prime divisors $D$ over $Y$ such that $0<\mu_D\mathbf{M}_i<1$ for some $i$ is countable because for each $i$ there are only finitely many $D$ 
with $0<\mu_D\mathbf{M}_i<1$. Moreover, if $\mu_D\mathbf{M}_i=1$ and $\mu_D\mathbf{M}_j=0$, and $D$ is not a component of $B_i$, then $D$ is exceptional over $X_i$ but is a divisor on $X_j$. Hence, for each pair $i,j$, there are only finitely many such $D$. It follows that the set of prime divisors $D$ for which the sequence $\mu_D\mathbf{M}_i$ is not constant is countable. Using the DCC property and a diagonal argument, after passing to a subsequence we may assume that $\mu_D\mathbf{M}_i$ is eventually nondecreasing for every prime divisor $D$ over $Y$.
Thus we can define a limiting b-divisor $\mathbf{C}$ by setting $\mu_D\mathbf{C}=\lim_i \mu_D \mathbf{M}_i$
for each $D$. By Step 1, the coefficients of $\mathbf{C}$ are in $\Phi$.\\

\emph{Step 4.}
Assume $(Y',\Delta')$ is a projective $\Q$-factorial strictly toroidal pair and $\pi\colon Y'\to Y$ a birational contraction such that $K_{Y'}+\Delta'$ is the pullback of $K_Y+\Delta$. 
For each $i$, taking a common log resolution $W$ of $X_i,Y'$ and then running 
an MMP on $K_W+\mathbf{M}_{i,W}$ over $Y'$ as in Step 2, we can construct $(X_i',B_i')$ 
over $Y'$ where $B_i'$ is the birational transform of $B_i$ plus the reduced exceptional divisor of $X_i'\bir X_i$, 
that is, $B_i'=\mathbf{M}_{i,X_i'}$.  
Since $K_{X_i}+B_i$ is a limit of movable$/Y$ $\R$-divisors, $X_i\bir X_i'$ does not contract any divisor. 
Note that 
$$
v_i=\vol(K_{X_i'}+B_i'+M_i')
$$ 
where $\phi_i'$ denotes $X_i'\to Y'$ and $M_i'=\phi_i'^*\pi^*\sum \mu_{i,j}N_{j}$. 

In the subsequent steps, when necessary, we will feel free to replace $(Y,\Delta)$ 
with $(Y',\Delta')$ and replace $(X_i,B_i+M_i)$ with $(X_i',B_i'+M_i')$. 
This preserves the b-divisors $\mathbf{M}_i$ and their limit $\mathbf{C}$ as $X_i\bir X_i'$ does not contract any divisor.\\

\emph{Step 5.}
Let $C:=\mathbf{C}_Y=\lim_i C_i=\lim_i{\phi_i}_*B_i$.
Let $\mathcal{D}_{\le}(Y,C)$ (resp. $\mathcal{D}_{<}(Y,C)$) be the set of exceptional prime divisors $D$ over $Y$, 
toroidal with respect to $(Y,\Delta)$, such that 
$$
(*) \ \ \ \mu_D\mathbf{C}\le 1-a(D,Y,C) \ \ \ \ \ \ \ \ \ (\mbox{resp. $\le$ replaced by $<$}).
$$  

Assume $D\in \mathcal{D}_{\le}(Y,C)$.
Since $D$ is toroidal with respect to $(Y,\Delta)$, 
there is a toroidal extremal birational contraction $\pi\colon Y'\to Y$ which extracts $D$ but no other divisors (see \ref{ss-toroidal-pairs}(4)).  Moreover, if $K_{Y'}+\Delta'$ is the pullback of 
$K_Y+\Delta$, then $(Y',\Delta')$ is $\Q$-factorial strictly toroidal.

Letting $C':=\mathbf{C}_{Y'}$, we have $K_{Y'}+C'\leq \pi^*(K_Y+C)$ by $(*)$ where equality holds iff $D\notin \mathcal{D}_{<}(Y,C)$. In particular,
$$
\mathcal{D}_{\le}(Y',C')\subsetneq\mathcal{D}_{\le}(Y,C)
$$ 
as $D$ does not belong to the former set. Moreover, 
$$
\mathcal{D}_{<}(Y',C')\subseteq\mathcal{D}_{<}(Y,C)
$$ 
and $\subsetneq$ holds if $D\in \mathcal{D}_{<}(Y,C)$.\\

\emph{Step 6.}
We associate a weight $w=(p,r,l,d)$ to $(Y,C)$ as follows. 
We first define the weight $w_V$ for each non-klt centre $V$ of $(Y,C)$. 
Let $r$ be the codimension of $V$. Since $(Y,\Delta)$ is strictly toroidal, 
we have an adjunction formula   
$K_V+C_V=(K_Y+C)|_V$ (see \ref{ss-toroidal-pairs}(3)). 
Let $l$ be the number of prime exceptional divisors $S$ over $V$ mapping into
the klt locus of $(V,C_V)$ and with $a(S,V,C_V)<1$. And let $d$ be the sum of the coefficients of 
$C_V$. Put $w_V=(r,l,d)$.

Now we define the weight $w$ of $(Y,C)$.
Let $\mathcal{V}(Y,C)$ be the set of non-klt 
centres of $(Y,C)$ which intersect the centre of some $D\in \mathcal{D}_{<}(Y,C)$. 
Let $p$ be the number of elements of $\mathcal{V}(Y,C)$.
If $p=0$, then let $r=0$, 
let $l$ be the number of elements of $\mathcal{D}_{<}(Y,C)$, which is finite in this case, 
and let $d$ be the sum of the coefficients of $C$.
If $p>0$, then choose $V\in \mathcal{V}(Y,C)$ with maximal $w_V=(r,l,d)$ with respect to the lexicographic order 
on 3-tuples. Then in any case let $w:=(p,r,l,d)$. 

Note that $p,r,l\in \N\cup \{0\}$ while $d$ belongs to a DCC set $\Theta$ depending only on 
$\Phi, \dim Y$ because the coefficients of $C$ belong to $\Phi$ and the coefficients of $C_V$ 
belong to a DCC set $\Psi$ depending only on $\Phi, \dim Y$. We consider the lexicographic order on the weights $w$.\\

\emph{Step 7.}
Assume $w=(p,r,l,d)$ is the weight of $(Y,C)$. In this step assume $(p,r,l)\neq (0,0,0)$. 
If $p=0$, then $r=0$ but $l\neq 0$, so there is $D\in \mathcal{D}_{<}(Y,C)$ with centre $L$ 
contained in the klt locus of $(Y,C)$. 
But if $p>0$, then $r>0$ and there is $D\in \mathcal{D}_{<}(Y,C)$ with centre $L$ intersecting some 
element of $\mathcal{V}(Y,C)$.
In either case, let $\pi\colon Y'\to Y$ and $C'$ be as in Step 5 constructed for $D$. 
Then each element of $\mathcal{V}(Y',C')$ maps birationally onto an 
element of $\mathcal{V}(Y,C)$.

Assume $w'=(p',r',l',d')$ is the weight of $(Y',C')$ defined similarly as in Step 6. Then $p'\le p$ by the previous paragraph.
If $p=0$, then $p'=r'=0$ but $l'<l$ by Step 5, so $w'<w$. If $p'<p$, then again $w'<w$; 
for example, if there is a non-klt centre of $(Y,C)$ contained in $L$, then $p'<p$ holds. 
In these cases, we replace $(Y,\Delta)$ with $(Y',\Delta')$ as in Step 4, 
so decrease the weight of $(Y,C)$. 

Now assume $p'=p>0$, in particular, no non-klt centre of $(Y,C)$ is contained in $L$.
Assume that $V\in \mathcal{V}(Y,C)$ is minimal with respect to inclusion 
among the non-klt centres intersecting $L$. Let 
$V'\subset Y'$ be its birational transform. Assume $w_V=(s,m,e)$ and $w_{V'}=(s',m',e')$. 
We claim that $w_{V'}<w_V$. Clearly $s'=s$.  
Define 
$$
K_{V}+C_{V}=(K_{Y}+C)|_{V} ~~~\mbox{and}~~~ K_{V'}+C_{V'}'=(K_{Y'}+C')|_{V'}
$$ 
by adjunction. Then $(V,C_V)$ is klt near $L\cap V$ by the minimality of $V$. 
By construction, $K_{Y'}+C'\lneq \pi^*(K_Y+C)$. 
Then $K_{V'}+C_{V'}'\lneq \rho^*(K_V+C_V)$ where $\rho$ denotes $V'\to V$,  
and equality holds over the non-klt locus of $(V,C_V)$ as this non-klt locus is disjoint from $L$. Thus $m'\le m$.
If $\rho$ contracts a divisor, then $m'<m$. But if $\rho$ does not contract any divisor, 
then $e'<e$. Thus the claim $w_{V'}<w_V$ follows. 

Assume $U'\in\mathcal{V}(Y',C')$ and let $U$ be its image on $Y$. 
If $U$ does not intersect $L$, then $w_{U'}=w_U$. 
If $U$ intersects $L$ and $U$ is minimal among the non-klt centres 
intersecting $L$, then $w_{U'}<w_U$ by the previous paragraph. 
But if $U$ intersects $L$ and $U$ is not minimal among the non-klt centres 
intersecting $L$, then there is minimal $T\in \mathcal{V}(Y,C)$ intersecting $L$ and $T\subsetneq U$, so
$w_{U'}<w_{T'}<w_T$ where $T'$ is the birational transform of $T$. 

Now replace $(Y,\Delta)$ with $(Y',\Delta')$ as in Step 4, and repeat this step. 
The above arguments show that eventually we can decrease the weight of $(Y,C)$ by choosing $D$ appropriately in each step. 
Since we cannot decrease the weight infinitely many times we arrive at the case $p=r=l=0$. 
\\ 

\emph{Step 8.}
From now we assume $w=(p,r,l,d)=(0,0,0,d)$, equivalently, $\mathcal{D}_{<}(Y,C)=\emptyset$.
If there exists $D\in \mathcal{D}_{\le}(Y,C)$ with $a(D,Y,C)<1$ and whose centre 
on $Y$ is not contained in $\Supp C$, then we replace $(Y,C)$ with $(Y',C')$ where the latter is constructed for $D$
as in Step 5. The condition  $\mathcal{D}_{<}(Y,C)=\emptyset$ is preserved. 
Since $Y$ is klt outside $\Supp C$, repeating this finitely many times 
we can assume there is no such $D$. 

Pick $t\in (0,1)$. Then $(Y,tC)$ is klt, so there are finitely many prime divisors $D$ 
over $Y$ with $a(D,Y,tC)< 1$. We claim that, if $i\gg 0$, then $K_{X_i}+B_i\ge \phi_i^*(K_Y+tC)$: 
it is enough to check this in prime divisors $D\subset X_i$ as in the previous sentence which are 
automatically toroidal with respect to $(Y,\Delta)$; 
for such $D$ we can assume $\mu_DB_i$ is sufficiently close to $\mu_D \mathbf{C}_{X_i}$; now 
if $D$ is not exceptional over $Y$, then the claim follows from $t<1$;
if it is exceptional, then since $\mathcal{D}_{<}(Y,C)=\emptyset$, writing 
$K_{X_i}+A_i=\phi_i^*(K_Y+C)$, we have $\mu_D \mathbf{C}_{X_i}\ge \mu_DA_i$; thus  
$\mu_D \mathbf{C}_{X_i}=\mu_DA_i$ because by Step 2 and $C_i={\phi_i}_*B_i\le C$ we have $B_i\le A_i$;  
therefore $D\in \mathcal{D}_{\le}(Y,C)$ and by the previous paragraph, the centre of $D$ is contained in $\Supp C$;   
the claim then again follows from $t<1$.

Now for $i\gg 0$, we get  
$$
v_i=\vol(K_{X_i}+B_i+M_i)\ge \vol(K_Y+tC+\sum \mu_{i,j}N_j).
$$
We can assume that, for each $j$, the $\mu_{i,j}$ are increasing sequences with limit $\mu_j$. 
 Therefore, 
$$
v=\lim_i v_i\ge \lim_i \vol(K_Y+tC+\sum \mu_{i,j}N_j)=\vol(K_Y+tC+\sum \mu_{j}N_j)
$$
because the volume is a continuous function on the numerical classes of divisors [\ref{Lazarsfeld}, Corollary 2.2.45].
Taking the limit when $t$ approaches $1$, we see that for each $i$, we have 
$$
\vol(K_Y+C+\sum \mu_{i,j}N_j)\ge \vol(K_{X_i}+B_i+M_i)\ge v\ge \vol(K_Y+C+\sum \mu_{j}N_j).
$$
Since $N_j$ are pseudo-effective and $\mu_{i,j}\le \mu_j$, all the inequalities are equalities and $v_i=v$, a contradiction.
\end{proof}


\subsection{Proof of Theorem \ref{t-dcc-vol-gen-pairs}}

\begin{proof}[Proof of Theorem \ref{t-dcc-vol-gen-pairs}]
\emph{Step 1.}
Assume the theorem does not hold. Then there is a sequence $(X_i,B_i+M_i)$ of generalised pairs 
in  $\mathcal{G}_{glc}(d,\Phi)$ with data $X_i'\to X_i$ and $M'=\sum \mu_{i,j}M_{i,j}'$ 
so that the volumes 
$$
v_i=\vol(K_{X_i}+B_i+M_i)
$$ 
form a strictly decreasing sequence of numbers. Adding $1$ to $\Phi$ and replacing $(X_i,B_i+M_i)$ with a  $\Q$-factorial generalised
dlt model (such models exist by [\ref{BZh}, Lemma 4.5] and its proof; also see [\ref{B-compl}, 2.13(2),(3)]), we can assume $X_i$ is $\Q$-factorial. We can discard any $M_{i,j}'\equiv 0$.\\

\emph{Step 2.}
By Theorem \ref{t-descent-nef-divs-to-bnd-models},
there exists a bounded set of couples $\mathcal{P}$ depending only on $d,\Phi, v_1$ such that for each $i$, 
there exist a log smooth couple $(\overline{X}_i,\overline{\Sigma}_i)\in \mathcal{P}$ and a birational map 
$\psi_i\colon X_i\bir \overline{X}_i$ satisfying the following: 
\begin{itemize}
\item $\overline{\Sigma}_i$ contains the exceptional divisors of $\psi_i^{-1}$ and the support of the birational 
transform of $B_i$, and 

\item each $M_{i,j}'$ descends to $\overline{X}_i$, say as $\overline{M}_{i,j}$, and 
\item $\overline{A}_i-\sum \mu_{i,j}\overline{M}_{i,j}$ is pseudo-effective
 for some very ample divisor $\overline{A}_i\le \overline{\Sigma}_i$.
\end{itemize}
In particular, $\overline{A}_i^{d-1}\cdot (\sum \mu_{i,j}\overline{M}_{i,j})$ is bounded from above, so the number of the $M_{i,j}'$ is bounded. 

Applying Lemma \ref{lem-blowup-model} reduces the problem to the case where   
$X_i$ is obtained from $\overline{X}_i$ by a sequence of  
smooth blowups toroidal with respect to $(\overline{X}_i,\overline{\Sigma}_i)$.\\

\emph{Step 3.}
By Lemma \ref{lem-nef-div}, there is $n\in \N$ depending only on $d$ such that 
$$
|n(K_{\overline{X}_i}+3d\overline{A}_i+\overline{M}_{i,j})|
$$ 
and 
$$
|n(K_{\overline{X}_i}+3d\overline{A}_i+2\overline{M}_{i,j})|
$$
are base point free for each $i,j$. Therefore, we can write $n\overline{M}_{i,j}\sim \overline{P}_{i,j}-\overline{Q}_{i,j}$ 
where $\overline{Q}_{i,j},\overline{P}_{i,j}$ are general members of the above linear systems, respectively. 
Since $\deg_{\overline{A}_i}\overline{Q}_{i,j}$ and $\deg_{\overline{A}_i}\overline{P}_{i,j}$ are bounded from above, 
adding $\overline{Q}_{i,j},\overline{P}_{i,j}$ to $\overline{\Sigma}_i$, we can assume that $\overline{P}_{i,j}+\overline{Q}_{i,j}\le \overline{\Sigma}_{i}$.\\

\emph{Step 4.}
Replacing the sequence $(X_i,B_i+M_i)$ with a subsequence, 
we can assume that there is a log smooth couple $(\overline{V},\overline{\Sigma})$ and a smooth projective morphism 
$\overline{V}\to T$ onto a smooth variety such that for each $i$, 
there is a closed point $t_i\in T$ so that we can identify 
$\overline{X}_i$ with $\overline{V}_{t_i}$ and that $\overline{\Sigma}_i\le \overline{\Sigma}_{t_i}$ 
where the subscript $t_i$ means fibre over $t_i$, and that $\{t_i\}$ is dense in $T$. Replacing 
$\overline{\Sigma}_i$ we can assume we have $\overline{\Sigma}_i=\overline{\Sigma}_{t_i}$. 
After a finite base change and possibly shrinking $T$ 
we can assume that $(\overline{V},\overline{\Sigma})$ is log smooth over $T$ with irreducible fibres 
i.e. $I\to T$ is smooth with irreducible fibres if $I=\overline{V}$ or if $I$ is a stratum of 
$(\overline{V},\overline{\Sigma})$. In particular, there is a 1-1 correspondence between the 
strata of $(\overline{V},\overline{\Sigma})$ and the strata of $(\overline{X}_i,\overline{\Sigma}_i)$, for each $i$.

Thus, by the previous step, for each $i,j$, there exist irreducible components $E,F$ of $\overline{\Sigma}$ 
such that $E|_{\overline{X}_i}=\overline{Q}_{i,j}$ and $F|_{\overline{X}_i}=\overline{P}_{i,j}$, hence 
$(F-E)|_{\overline{X}_i}\sim n\overline{M}_{i,j}$. Therefore, passing to a subsequence, we can assume that 
there exist Cartier divisors $n\overline{M}_j$ on $\overline{V}$ such that $n\overline{M}_j|_{X_i}\sim n\overline{M}_{i,j}$ for each $i,j$ (note we are not assuming $\overline{M}_j$ to be Cartier).\\

\emph{Step 5.}
Fixing $i$, we construct a model $W_i$ over $\overline{X}_1$. 
Recall that $X_i\to \overline{X}_i$ is a sequence of smooth blowups toroidal with respect to 
$(\overline{X}_i,\Sigma_i)$. This is induced by a sequence $\alpha\colon {W}\to \overline{V}$ 
of smooth blowups toroidal with respect to $(\overline{V},\overline{\Sigma})$.
Moreover, there is a boundary $\Gamma$ on $W$ supported in 
the birational transform of $\overline{\Sigma}$ plus the exceptional divisors of $\alpha$ and with 
coefficients in $\Phi$ such that 
$\Gamma|_{X_i}=B_i$. Let $W_i$ be the fibre of ${W}\to T$ over $t_1$. Then 
we get an induced birational morphism $\beta_i\colon W_i\to \overline{X}_1$.

Define $\Gamma_i=\Gamma|_{W_i}$ and 
$$
N_i=\beta_i^*\sum \mu_{i,j}\overline{M}_{1,j}\sim_\R (\alpha^*\sum \mu_{i,j}\overline{M}_{j})|_{W_i}.
$$\

\emph{Step 6.}
By Lemma \ref{lem-vol-const-families},
$$
v_i=\vol(K_{X_i}+B_i+M_i)=\vol(K_{W_i}+\Gamma_i+N_i)
$$ 
because  
$$
M_i=(\alpha^*\sum \mu_{i,j}\overline{M}_{j})|_{X_i}, ~~\ ~~~~~~~~\ ~~ 
N_i=(\alpha^*\sum \mu_{i,j}\overline{M}_{j})|_{W_i},
$$ 
and $\alpha^*\overline{M}_{j}|_{X_i}$ and $\alpha^*\overline{M}_{j}|_{W_i}$ are nef. 

Finally, we get a contradiction by applying Proposition \ref{p-dcc-vol-gen-pair-fixed-base} 
to the generalised pairs $({W_i},\Gamma_i+N_i)$ which are equipped with the birational morphisms 
$W_i\to \overline{X}_1$.
\end{proof}

\begin{lem}\label{lem-vol-const-families}
Let $(W,\Gamma)$ be an lc pair and $W\to U$ be a contraction where $U$ is a smooth variety and $(W,\Gamma)$ is log smooth over $U$. Suppose $M_i$ are $\Q$-Cartier $\Q$-divisors on $W$, for $i=1,\dots,r$. 
Then for any two closed points $u,v\in U$ such that $M_i|_{W_u}$ and $M_i|_{W_v}$ are all nef, we have 
$$
\vol((K_{W}+\Gamma+\sum \nu_iM_{i})|_{W_u})=\vol((K_{W}+\Gamma+\sum \nu_iM_{i})|_{W_v})
$$ 
for any real numbers $\nu_i\ge 0$.
\end{lem}
\begin{proof}
First assume that $M_i=0$ for all $i$. 
If $\Gamma$ is a $\Q$-divisor, then we use [\ref{HMX1}, Theorem 1.8]. If $\Gamma$ is an $\R$-divisor, we can find an increasing sequence of $\Q$-divisors $0\le \Gamma^j\le \Gamma$ so that $\Gamma=\lim_j \Gamma^j$, and then observe that
$$
\vol((K_{W}+\Gamma)|_{W_u})=\lim_j\vol((K_{W}+\Gamma^j)|_{W_u})=\lim_j\vol((K_{W}+\Gamma^j)|_{W_v})=\vol((K_{W}+\Gamma)|_{W_v})
$$
where the middle equality again follows from [\ref{HMX1}, Theorem 1.8], and the first and third equalities follow from continuity of the volume function on numerical classes of divisors.

Now we treat the general case where $M_i$ may not be zero. 
Let ${H}$ be an ample$/U$ divisor on $W$. Since $M_i|_{W_u}$ and $M_i|_{W_v}$ are all nef,
for each $l\in\N$ the divisors 
$(\sum \nu_iM_{i}+\frac{1}{l}H)|_{W_u}$ and $(\sum \nu_iM_{i}+\frac{1}{l}H)|_{W_v}$ are ample, so we can find a 
common open neighbourhood $U_l$ of $u,v$ such that $\sum \nu_iM_{i}+\frac{1}{l}H$ 
is ample over $U_l$. Shrinking $U_l$, we can assume that over $U_l$ there is a boundary  
$$
\Theta_l\sim_\R \Gamma+\sum \nu_iM_{i}+\frac{1}{l}H/U_l
$$  
so that $(W,\Theta_l)$ is lc and log smooth over $U_l$. Then by the previous paragraph, 
$$
\vol((K_{W}+\Gamma+\sum \nu_iM_{i}+\frac{1}{l}H)|_{W_u})=\vol((K_{W}+\Theta_{l})|_{W_u})
$$
$$
=\vol((K_W+\Theta_l)|_{W_v})=\vol((K_{W}+\Gamma+\sum \nu_iM_{i}+\frac{1}{l}H)|_{W_v}). 
$$
Taking the limit when $l\to \infty$ proves the lemma. 
\end{proof}


\section{\bf Bounded birational models for $\mathcal{G}_{glc}(d,\Phi,v)$}

In this section we prove a stronger version of Theorem \ref{t-descent-nef-divs-to-bnd-models} for the 
generalised pairs in $\mathcal{G}_{glc}(d,\Phi,v)$, following [\ref{HMX3}, Lemmas 7.1 and 7.2]. 
This is used in later sections and its proof uses ideas similar to 
the proof of Theorem \ref{t-dcc-vol-gen-pairs}.

\begin{lem}\label{l-bir-bnd-model-for-fixed-vol-fixed-base}
Let $d\in \N$, $\Phi\subset \R^{\ge 0}$ be a DCC set, and $v\in \R^{>0}$. 
Let $(Y,\Delta)$ be a projective log smooth toroidal pair and 
$N_1,\dots,N_q$ be nef Cartier divisors on $Y$.  Consider
the set $\mathcal{H}$ of 
$$
(X,B+M)\in \mathcal{G}_{glc}(d,\Phi,v)
$$ 
with data $X'\overset{\phi}\to X$ and $M'=\sum \mu_iM_i'$, admitting a birational morphism $\psi \colon X\to Y$ 
with $\psi_*B\le \Delta$ and $M_i'=\phi^*\psi^*N_i$. 

Then there exists a sequence of smooth blowups $\rho \colon U \to Y$ toroidal with respect to $(Y,\Delta)$ 
such that for each 
$
(X,B+M)\in \mathcal{H},
$
\begin{enumerate}
\item $\Theta\ge {\Gamma}$ where $K_U+\Theta=\rho^*(K_Y+\Delta)$, and 
${\Gamma}$ is the sum of the exceptional divisors of $U\bir X$ plus the birational transform of $B$, and 

\item letting $N=\sum \mu_i\rho^* N_i$, we have 
$$
\vol(K_{U}+\Gamma+N)=v.
$$
\end{enumerate}
\end{lem}
\begin{proof}
Enlarging $\Phi$ we can assume $1\in \Phi$.
Consider the set $\mathcal{E}$ of 
$$
(U,\Gamma+N)\in \mathcal{G}_{glc}(d,\Phi)
$$ 
with data $U'\overset{\phi}\to U$ and $N'=\sum \mu_iN_i'$, admitting a birational morphism $\rho \colon U\to Y$ 
with $\rho_*\Gamma\le \Delta$ and $N_i'=\phi^*\rho^*N_i$. 
By Proposition \ref{p-dcc-vol-gen-pair-fixed-base}, the set 
$$
\{\vol(K_U+\Gamma+N) \mid (U,\Gamma+N)\in \mathcal{E}\}
$$
satisfies DCC, hence there is $\delta>0$ such that no member of this set belongs to $(v,v+\delta)$. 
On the other hand, by [\ref{BZh}, Theorem 8.1], there is $e\in (0,1)$ such that $K_U+e\Gamma+N$ is big 
for any $(U,\Gamma+N)\in \mathcal{E}$.

Let $\epsilon\in \R^{>0}$ so that $(1-\epsilon)^d>\frac{v}{v+\delta}$ and let $a=(1-\epsilon)+\epsilon e$. 
Then for any $(U,\Gamma+N)\in \mathcal{E}$, we have  
$$
\vol(K_U+a\Gamma+N)=\vol((1-\epsilon)(K_U+\Gamma+N)+\epsilon (K_U+e\Gamma+N))
$$
$$
\ge (1-\epsilon)^d\vol(K_U+\Gamma+N).
$$
 
Now since $a<1$, $(Y,a\Delta)$ is klt, so there is a sequence of smooth blowups 
$\rho \colon U \to Y$ toroidal with respect to $(Y,\Delta)$
 so that we can write  
$$
K_U+\Psi=\rho^*(K_Y+a\Delta)+E
$$
where $\Psi,E\ge 0$ have no common components and $(U,\Psi)$ has terminal singularities in codimension $\ge 2$. 
Let $K_U+\Theta$ be the pullback of $K_Y+\Delta$.
We will show that $(U,\Theta)$ satisfies the properties of the lemma. 

Pick $(X,B+M)\in \mathcal{H}$. We can 
replace $X$ with a high resolution and replace $B$ with the sum of its briational transform and the reduced exceptional divisor of the resolution, hence assume $\pi\colon X\bir U$ is a morphism. Let $\Gamma:=\pi_*B$ and $N:=\pi_*M$.  
Let $\Xi$ be the largest divisor satisfying $\Xi\le a\Gamma$ and $\Xi\le \Psi$. Then 
$$
\vol(K_U+a\Gamma+N)
=\vol(K_U+\Xi+N)
$$
where we use $\rho_*a\Gamma\le a\Delta$ together with either [\ref{HMX1}, Lemma 5.3(2)] adapted 
to our situation or running an MMP on $K_U+a\Gamma$ over $Y$ to get a minimal model $U'$ and then observing that the pullback of 
$K_{U'}+a\Gamma'$ to $U$ is $\le \rho^*(K_Y+a\Delta)\le K_U+\Psi$. On the other hand, 
$$
\vol(K_U+\Xi+N)\le \vol(K_X+\Xi^\sim+M)\le v
$$
where we use the fact that $(U,\Xi)$ 
has terminal singularities in codimension $\ge 2$ and that $\Xi^\sim \le aB$. Here $\Xi^\sim$ denotes 
the birational transform of $\Xi$.
 
Now observe that $(U,\Gamma+N)$ is a generalised pair with data $U':=X'\to U$ and $N':=M'$. 
In fact, $(U,\Gamma+N)\in \mathcal{E}$. Indeed $\rho$ is toroidal, so $\Theta$ is the reduced toroidal boundary; every exceptional divisor of $U\dashrightarrow X$ is toroidal over $(Y,\Delta)$, and the nonexceptional part of $\Gamma$ maps into $\Delta$ because $\psi_*B\leq\Delta$; since the coefficients of $B$ are at most $1$, we get $\Gamma\leq\Theta$; moreover, $N$ descends to $U$, hence $(U,\Gamma+N)$ is generalised lc. 

Then by the above arguments we have  
$$
v=\vol(K_X+B+M)\le \vol(K_{U}+\Gamma+N)
$$
$$
\le \frac{1}{(1-\epsilon)^d}\vol(K_{U}+a\Gamma+N)\le \frac{v}{(1-\epsilon)^d} <v+\delta, 
$$
so we get $\vol(K_{U}+\Gamma+N)=v$ by our choice of $\delta$.  
\end{proof}

\bigskip

\begin{prop}\label{p-bir-bnd-model-for-fixed-vol}
Let $d\in \N$, $\Phi\subset \R^{\ge 0}$ be a DCC set, and $v\in \R^{>0}$.
Then there exists a bounded set of couples $\mathcal{P}$ such that for each 
$$
(X,B+M)\in \mathcal{G}_{glc}(d,\Phi,v)
$$ 
with data $X'\overset{\phi}\to X$ and $M'=\sum \mu_iM_i'$, there is a log smooth couple $(\overline{X},\overline{\Sigma})\in \mathcal{P}$ and a birational map 
$\overline{X}\bir X$ such that 
\begin{itemize}
\item $\overline{\Sigma}\ge \overline{B}$ where $\overline{B}$ is the sum of the reduced exceptional 
divisor of $\overline{X}\bir X$ plus the birational transform of $B$,

\item each $M_i'$  descends to $\overline{X}$, say as $\overline{M}_i$,  and 

\item letting $\overline{M}=\sum \mu_i\overline{M}_i$, we have 
$$
\vol(K_{\overline{X}}+\overline{B}+\overline{M})=v.
$$
\end{itemize}
\end{prop}
\begin{proof}
\emph{Step 1.}
It is enough to prove the proposition for those $(X,B+M)$ in an arbitrary subset $\mathcal{G}\subseteq \mathcal{G}_{glc}(d,\Phi,v)$ (this formulation is for convenience of notation which will become clear later in the proof; of course initially we could take $\mathcal{G}$ to be the whole $\mathcal{G}_{glc}(d,\Phi,v)$).
 Applying Theorem \ref{t-descent-nef-divs-to-bnd-models}, 
there exists a bounded set of couples $\mathcal{Q}$ such that for each 
$$
(X,B+M)\in \mathcal{G}
$$ 
with data $X'\overset{\phi}\to X$ and $M'=\sum \mu_iM_i'$, there is a log smooth couple $(\overline{X},\overline{\Sigma})\in \mathcal{Q}$ and a birational map 
$\overline{X}\bir X$ such that 
\begin{enumerate}
\item $\overline{\Sigma}\ge \overline{B}$ where $\overline{B}$ is the sum of the reduced 
exceptional divisor of $\overline{X}\bir X$ plus the birational transform of $B$, 

\item each $M_i'$  descends to $\overline{X}$, say as $\overline{M}_i$, and 

\item $\overline{A}-\sum \mu_{i}\overline{M}_{i}$ is pseudo-effective
 for some very ample divisor $\overline{A}\le \overline{\Sigma}$.
\end{enumerate}
In particular, there is a fixed $l\in \N$ so that we can assume $\overline{M}_{i}\equiv 0$ for any $i>l$.
In addition, applying Lemma \ref{lem-nef-div}, there exists a number $n\in \N$ depending only on $d$ such that 
$$
|n(K_{\overline{X}}+3d\overline{A}+\overline{M}_{i})| ~~~\mbox{and}~~~|n(K_{\overline{X}}+3d\overline{A}+2\overline{M}_{i})|
$$ 
are base point free for each $i\le l$, so we can write $n\overline{M}_i\sim \overline{P}_i-\overline{Q}_i$ where $\overline{P}_i,\overline{Q}_i\ge 0$ are reduced, $(\overline{X},\overline{\Sigma}+\overline{P}_i+\overline{Q}_i)$ is log smooth, and $\overline{A}^{d-1}\cdot \overline{P}_i$ and $\overline{A}^{d-1}\cdot \overline{Q}_i$ are bounded. Adding $\overline{P}_i,\overline{Q}_i$ to $\overline{\Sigma}$ we can assume $\overline{P}_i+\overline{Q}_i\le \overline{\Sigma}$.

Although $\overline{X}, \overline{\Sigma}, \overline{B}, \overline{M}$ enjoy some of the properties of the proposition but we do not know whether $\vol(K_{\overline{X}}+\overline{B}+\overline{M})=v$. We will eventually replace $\overline{X}, \overline{\Sigma}, \overline{B}, \overline{M}$ so that the volume property is also satisfied.\\

\emph{Step 2}.
There exist finitely many couples $(\overline{V}^j,\overline{\Lambda}^j)$ log smooth over a smooth base  
$T^j$ such that for each  $(\overline{X},\overline{\Sigma})$ as in Step 1 there is a closed point $t\in T^j$, for some $j$, such that $\overline{X}$ is isomorphic to the fibre $\overline{V}^j_t$ and $\overline{\Sigma}\le \overline{\Lambda}^j_t$. Moreover, we can assume that the strata of $(\overline{V}^j,\overline{\Lambda}^j)$ are smooth over $T^j$ with irreducible fibres. 

We can then partition $\mathcal{G}$ into finitely many subsets according to the $j$, hence replacing $\mathcal{G}$ with one of the subsets, we can assume $j$ is fixed. For simplicity of notation we can remove $j$ and write the family as $(\overline{V},\overline{\Lambda})\to T$. 
Moreover, if $\overline{X}$ corresponds to the fibre $\overline{V}_t$, then we can replace $\overline{\Sigma}$ and assume $\overline{\Sigma}=\overline{\Lambda}_t$ (so we are replacing the family $\mathcal{Q}$ with the log fibres of $(\overline{V},\overline{\Lambda})\to T$). 
In particular, there is a 1-1 correspondence between the 
strata of $(\overline{V},\overline{\Lambda})$ and the strata of $(\overline{X},\overline{\Sigma})$.

Also by Step 1 and perhaps after partitioning $\mathcal{G}$ again and replacing it with one of the parts, we can assume that there exist Cartier divisors $n\overline{N}_i$ on $\overline{V}$, for $i\le l$, such that $n\overline{N}_i|_{\overline{X}}\sim n\overline{M}_{i}$ for each $i$.\\

\emph{Step 3.}
Fix a closed point $t\in T$ so that $\overline{N}_i|_{\overline{V}_t}$ are nef. 
Let 
$$
\mathcal{E}\subset \mathcal{G}_{glc}(d,\Phi\cup \frac{1}{n}\Phi,v)
$$ 
consist of those generalised pairs $(S,C+L)$ 
with data $S'\to S$ and $L'=\sum_{i\le l} \nu_iL_i'$ admitting a birational contraction 
$\alpha\colon S\to \overline{V}_t$ such that $\alpha_*C\le \overline{\Lambda}_t$ and such that $L_{i}'$ 
descends to $\overline{V}_t$ as $\overline{L}_{i}$ with $\overline{L}_{i}= n\overline{N}_i|_{\overline{V}_t}$.

Applying Lemma \ref{l-bir-bnd-model-for-fixed-vol-fixed-base} to the couple $(\overline{V}_t,\overline{\Lambda}_t)$ and the nef Cartier divisors $n\overline{N}_i|_{\overline{V}_t}$ and the family $\mathcal{E}$, there is a sequence of smooth blowups 
$
\rho_t\colon \overline{U}_t\to \overline{V}_t
$
toroidal with respect to $(\overline{V}_t,\overline{\Lambda}_t)$ so that if 
$$
K_{\overline{U}_t}+\overline{\Theta}_t=\rho_t^*(K_{\overline{V}_t}+\overline{\Lambda}_t),
$$ 
then we have the following: 
for any 
$
(S,C+L)\in \mathcal{E}
$ 
with data $S'\to S$ and $L'=\sum_{i\le l} \nu_iL_{i}'$,  
\begin{itemize}
\item $\overline{\Theta}_t\ge \overline{\Gamma}_t$ where $\overline{\Gamma}_t$ is the sum of the 
reduced exceptional divisor of 
$\overline{U}_t\bir S$ plus the birational transform of $C$, and 

\item letting $\overline{L}_t=\sum_{i\le l} \nu_i\rho_t^*\overline{L}_{i}$, we have 
$$
\vol(K_{\overline{U}_t}+\overline{\Gamma}_t+\overline{L}_t)=v.
$$
\end{itemize}

The morphism $\rho_t$ is induced by a sequence of smooth blowups 
$\rho\colon \overline{U}\to \overline{V}$ toroidal with respect to $(\overline{V},\overline{\Lambda})$. 
Let $K_{\overline{U}}+\overline{\Theta}$ be the pullback of $K_{\overline{V}}+\overline{\Lambda}$.\\

\emph{Step 4.}
Now pick any $(X,B+M)\in \mathcal{G}$ with data $X'\overset{\phi}\to X$ and $M'=\sum \mu_iM_{i}'$. 
Then there is a closed point $q\in T$ such that 
the couple in $\mathcal{Q}$ correpsponding to $(X,B+M)$ is 
$(\overline{X},\overline{\Sigma})=(\overline{V}_{q},\overline{\Lambda}_q)$. 
Applying Lemma \ref{lem-blowup-model}, we can replace $X$ hence assume $\pi\colon X\bir \overline{X}$ is a 
sequence of smooth blowups toroidal with respect to 
$(\overline{X},\overline{\Sigma})$, and that $\pi_*B\le \overline{\Sigma}$. Moreover, $\pi\colon X\to \overline{X}$ is induced by a sequence 
$\beta\colon {Z}\to \overline{V}$ of smooth blowups toroidal with respect to $(\overline{V},\overline{\Lambda})$. 
In particular, there is a boundary $C$ on $Z$ such that $(X,B)=(Z_q,C_q)$, and the set of coefficients of $C$ is the same as that of $B$. Also $\beta_*C\le \overline{\Lambda}$.

Fix $t$ as in Step 2, 
let $\psi$ denote $Z_t\to \overline{V}_t$ and let 
$$
L_t=L_t'=\sum_{i\le l} \frac{\mu_i}{n} \psi^*n\overline{N}_i|_{\overline{V}_t}.
$$ 
Then $(Z_t,C_t+L_t)$ is a generalised pair with data $Z_t'=Z_t$ and $L_t'$. Moreover, $(Z_t,C_t+L_t)$ belongs to $\mathcal{E}$
because the coefficients of $C_t$ and the $\frac{\mu_i}{n}$ are in $\Phi\cup \frac{1}{n}\Phi$, $\psi_*C_t\le \overline{\Lambda}_t$, and because    
$$
\vol(K_{Z_t}+C_t+L_t)=\vol((K_Z+C+\sum_{i\le l} {\mu_i}\beta^*\overline{N}_i)|_{Z_t})
$$
$$
=\vol((K_Z+C+\sum_{i\le l} {\mu_i}\beta^*\overline{N}_i)|_{Z_q})=\vol(K_X+B+\sum_{i\le l} {\mu_i}\pi^*\overline{M}_i))
$$
$$
=\vol(K_X+B+M)=v
$$
by Lemma \ref{lem-vol-const-families}.\\
 
\emph{Step 5}. 
Letting $\overline{\Gamma}_t$ be the 
reduced exceptional divisor of 
$\overline{U}_t\bir Z_t$ plus the birational transform of $C_t$, and letting 
$\overline{L}_t=\sum \frac{\mu_i}{n} \rho_t^*n\overline{N}_i|_{\overline{V}_t}$, by Step 2, we have 
$$
\vol(K_{\overline{U}_t}+\overline{\Gamma}_t+\overline{L}_t)=v.
$$
Moreover, if $\overline{\Gamma}$ is the sum of the birational transform of $C$ and the reduced exceptional divisor of $\overline{U}\bir Z$, then $\overline{\Gamma}_t=\overline{\Gamma}|_{\overline{U}_t}$.

Now restricting $K_{\overline{U}}+\overline{\Gamma}+\sum_{i\le l} {\mu_i} \rho^*\overline{N}_i$ to $\overline{U}_q$ and $\overline{U}_t$ and applying  Lemma \ref{lem-vol-const-families}, we get
$$
\vol(K_{\overline{U}_q}+\overline{\Gamma}_q+\overline{L}_q)=\vol(K_{\overline{U}_t}+\overline{\Gamma}_t+\overline{L}_t)=v
$$
where $\overline{\Gamma}_q:=\overline{\Gamma}|_{\overline{U}_q}$ is the sum of the birational transform of $B=C_q$ and the reduced exceptional divisor of $\overline{U}_q\bir X=Z_q$, and $\overline{L}_q=(\sum_{i\le l} {\mu_i}\rho^*\overline{N}_i)|_{\overline{U}_q}$. Denoting $\rho_q\colon \overline{U}_q\to \overline{V}_q=\overline{X}$, we get 
$$
\vol(K_{\overline{U}_q}+\overline{\Gamma}_q+\alpha_q^*\overline{M})=v
$$ 
because $\alpha_q^*\overline{M}\equiv \overline{L}_q$ (recall that $\overline{N}_i|_{\overline{V}_q}\equiv \overline{M}_i$ for $i\le l$ while $\overline{M}_i\equiv 0$ for $i>l$).

Finally, replacing $\overline{X}, \overline{\Sigma}, \overline{B}, \overline{M}$ with 
$\overline{U}_q,\overline{\Theta}_q, \overline{\Gamma}_q,\alpha_q^*\overline{M}$, respectively, and letting $\mathcal{P}$ be the log fibres of $(\overline{U},\overline{\Theta})\to T$ over closed points, we are done.
\end{proof}


\section{\bf Boundedness of generalised pairs}

In this section we prove the main results on boundedness of generalised pairs, that is, Theorems 
\ref{t-bnd-gen-pairs-vol=v} and \ref{t-disc-for-F-glc}. First we discuss how volume changes 
after pulling back a generalised log divisor under a birational contraction and then removing a  
multiple of a prime divisor. Next we treat \ref{t-disc-for-F-glc} on log discrepancies which is important 
for proving \ref{t-bnd-gen-pairs-vol=v} and also for the boundedness results on stable  
 minimal models. At the end of the section we prove \ref{t-bnd-gen-pairs-vol=v}.

\subsection{Decrease of volume}

\begin{lem}\label{l-vol-blowup-minus-exc}
Let $X$ be a normal projective variety, $Y$ be the normalisation of the blowup of $X$ at a 
closed point $x\in X$, $\phi\colon Y\to X$ be the induced morphism, and let $E\ge 0$ be a $\Q$-Cartier divisor whose support contains $\phi^{-1}\{x\}$. 
Let $A$ be an ample $\R$-divisor on $X$.
Then 
$$
\vol(\phi^*A-tE)<\vol(A)
$$ 
for any $t>0$.
\end{lem}
\begin{proof}
First we treat the case when $x$ is smooth and $E$ is the exceptional divisor.
It is enough to assume $t$ is sufficiently small. In this case, $\phi^*A-tE$ is 
ample. Thus letting $d=\dim X$ and using the fact that $E$ is contracted to a point we have 
$$
\vol(\phi^*A-tE)=(\phi^*A-tE)^d=(\phi^*A)^d+(-tE)^d<A^d=\vol(A)
$$
because  
$$
(\phi^*A)^i\cdot (-E)^{d-i}=-(\phi^*A|_E)^i\cdot (-E|_E)^{d-i-1}=0
$$ 
for each $0<i<d$ and because $-E|_E$ is ample which ensures 
$$
(-E)^d=-(-E|_E)^{d-1}<0.
$$ 

Now we prove the general case. There is a Cartier divisor $G\ge 0$ 
on $Y$ mapping to $x$ such that $-G$ is ample over $X$. Then 
$\phi^*A-sG$ is ample for a sufficiently small $s$. We claim that 
$$
\vol(\phi^*A-sG)<\vol(A).
$$
To see this let $g\colon \tilde{Y}\to Y$ be the normalisation of $Y$ and 
let $\tilde{y}$ be a smooth point on the inverse image of $G$. 
Next let $\psi\colon U\to \tilde{Y}$ be the blowup 
of $\tilde{Y}$ at $\tilde{y}$ with exceptional divisor $S$. Then by the above arguments, 
$$
\vol(\psi^*g^*(\phi^*A-s'G)-tS)<\vol(g^*(\phi^*A-s'G))
$$
for any $t>0$ where $s'\in (0,s)$. But then for $0<t\ll s-s'$ we have 
$$
\vol(\phi^*A-sG)=\vol(g^*(\phi^*A-s'G-(s-s')G))
$$
$$
\le \vol(\psi^*g^*(\phi^*A-s'G)-tS)<\vol(g^*(\phi^*A-s'G))\le \vol(A).
$$

Finally, since $\Supp E$ contains the fibre of $\phi$ over $x$, it contains $\Supp G$, 
so $lE\ge G$ for some $l>0$. Therefore, 
$\vol(\phi^*A-tE)<\vol(A)$ for any $t>0$.
\end{proof}

\begin{lem}\label{l-vol-glc-model-minus-effect-div}
Assume that 
\begin{enumerate}
\item $(X,B+M)$ is a projective generalised lc pair with data $X'\overset{\phi}\to X$ and $M'$, 
and with ample $K_X+B+M$, 

\item $\rho\colon Y\bir X$ is a birational map, $\rho^{-1}$ not contracting any divisor, 

\item $(Y,\Gamma_Y+M_Y)$ is a projective generalised lc pair with data $X'\overset{\psi}\to Y$ and $M'$,

\item $\rho_*\Gamma_Y\le B$, and 

\item $D$ is a $\Q$-Cartier prime divisor on $Y$ with 
$$
\mu_D\Gamma_Y\le 1-a(D,X,B+M)\neq 0.
$$
\end{enumerate}
Then  
$$
\vol(K_Y+\Gamma_Y+M_Y-tD)<\vol(K_X+B+M)
$$
for any $t>0$.
\end{lem}
\begin{proof}
Let $\Gamma'$ on $X'$ be the birational transform of 
$\Gamma_Y$ plus the exceptional divisors of $\psi$, and let $D'$ be the birational transform of $D$. Then 
we can write 
$$
K_{X'}+\Gamma'+M'=\psi^*(K_Y+\Gamma_Y+M_Y)+F'
$$
where $F'$ is effective and exceptional over $Y$. Moreover, $G':=\psi^*D-D'$ is effective 
and exceptional over $Y$. Thus 
$$
\vol(K_{X'}+\Gamma'+M'-tD')=\vol(\psi^*(K_Y+\Gamma_Y+M_Y)+F'+tG'-t\psi^*D)
$$
$$
=\vol(K_Y+\Gamma_Y+M_Y-tD)
$$
for any $t\ge 0$. 

Let $B'$ be the boundary on $X'$ such that its coefficient in $D'$ is $1-a(D,X,B+M)$, 
its coefficient in any exceptional$/X$ prime divisor other than $D'$ is $1$, and 
its coefficient in any non-exceptional prime divisor is the same as the coefficient in the 
birational transform $B^\sim$. Note that if $D'$ is not exceptional over $X$, then 
the coefficient of $D'$ in $B'$ is the same as its coefficient in the birational 
transform $B^\sim$. Also note that $\mu_{D'}B'>0$, by (5).

Now $B'\ge \Gamma'$: indeed, let $T'$ be a prime divisor on $X'$; 
if $T'$ is not exceptional over $X$, then $\mu_{T'}B'=\mu_{T'}B^\sim\ge \mu_{T'}\Gamma'$ by (4); 
if $T'=D'$, then  $\mu_{T'}B'\ge \mu_{T'}\Gamma'$ by (5); otherwise $T'$ is exceptional over $X$ and 
$\mu_{T'}B'=1\ge \mu_{T'}\Gamma'$ clearly holds.

By definition of $B'$, we can write 
$$
K_{X'}+B'+M'=\phi^*(K_X+B+M)+E'
$$
where $E'$ is effective and exceptional over $X$. Running an MMP on $K_{X'}+B'+M'$ 
over $X$ with scaling of some ample divisor contracts $E'$ (cf. [\ref{B-lc-flips}, Theorem 1.8]) and ends with a $\Q$-factorial generalised 
dlt pair $(X'',B''+M'')$ with $K_{X''}+B''+M''$ being the pullback of $K_X+B+M$. 
Then  
$$
\vol(K_{X'}+\Gamma'+M'-tD')\le \vol(K_{X''}+\Gamma''+M''-tD'')
$$
$$
\le\vol(K_{X''}+B''+M''-tD''),
$$
hence it is enough to verify
$$
\vol(K_{X''}+B''+M''-tD'')<\vol(K_X+B+M)
$$
for $t>0$. 

Pick a closed point $x\in \gamma(D'')$ and let $V=\gamma^{-1}\{x\}$ where $\gamma$ denotes $X''\to X$. 
Pick a small $t>0$. Run an MMP on $K_{X''}+B''+M''-tD''$   
over $X$ with scaling of some ample divisor. 
Since the MMP is a $-D''$-MMP over $X$, if $X_i''\bir X_{i+1}''$ is a step of the MMP 
and if $S_{i+1}\subset X_{i+1}''$ is a subvariety not contained in $D_{i+1}''$, then 
$S_{i+1}$ has a birational transform $S_{i}\subset X_{i}''$ not contained in $D_i''$.   
Thus replacing $X''$ with some $X_i''$ for $i\gg 0$, we can assume that the irreducible 
components of $V$ not contained in $D''$ stabilise, that is, they are not contained in the 
exceptional locus of any step of the MMP beyond $X''$. Thus by definition of MMP with scaling, 
if $S$ is a component of $V$ not contained in $D''$, then $-D''|_S$ 
is pseudo-effective. Indeed, if the MMP terminates, then on the resulting model $-D''$ is nef over $X$;
if it does not terminate, the usual scaling argument shows that $-D''$ is a limit of movable$/X$ divisors;
since $S$ is not contracted by any subsequent step, restriction to $S$ shows that $-D''|_S$ is pseudo-effective.
This implies that $V\subset D''$ otherwise we can pick $S$ so that 
it intersects $D''$ in which case $-D''|_S$ cannot be pseudo-effective. 

Now let $\pi\colon Z\to X$ be the blowup of $X$ at $x$. There is a Cartier divisor $P\ge 0$ 
on $Z$ mapping to $x$ such that $-P$ is ample over $X$. By Lemma \ref{l-vol-blowup-minus-exc},   
$$
\vol(\pi^*(K_X+B+M)-sP)<\vol(K_X+B+M)
$$
for any $s>0$.

Let $\alpha\colon W\to X''$ and $\beta \colon W\to Z$ 
be a common resolution. Since $D''$ contains the fibre of $X''\to X$ over $x$ and since $P$ 
is contracted to $x$, we have $l\alpha^*D''\ge \beta^* P$ for some $l\in \N$. But then 
$$
\vol(K_{X''}+B''+M''-tD'')=\vol(\alpha^*(K_{X''}+B''+M''-tD''))
$$
$$
\le \vol(\beta^*\pi^*(K_X+B+M)-\frac{t}{l}\beta^*P)= \vol(\pi^*(K_X+B+M)-\frac{t}{l}P)
$$
$$
<\vol(K_X+B+M)
$$ 
for $t>0$.
\end{proof}


\subsection{Log discrepancies of $\mathcal{F}_{glc}(d,\Phi,v)$}
In this subsection we prove Theorem \ref{t-disc-for-F-glc} which says that the 
log discrepancies of the generalised pairs in $\mathcal{F}_{glc}(d,\phi,v)$ not exceeding $1$ 
form a finite set. In particular, this implies that such log discrepancies are 
zero  if they are sufficiently small.

\begin{lem}\label{l-disc-fixed-pair}
Let $(X,B+M)$ be a generalised lc pair with data $X'\overset{\phi}\to X$ and $M'=\sum_1^p \mu_{i}M_{i}'$ 
where $M_i'$ are Cartier and nef$/X$ and $\mu_i\in \R^{\ge 0}$. Then for each $a\in \R^{\ge 0}$, 
the set 
$$
\{a(D,X,B+M)\le a\mid \mbox{$D$ prime divisor over $X$}\}
$$
is finite. 
\end{lem}
\begin{proof}
The lemma is obvious if $B,M'$ are $\Q$-divisors because in this case the log discrepancies belong to 
$\frac{1}{l}\Z^{\ge 0}$ for some $l\in \N$. To treat the general case we use 
approximation. First, replacing the pair with a $\Q$-factorial generalised dlt model, 
we can assume $(X,B+M)$ is $\Q$-factorial dlt. Replacing $p$, we can assume $B=\sum_1^p b_iB_i$ where 
$B_i$ are distinct prime divisors and $b_i\ge 0$. Consider the point 
$$
v=(b_1,\dots,b_p,\mu_1,\dots,\mu_p)\in \R^p\times \R^p.
$$ 
Taking the smallest rational affine subspace of $\R^p\times \R^p$ passing through $v$, we can find a minimal number $q\le 2p+1$, positive real numbers $r_j$, and points 
$$
v_j=(b_{j,1},\dots,b_{j,p},\mu_{j,1},\dots,\mu_{j,p})\in \R^p\times \R^p
$$ 
with non-negative rational coordinates sufficiently close to those of $v$, for $1\le j\le q$, such that  
\begin{itemize}
\item $\sum_1^q r_j=1$ and $\sum_1^q r_jv_j=v$, and 
\item if the $k$-th coordinate of $v$ is rational, then all $v_j$ have the same 
$k$-th coordinate as $v$. 
\end{itemize}

Letting $B^j:=\sum_1^q b_{j,i}B_i$ and $M'^j:=\sum_1^q \mu_{j,i}M_i'$, we get 
$$
\sum_{j=1}^q r_jB^j=\sum_{j=1}^qr_j(\sum_{i=1}^pb_{j,i}B_i)
=\sum_{i=1}^p(\sum_{j=1}^qr_jb_{j,i})B_i=\sum_{i=1}^pb_iB_i=B
$$ 
and similarly $M'=\sum_1^q r_jM'^j$. Thus, letting $M^j=\phi_*M'^j$, we have  
$$
(*) \ \ \sum_1^q r_j (K_X+B^j+M^j)=K_X+B+M.
$$
On the other hand, $(X,B+M)$ is $\Q$-factorial generalised dlt, 
so 
$$
(X,u(B-\rddown{B})+\rddown{B}+uM)
$$ 
is generalised dlt for some $u>1$.
Therefore,  
we can assume that the generalised pairs $(X,B^j+M^j)$, which come with data 
$X'\to X$ and $M'^j$, are all generalised dlt.

Now let $D$ be a prime divisor over $X$ with $a(D,X,B+M)\le a$. Then by $(*)$ and by $M'=\sum_1^q r_jM'^j$, 
we have 
$$
\sum_1^q r_j a(D,X,B^j+M^j)=a(D,X,B+M)\le a.
$$
Since $r_j$ are fixed and 
$$
a(D,X,B^j+M^j)\le \frac{a}{r_j},
$$
there are finitely many possibilities for the $a(D,X,B^j+M^j)$ as $B^j,M'^j$ are $\Q$-divisors. Thus 
there are also finitely many possibilities for the $a(D,X,B+M)$.
\end{proof}

\begin{proof}[Proof of Theorem \ref{t-disc-for-F-glc}]
\emph{Step 1.}
We prove both claims of the theorem together. 
Assume that the theorem does not holds. Then there is a sequence 
$(X_i,B_i+M_i)$ of generalised pairs with data 
$X_i'\overset{\phi_i}\to X_i$ and $M_i'=\sum \mu_{i,j}M_{i,j}'$, and prime divisors $D_i$ over $X_i$ such that  
$$
a(D_i,X_i,B_i+M_i)\le 1
$$
and such that the set
$$
\{a(D_i,X_i,B_i+M_i)\}_{i\in \N} \cup \{\mu_{i,j} \mid M_{i,j}'\not\equiv 0\}_{i\in \N}   
$$ 
is not finite.\\

\emph{Step 2.}
Applying Proposition \ref{p-bir-bnd-model-for-fixed-vol}, 
there exists a bounded set of couples $\mathcal{P}$ such that for each $i$ there exist a 
log smooth couple $(\overline{X}_i,\overline{\Sigma}_i)\in \mathcal{P}$ and a birational map 
$\overline{X}_i\bir X_i$ such that 
\begin{itemize}
\item $\overline{\Sigma}_i\ge \overline{B}_i$ where $\overline{B}_i$ is the sum of the reduced exceptional 
divisor of $\overline{X}_i\bir X_i$ plus the birational transform of $B_i$,

\item each $M_{i,j}'$  descends to $\overline{X}_i$, say as $\overline{M}_{i,j}$,  and 

\item letting $\overline{M}_i=\sum \mu_{i,j}\overline{M}_{i,j}$, we have 
$$
\vol(K_{\overline{X}_i}+\overline{B}_i+\overline{M}_i)=v.
$$
\end{itemize}

We can assume that $\psi_i\colon X_i'\bir \overline{X}_i$ is a morphism. Since $K_{X_i}+B_i+M_i$ is ample, 
$$
\phi_i^*(K_{X_i}+B_i+M_i)\le \psi_i^*(K_{\overline{X}_i}+\overline{B}_i+\overline{M}_i)
$$
by the negativity lemma. 
Thus, by Lemma \ref{l-rat-map-to-lc-model}, $X_i\bir \overline{X}_i$ does not contract any divisor.\\

\emph{Step 3.}
As in Step 4 of the proof of Theorem \ref{t-dcc-vol-gen-pairs}, we can assume that 
there is a couple $(\overline{V},\overline{\Sigma})$ which is relatively log smooth  
over a smooth variety $T$ such that for each $i$ there is a closed point $t_i\in T$ so that we can identify 
$(\overline{X}_i,\overline{\Sigma}_i)$ with $(\overline{V}_{t_i},\overline{\Sigma}_{t_i})$ 
where the subscript $t_i$ means fibre over $t_i$. Moreover, 
we can assume that the strata of $(\overline{V},\overline{\Sigma})$ 
have irreducible fibres over $T$. In particular, we can assume that there is a 1-1 correspondence between the 
strata of $(\overline{V},\overline{\Sigma})$ and the strata of $(\overline{X}_i,\overline{\Sigma}_i)$ for each $i$.

In addition, we can assume there exist $q\in \N$ such that for each $i$ and each $j>q$, we have  
$M_{i,j}\equiv 0$. 
Also, we can assume there exist $n\in \N$ and Cartier divisors $n\overline{N}_j$ on $\overline{V}$ such that 
$n\overline{N}_j|_{\overline{X}_i}\sim n\overline{M}_{i,j}$ for each $i$ and each $j\le q$.

For each $i$, write $\overline{B}_i=\sum b_{i,j}\overline{B}_{i,j}$ where $\overline{B}_{i,j}$ are 
distinct prime divisors. Replacing $q$ we can assume that $b_{i,j}=0$ for $j>q$. 
Since $b_{i,j},\mu_{i,j}$ are in the DCC set $\Phi$, 
replacing the sequence $(X_i,B_i+M_i)$ with a subsequence, we can assume that for each $j$ the numbers 
$b_{i,j}$ form an increasing sequence and the numbers 
$\mu_{i,j}$ also form an increasing sequence for $j\le q$. Moreover, we can assume that for each $j$, there 
is a component of $\overline{\Sigma}$ whose restriction to $\overline{X}_i$ is $\overline{B}_{i,j}$, 
for every $i$.\\

\emph{Step 4.}
For each $i$, there is $\overline{\Gamma}_i\le \overline{\Sigma}$ on $\overline{V}$ such that 
$\overline{\Gamma}_i|_{\overline{X}_i}=\overline{B}_{i}$. Moreover, for each $i$, we have 
$$\begin{aligned}
(*)~~~ \vol(K_{\overline{X}_1}+\overline{B}_1+\overline{M}_1)& = 
\vol(K_{\overline{X}_1}+\overline{B}_1+\sum_{j\le q}  \mu_{1,j}\overline{M}_{1,j}) \\
&=\vol(K_{\overline{X}_i}+\overline{B}_i+\sum_{j\le q}  \mu_{i,j}\overline{M}_{i,j})\\
&=\vol(K_{\overline{X}_i}+\overline{\Gamma}_i|_{\overline{X}_i}+\sum_{j\le q} \mu_{i,j}\overline{N}_j|_{\overline{X}_i})\\
&=\vol(K_{\overline{X}_1}+\overline{\Gamma}_i|_{\overline{X}_1}+\sum_{j\le q} \mu_{i,j}\overline{N}_j|_{\overline{X}_1})\\
&=\vol(K_{\overline{X}_1}+\overline{\Gamma}_i|_{\overline{X}_1}+\sum_{j\le q} \mu_{i,j}\overline{M}_{1,j})
\end{aligned}$$
where the fourth equality holds by Lemma \ref{lem-vol-const-families}. 

We claim that $\overline{\Gamma}_i|_{\overline{X}_1}=\overline{B}_1$: by the previous step, 
$$
\overline{R}:=\overline{\Gamma}_i|_{\overline{X}_1}-\overline{B}_1=\sum_j (b_{i,j}-b_{1,j})\overline{B}_{1,j}\ge 0
$$ 
as $b_{i,j}\ge b_{1,j}$; assume $\overline{R}\neq 0$ and let $\overline{S}$ be an irreducible component;  
then $\overline{S}$ is not exceptional over $X_1$ because the exceptional divisors have 
coefficient $1$ in $\overline{B}_1$; but then the first and the last entry in the equation $(*)$ 
cannot be equal by Lemma \ref{l-vol-ample-div-increase} applied to the pullbacks of  
$K_{{X}_1}+{B}_1+{M}_1$ and $K_{\overline{X}_1}+\overline{\Gamma}_i|_{\overline{X}_1}+\sum_{j} \mu_{i,j}\overline{M}_{1,j}$ on 
a common resolution of $X_1$ and $\overline{X}_1$, a contradiction. Therefore, $\overline{\Gamma}_i|_{\overline{X}_1}=\overline{B}_1$ and 
$b_{i,j}=b_{1,j}$ for every $i,j$.
In particular, the coefficients of all the $B_i$ put together is a finite set.

Replacing the sequence $(X_i,B_i,M_i)$ with a subsequence we can assume that $M_{i,j}'\equiv 0$ iff $M_{1,j}'\equiv 0$, for each $i,j$. Then applying Lemma \ref{l-vol-ample-div-increase} again as in the previous paragraph shows $\mu_{i,j}=\mu_{1,j}$ for every $i$ and $j$ whenever $M_{i,j}'\not\equiv 0$.
Therefore, from now on we can assume that the set 
$\{a(D_i,X_i,B_i+M_i)\}_{i\in \N}$ is not finite. In particular, we can assume $D_i$ is exceptional over $X_i$.\\

\emph{Step 5.}
From here to the end of Step 8, fix $i>1$. We can assume that $D_i$ is a divisor on $X_i'$.  
Moreover, we can assume 
$$
a(D_i,X_i,B_i+M_i)<1
$$ 
which ensures that 
$$
a(D_i,\overline{X}_i,\overline{\Sigma}_i)\le a(D_i,\overline{X}_i,\overline{B}_i)=a(D_i,\overline{X}_i,\overline{B}_i+\overline{M}_i)\le 
a(D_i,X_i,B_i+M_i)<1
$$
where the first inequality follows from $\overline{B}_i\le \overline{\Sigma}_i$, 
the equality follows from the fact that $M_i'$ descends to $\overline{X}_i$ as $\overline{M}_i$, and 
the second inequality follows from the fact that $K_X+B+M$ is ample and that 
$\overline{B}_i$ is the sum of the reduced exceptional 
divisor of $\overline{X}_i\bir X_i$ plus the birational transform of $B_i$.
Moreover, $(\overline{X}_i,\overline{\Sigma}_i)$ is log smooth and lc, and $K_{\overline{X}_i}+\overline{\Sigma}_i$ is Cartier, in particular, $(\overline{X}_i,\overline{\Sigma}_i)$ is toroidal. Therefore, we have $a(D_i,\overline{X}_i,\overline{\Sigma}_i)=0$ which by definition means that 
$D_i$ is toroidal with respect to $(\overline{X}_i,\overline{\Sigma}_i)$.\\ 

\emph{Step 6.}
There is a 
sequence $Y_i\to \overline{X}_i$ of smooth blowups, toroidal with respect to $(\overline{X}_i,\overline{\Sigma}_i)$, 
so that $D_i$ is not exceptional over 
$Y_i$. Abusing notation we denote the birational transform of $D_i$ on $Y_i$ again by $D_i$. 
Let $\Theta_i$ be the sum of the reduced exceptional divisor of $Y_i\to \overline{X}_i$ and 
the birational transform of $\overline{B}_i$. And then 
let $\Delta_i$ be the same as $\Theta_i$ except that we replace the coefficient of $D_i$ with  
$$
1-a(D_i,X_i,B_i+M_i).
$$  

Let $M_{Y_i,j}$ be the pullback of $\overline{M}_{i,j}$, and let $M_{Y_i}= \sum \mu_{i,j}M_{Y_i,j}$. 
We can assume $X_i'\bir Y_i$ is a morphism and then 
consider $(Y_i,\Delta_i+M_{Y_i})$ as a generalised pair with data $X_i'\to Y_i$ and $M_i'$ (note that $M_i'$ is the pullback of $M_{Y_i}$). 

Now we have 
$$
v=\vol(K_{X_i}+B_i+M_i)\le \vol(K_{Y_i}+\Delta_i+M_{Y_i})
$$
$$
\le \vol(K_{\overline{X}_i}+\overline{B}_i+\overline{M}_i)=v,
$$
where the first inequality follows from the definition of $\Delta_i$ and the last inequality 
follows from the fact that the pushdown of $\Delta_i$ to $\overline{X}_i$ is $\le \overline{B}_i$. 
Therefore, all the inequalities are actually equalities. 

On the other hand, the pushdown of $\Delta_i$ to $X_i$ is $B_i$ and 
$$
0<\mu_{D_i}\Delta_i=1-a(D_i,X_i,B_i+M_i).
$$
Therefore,
$$
\vol(K_{Y_i}+\Delta_i+M_{Y_i}-tD_i)<v
$$
for any $t>0$, by Lemma \ref{l-vol-glc-model-minus-effect-div}.\\ 

\emph{Step 7.}
Now $Y_i\to \overline{X}_i$ being a sequence of smooth blowups toroidal with respect to $(\overline{X}_i,\overline{\Sigma}_i)$, it is induced by a sequence $\alpha\colon W\to \overline{V}$ of smooth blowups  
toroidal with respect to $(\overline{V},\overline{\Sigma})$, 
and there is a boundary $\Delta$ on $W$ such that $\Delta|_{Y_i}=\Delta_i$. Also since $D_i$ is a component of $\Delta_i$, there is a component $J$ of $\Delta$ such that $J|_{Y_i}=D_i$. 

Let $(W_i,\Lambda_i)$ and $S_i$ be the fibres of $(W,\Delta)\to T$ and $J\to T$ over $t_1$, respectively. Then 
we get a birational morphism $\beta_i\colon W_i\to \overline{X}_1$. Let  
$M_{W_i,j}=\beta_i^*\overline{M}_{1,j}$ and let $M_{W_i}=\sum \mu_{1,j}M_{W_i,j}=\beta_i^*\overline{M}_{1}$. Then recalling from Step 4 that $\mu_{i,j}=\mu_{1,j}$ whenever $M_{i,j}'\not\equiv 0$, 
we have 
$$\begin{aligned}
\vol(K_{W_i}+\Lambda_i+M_{W_i}-tS_i)
&=\vol(K_{W_i}+\Lambda_i+\sum_{j\le q}\mu_{1,j}M_{W_i,j}-tS_i)\\ 
&=\vol(K_{W_i}+\Lambda_i+\sum_{j\le q}\mu_{i,j}\beta_i^*\overline{M}_{1,j}-tS_i) \\
& =\vol((K_{W}+\Delta+\sum_{j\le q}\mu_{i,j}\alpha^*\overline{N}_{j}-tJ)|_{W_i}) \\
& =\vol((K_{W}+\Delta+\sum_{j\le q}\mu_{i,j}\alpha^*\overline{N}_{j}-tJ)|_{Y_i})  \\
& =\vol(K_{Y_i}+\Delta_i+\sum_{j\le q}\mu_{i,j}M_{Y_i,j}-tD_i)\\
& =\vol(K_{Y_i}+\Delta_i+M_{Y_i}-tD_i)
\end{aligned}
$$
for any small $t\ge 0$, where the fourth equality follows from Lemma \ref{lem-vol-const-families}. 

We can assume that $X_1'\bir W_i$ is a morphism. Then we get 
$$
(**)~~~ \vol(K_{W_i}+\Lambda_i+M_{W_i}-tS_i)~~
\begin{cases}
=v ~~\mbox{if $t=0$,}\\
<v ~~\mbox{if $t>0$.}
\end{cases}
$$

\emph{Step 8.}
By construction, $\Delta$ on $W$ is the sum of the reduced exceptional divisor of $W\to \overline{V}$ plus 
the birational transform of $\overline{\Gamma}_i$ except that its coefficient at $J$ is 
$$
c_i:=1-a(D_i,X_i,B_i+M_i).
$$ 
Thus  $\Lambda_i$ is the sum of the reduced exceptional divisor of $W_i\to \overline{X}_1$ 
plus the birational transform of 
$$
\sum b_{i,j}\overline{B}_{1,j}=\sum b_{1,j}\overline{B}_{1,j}=\overline{B}_1
$$ 
except that its coefficient at $S_i$ is $c_i$.

We claim that 
$$
c_i=1-a(S_i,X_1,B_1+M_1)=:e_i.
$$
Let $\Omega_i$ be the same as $\Lambda_i$ except that we replace the coefficient of $S_i$ 
with $e_i$. Then $\Omega_i$ is the sum of the reduced exceptional divisor of $W_i\bir {X}_1$ 
plus the birational transform of $B_1$ except that its coefficient at $S_i$ is $e_i$.
In particular, the pushdown of $\Omega_i$ to $X_1$ is just $B_1$. 

First assume $c_i<e_i$. Then $\Omega_i\ge 0$ and $(W_i,\Omega_i+M_{W_i})$ is a generalised pair with nef part $M_1'$ (note that $M_1'$ is the pullback of $M_{W_i}$). If $\rho_i$ denotes $W_i\bir X_1$, then ${\rho_i}_*\Omega_i= B_1$ and 
$$
\mu_{S_i}\Omega_i=e_i=1-a(S_i,X_1,B_1+M_1).
$$ 
So applying Lemma \ref{l-vol-glc-model-minus-effect-div} by taking 
$(X,B+M)$, $(Y,\Gamma_Y+M_Y)$, $D$ in the lemma to be $(X_1,B_1+M_1)$, $(W_i,\Omega_i+M_{W_i})$, $S_i$ respectively, we get 
$$
\vol(K_{W_i}+\Lambda_i+M_{W_i})=\vol(K_{W_i}+\Omega_i+M_{W_i}-(e_i-c_i)S_i)<v,
$$
which contradicts formula $(**)$. 

Next assume $c_i>e_i$. Then 
$$
v=\vol(K_{W_i}+\Omega_i+M_{W_i})=\vol(K_{W_i}+\Lambda_i+M_{W_i}-(c_i-e_i)S_i)<v,
$$
where the first equality follows from the definition of $\Omega_i$ and 
the inequality follows from $(**)$. This is a contradiction. Therefore, $c_i=e_i$ as claimed.\\

\emph{Step 9.}
By the previous step and by assumptions, 
$$
a(S_i,X_1,B_1+M_1)=a(D_i,X_i,B_i+M_i)< 1
$$
for every $i$. But by Lemma \ref{l-disc-fixed-pair}, the left hand of  
$$
\{a(S_i,X_1,B_1+M_1)\}_{i\in \N}=\{a(D_i,X_i,B_i+M_i)\}_{i\in \N}
$$
is a finite set. This is a contradiction. 
\end{proof}

\subsection{Boundedness of $\mathcal{F}_{gklt}(d,\Phi,v)$}
In this subsection we prove Theorem \ref{t-bnd-gen-pairs-vol=v}. The key ingredients are 
descent of nef divisors to bounded models, more precisely, Proposition \ref{p-bir-bnd-model-for-fixed-vol}, 
together with Theorem \ref{t-disc-for-F-glc}.

\begin{proof}(of Theorem \ref{t-bnd-gen-pairs-vol=v})
\emph{Step 1.}
By Theorem \ref{t-disc-for-F-glc}, 
$$
\{a(D,X,B+M)\le 1 \mid (X,B+M)\in \mathcal{F}_{gklt}(d,\Phi,v), ~~ \mbox{$D$ prime divisor over $X$}\}
$$
is a finite set. Thus there is $\epsilon>0$ depending only on $d,\Phi,v$ 
such that every 
$$
(X,B+M)\in \mathcal{F}_{gklt}(d,\Phi,v)
$$ 
is generalised $\epsilon$-lc.\\

\emph{Step 2.}
On the other hand, by Proposition \ref{p-bir-bnd-model-for-fixed-vol}, 
there exists a bounded set of couples $\mathcal{P}$ such that for each 
$$
(X,B+M)\in \mathcal{F}_{gklt}(d,\Phi,v)
$$ 
with data $X'\overset{\phi}\to X$ and $M'=\sum \mu_i M_i'$ there exist a 
log smooth couple $(\overline{X},\overline{\Sigma})\in \mathcal{P}$ and a birational map 
$\overline{X}\bir X$ such that 
\begin{itemize}
\item $\overline{\Sigma}\ge \overline{B}$ where $\overline{B}$ is the sum of the exceptional 
divisors of $\overline{X}\bir X$ plus the birational transform of $B$,

\item each $M_{i}'$  descends to $\overline{X}$, say as $\overline{M}_{i}$,  and 

\item letting $\overline{M}=\sum \mu_{i}\overline{M}_{i}$, we have 
$$
\vol(K_{\overline{X}}+\overline{B}+\overline{M})=v.
$$
\end{itemize}

In addition, by the proof of \ref{p-bir-bnd-model-for-fixed-vol},
there is a very ample divisor $\overline{A}\ge 0$ on $\overline{X}$ such that 
$\overline{A}^d$ is bounded from above depending only on $d,\Phi,v$, and that 
$$
\overline{A}-(\overline{B}+\overline{M}), \ \ \overline{A}-(K_{\overline{X}}+\overline{B}+\overline{M})\ \ 
$$
are pseudo-effective. Perhaps after replacing $\overline{A}$ we can also assume that 
$\overline{A}-K_{\overline{X}}$ is ample.\\

\emph{Step 3.}
Fix $(X,B+M)\in \mathcal{F}_{gklt}(d,\Phi,v)$ throughout the proof.
Since $(X,B+M)$ is generalised klt, $\phi_*M_i'$ is $\Q$-Cartier if $M_i'\equiv 0$: this can be seen 
by considering a small $\Q$-factorialisation $Y\to X$ and then noting that $Y\to X$ can be decomposed 
into a finite number of extremal contractions, so we can use the cone theorem to ensure 
$\phi_*M_i'$ is $\Q$-Cartier. Thus we can remove any $M_i'\equiv 0$ from $M'$. In particular, 
we can assume that the number of the $\mu_i$ is bounded by some number $q\in \N$ otherwise 
$$
(K_{\overline{X}}+\overline{B}+\overline{M})\cdot \overline{A}^{d-1}
$$ 
would not be bounded. Moreover, by 
Theorem \ref{t-disc-for-F-glc}, the coefficients of $B$ (hence also of $\overline{B}$) and the $\mu_i$ 
all belong to a fixed finite set depending only on $d,\Phi,v$.\\

\emph{Step 4.}
Since $K_X+B+M$ is ample and since $\overline{B}$ is the sum of the exceptional 
divisors of $\overline{X}\bir X$ plus the birational transform of $B$, the negativity lemma implies 
$$
\phi^*(K_X+B+M)\le \psi^*(K_{\overline{X}}+\overline{B}+\overline{M})
$$
where $\psi$ denotes $X'\bir \overline{X}$ which we can assume to be a morphism. This in turn implies that 
$$
a(D,X,B+M)\ge a(D,{\overline{X}},\overline{B}+\overline{M})
$$
for any prime divisor $D$ over $X$. Here $({\overline{X}},\overline{B}+\overline{M})$ is a generalised pair with data $X'\to \overline{X}$ and $M'$.

By Lemma \ref{l-rat-map-to-lc-model} applied to $(X,B+M)$ and $(\overline{X},\overline{B}+\overline{M})$, $X\bir \overline{X}$ does not contract any divisor 
and $(X,B+M)$ is the generalised lc model of $(\overline{X},\overline{B}+\overline{M})$.
Let $\overline{\Gamma}$ be obtained from $\overline{B}$ by replacing each coefficient in 
$(1-\epsilon,1]$ with $1-\epsilon$. Since $(X,B+M)$ is generalised $\epsilon$-lc, 
$(X,B+M)$ is also the generalised lc model of $(\overline{X},\overline{\Gamma}+\overline{M})$. 

Now by [\ref{BZh}, Theorem 8.1], we can find rational numbers $\nu_i\le \mu_i$ belonging 
to a finite set $\Psi$ depending only on $d,\Phi,v$ 
such that letting ${N}':=\sum \nu_i {M}_i'$ and letting $\overline{N}=\sum \nu_i\overline{M}_i$, the divisor  
$K_{\overline{X}}+\overline{\Gamma}+\overline{N}$ is big and $\vol(K_{\overline{X}}+\overline{\Gamma}+\overline{N}+\overline{A})$ is bounded from above by $(2\overline{A})^d$.  
Then applying 
Lemma \ref{lem-big-log-div}, there is $m\in \N$ depending only on $d,\Psi,\Phi,\overline{A}^d$, hence only on $d,\Phi,v$, such that there exists $\overline{L}$ satisfying 
$$
0\le \overline{A}\le \overline{L}\sim m(K_{\overline{X}}+\overline{\Gamma}+\overline{N})
$$
and $m\overline{N}$ is integral. In particular, the coefficients of $\overline{L}$ 
belong to a fixed DCC set: indeed, $\overline{L}=\overline{P}+m\overline{\Gamma}$ 
for some integral divisor $\overline{P}$; the coefficients of $m\overline{\Gamma}$ are in a DCC subset of $[0,m]$; 
so the coefficients of $\overline{P}$ are $\ge -m$, hence they are in a DCC set; this implies that the coefficients 
of $\overline{L}$ are in a DCC set.\\  

\emph{Step 5.}
By construction, 
$$
K_{\overline{X}}+\overline{\Gamma}+\overline{M}\sim_\Q\frac{1}{m}\overline{L}+\overline{M}-\overline{N}.
$$
On the other hand, $(\overline{X},\overline{\Gamma})$ is $\epsilon$-lc and 
$\overline{A}-\overline{\Gamma}$ and $\overline{A}-\frac{1}{m}\overline{L}$ are 
pseudo-effective, hence by [\ref{B-BAB}, Theorem 1.8] (also see Lemma \ref{l-bnd-glct-e-lc}), 
there is a fixed $t>0$ depending only on $d,\epsilon,\overline{A}^d$, hence 
only on $d,\Phi,v$, such that 
$$
(\overline{X},\overline{\Gamma}+\frac{t}{m}\overline{L})
$$ 
is a klt pair (we view the pair as a usual pair, not a generalised pair). Replacing $t$ with $\frac{t}{2}$, we can assume  the pair is $\frac{\epsilon}{2}$-lc. 

Let 
$$
\overline{\Delta}:=\overline{\Gamma}+\frac{t}{m}\overline{L} \ \ \mbox{and} \ \  
{P}':=\sum p_iM_i':=(1+t){M}'-t{N}'
$$ 
and let $\overline{P}$ be the pushdown of $P'$ to $\overline{X}$.
Then viewing 
$
(\overline{X},\overline{\Delta}+\overline{P})
$ 
as a generalised pair with data $X'\to X$ and $P'$, we see that it is 
generalised $\frac{\epsilon}{2}$-lc as $P'$ is the pullback of $\overline{P}$. Note that $\overline{\Delta}\ge \frac{t}{m}\overline{A}$.

Moreover,
$$
K_{\overline{X}}+\overline{\Delta}+\overline{P}
= K_{\overline{X}}+\overline{\Gamma}+\overline{M}+\frac{t}{m}\overline{L}+t\overline{M}-t\overline{N}
$$
$$
\sim_\Q (1+t)(K_{\overline{X}}+\overline{\Gamma}+\overline{M}).
$$
Let $\Delta,P$ on $X$ be the pushdowns of $\overline{\Delta},\overline{P}$, and view $(X,\Delta+P)$ as a generalised pair with nef part $P'$. Then 
$$
K_X+\Delta+P\sim_\Q (1+t)(K_X+B+M)
$$
is ample and $(X,\Delta+P)$ is the generalised lc model of $(\overline{X},\overline{\Delta}+\overline{P})$.\\

\emph{Step 6.}
Recall from the end of Step 3 that $\overline{A}-K_{\overline{X}}$ is ample,
 hence  
fixing a sufficiently large $r\in \N$, there is $n\in\N$ depending only on $d$ such that 
$$
n(\overline{A}+r\overline{M}_i)= n(K_{\overline{X}}+\overline{A}-K_{\overline{X}}+r\overline{M}_i)
$$
is base point free, for each $i$, by the effective base point free theorem [\ref{kollar-ebpf}]. 
Pick general  
$$
0\le \overline{R}_i\sim n(\overline{A}+r\overline{M}_i)
$$
and let 
$$
\overline{\Theta}:=\overline{\Delta}-\sum \frac{p_i}{r}\overline{A}
+\sum \frac{p_i}{rn}\overline{R}_i.
$$
Then recalling that $p_i=(1+t)\mu_i-t\nu_i$ and that $r$ is large enough, we have $\overline{\Theta}\ge 0$ and  
 $(\overline{X},\overline{\Theta})$ is an $\frac{\epsilon}{2}$-lc pair so that 
$$
K_{\overline{X}}+\overline{\Theta}\sim_\R K_{\overline{X}}+\overline{\Delta}+\overline{P}.
$$\

\emph{Step 7.}
By construction, the coefficients of $\overline{\Theta}$ are bounded from below 
away from zero. Moreover, $(\overline{X},\overline{\Theta})$ belongs to a bounded family because $\overline{\Theta}\cdot \overline{A}^{d-1}$ is bounded from above. Additionally, if $\Theta$ is the pushdown of $\overline{\Theta}$ to $X$, then 
$$
K_X+\Theta\sim_\R (1+t)(K_X+B+M)
$$
is ample. Thus $(X,\Theta)$ is the lc model of $(\overline{X},\overline{\Theta})$, and 
$(X,\Theta)$ is $\frac{\epsilon}{2}$-lc. 
Now applying [\ref{HMX2}, Theorem 1.6] shows that $(X,\Theta)$ belongs to a bounded family.
 Therefore, $(X,B+M)$ also belongs to a bounded family in the sense that there is a 
 very ample divisor $H$ on $X$ with bounded $H^d$ and bounded
$$
(K_X+B+M)\cdot H^{d-1}.
$$
\end{proof}


\section{\bf DCC of Iitaka volumes}

In this section we prove our main result on the DCC of Iitaka volumes, that is, 
Theorem \ref{t-dcc-iitaka-volumes}. Given $(X,B)$ in $\mathcal{I}_{lc}(d,\Phi,u)$ 
or in $\mathcal{I}_{klt}(d,\Phi,<\!\!u)$ with the corresponding contraction $X\to Z$, the key 
idea is to control coefficients of the discriminant and moduli divisors appearing 
in the adjunction formula for $(X,B)\to Z$. After that we can just apply DCC  
of volumes of generalised pairs, that is, Theorem \ref{t-dcc-vol-gen-pairs}.
We start with some preparations.

\subsection{Adjunction formula for $\mathcal{I}_{lc}(d,\Phi,u)$}
Given $q\in \N$ and two $\R$-divisors $C,D$ on a normal variety $X$, 
by $C\sim_q D$ we mean that $qC\sim qD$.

\begin{lem}\label{l-bnd-torsion-index}
Let $\Phi\subset \Q^{\ge 0}$ be a DCC set. Let $\mathcal{P}$ be a bounded set of 
projective lc pairs $(X,B)$ where the coefficients of $B$ are in $\Phi$ and $K_X+B\sim_\Q 0$.
Then there exists $q\in \N$ depending only on $\Phi,\mathcal{P}$ such that $q(K_X+B)\sim 0$ 
for any $(X,B)\in \mathcal{P}$. 
\end{lem}
\begin{proof}
First applying [\ref{HMX2}, Theorem 1.5], we can assume that $\Phi$ is finite. 
Now assume the lemma does not hold. Then there is a sequence $(X_i,B_i)\in \mathcal{P}$ 
such that if $q_i\in\N$ is the smallest number satisfying $q_i(K_{X_i}+B_i)\sim 0$, then 
the $q_i$ form a strictly increasing sequence of numbers. 

For each $i$, there is a log resolution $\phi_i\colon V_i\to X_i$ so that if 
$C_i$ is the sum of the exceptional divisors of $\phi_i$ plus the birational 
transform of $B_i$, then the $(V_i,C_i)$ belong to some bounded family of pairs 
depending only on $\mathcal{P}$. Since $K_{X_i}+B_i\sim_\Q 0$, 
$$
\kappa(K_{V_i}+C_i)=\kappa_\sigma(K_{V_i}+C_i)=0.
$$
  
Replacing the sequence $(X_i,B_i)$ with a subsequence, we can assume that there is a 
pair $(V,C)$ that is log smooth over some smooth variety $T$ such that its strata have irreducible fibres over $T$, each $(V_i,C_i)$ is isomorphic to the log fibre of $(V,C)$ over some closed point $t_i\in T$, 
and the set $\{t_i\}$ is dense in $T$. Here we are making use of the fact that 
$\Phi$ is finite: indeed, we can assume that each irreducible component of $C_i$ is the restriction of an irreducible component of $C$, and moreover that 
the coefficients $b_1,\dots,b_r$ of $B=\sum b_iB_i$ is a sequence independent of the choice of $B$ and that they are also the corresponding coefficientes of $C$ ensuring $C_i=C|_{V_i}$.

Now by [\ref{HMX1}, Theorem 1.8], 
$$
 \kappa_\sigma(K_{G}+C_G)=\kappa_\sigma(K_{V_i}+C_i)=0
$$
for each $i$ where $(G,C_G)$ is the generic log fibre of $(V,C)$ over $T$. Alternatively, we can 
argue that $K_G+C_G$ is pseudo-effective by showing that an MMP on $K_V+C$ over $T$ cannot end 
with a Mori fibre space, so $0 \le \kappa_\sigma(K_{G}+C_G)$, and then use semi-continuity of 
cohomology to deduce $\kappa_\sigma(K_{G}+C_G)\le \kappa_\sigma(K_{V_i}+C_i)$, hence the equality above.

Applying [\ref{Gongyo}] to $(G,C_G)$ after passing to the algebraic closure of $k(T)$, 
we see that we have $\kappa(K_{G}+C_G)=0$ 
and that there is a sufficiently divisible $q\in\N$ such that $h^0(q(K_G+C_G))\neq 0$. 
Applying semi-continuity of cohomology we deduce that 
$h^0(q(K_{V_i}+C_i))\neq 0$ for every $i$. Therefore, for every $i$, $h^0(q(K_{X_i}+B_i))\neq 0$ and so   
$q(K_{X_i}+B_i)\sim 0$ as $K_{X_i}+B_i\sim_\Q 0$. This means $q_i\le q$ for every $i$, a contradiction.  
\end{proof}

\begin{lem}\label{l-cartier-index-moduli-div-over-curves}
Let $d,q\in\N$ and $u\in \Q^{>0}$. 
Then there exists $p\in \N$ depending only on $d,q,u$ satisfying the following. 
Assume that 
\begin{itemize}
\item $(X,B)$ is a projective lc pair of dimension $d$,

\item $f\colon X\to Z$ is a contraction onto a curve with $K_X+B\sim_\Q 0/Z$, 

\item we have an adjunction formula 
$$
(*)~~~\ \ K_X+B\sim_q f^*(K_Z+B_Z+M_Z),
$$ 

\item $A\ge 0$ is an integral divisor on $X$,

\item over some non-empty open subset $U\subset Z$: 
$(X,B+tA)$ is lc for some $t>0$, and $A$ is relatively semi-ample, and  

\item for a general fibre $F$ of $f$, we have $0<\vol(A|_F)\le u$.
\end{itemize} 
Then $pM_Z$ is integral.
\end{lem}
\begin{proof}
\emph{Step 1.}
Fix a closed point $z_0\in Z$. 
Shrinking $U$ we can assume that $z_0\notin U$, so 
it is enough to show $\mu_zpM_Z$, that is, the coefficient of $z$ in $pM_Z$, is integral for any $z\notin U$, for some $p\in\N$ depending only on $d,q,u$. 
Furthermore, by the formula $(*)$, $q(K_F+B_F)\sim 0$ for the general 
log fibres $(F,B_F)$, so shrinking $U$ again we can assume that over $U$, $qB$ is an integral divisor 
without vertical components. 

Take a log resolution $W\to X$ of $(X,B+A)$ and 
let $B_W$ be the sum of the reduced exceptional divisor of $W\to X$ plus the birational transform of 
the horizontal$/Z$ part of $B$ plus the birational transform of the reduction of the 
fibres of $f$ over $Z\setminus U$. Then every component of the fibres of $W\to Z$ over the points 
in $Z\setminus U$ is a 
component of $\rddown{B_W}$. Let $A_W$ be the birational transform of the horizontal$/Z$ part of $A$. 
Then $(W,B_W+tA_W)$ is lc for any small $t>0$ because $\rddown{B_W}$ and $A_W$ have no common components 
as $\rddown{B}$ and $A$ have no common components over the generic point of $Z$.\\

\emph{Step 2.}
Run an MMP on $K_W+B_W+tA_W$ over $X$ with scaling of some ample 
divisor for some sufficiently small $t>0$. By construction, over $f^{-1}U$, $B_W+tA_W$ is the sum of the 
reduced exceptional divisor of $W\to X$ and the birational transform of $B+tA$. Thus  
since $(X,B+tA)$ is lc on $f^{-1}U$, 
the MMP terminates over $U$, so we reach a model $Y$ so that $(Y,B_Y+tA_Y)$ is a 
$\Q$-factorial dlt model  of $(X,B+tA)$ over $U$. We can assume that $A$ does not contain any non-klt 
centre of $(X,B+tA)$ over $U$ as $t$ is sufficiently small, 
so $(Y,B_Y)$ is a $\Q$-factorial dlt model of $(X,B)$ and 
$A_Y$ is the pullback of $A$ coinciding with birational transfrom of $A$, over $U$. 
In particular, over $U$, $K_Y+B_Y\sim_\Q 0$ and 
$K_Y+B_Y+tA_Y$ is semi-ample. 
 
Next run an MMP on $K_Y+B_Y+tA_Y$ over $Z$ with scaling of some ample 
divisor. The MMP does not modify $Y$ over $U$. 
Moreover, the MMP is also an MMP on $K_Y+B_Y+tA_Y-aF_Y$ where $F_Y$ is the sum of the fibres of $Y\to Z$ 
over the points in $Z\setminus U$, 
and $a>0$ is a small number. Thus the MMP terminates with a good minimal model $V$, by [\ref{HX-closure}], 
because $K_Y+B_Y+tA_Y-aF_Y$ is semi-ample over $U$ and every non-klt centre of $(Y,B_Y+tA_Y-aF_Y)$ 
intersects the inverse image of $U$.\\ 

\emph{Step 3.}
Decreasing $t$ if necessary and running MMP as in the previous paragraph, 
we can assume that running MMP on $K_V+B_V+sA_V$ over $Z$ with scaling of an ample divisor for any 
$0<s<t$ does not contract any divisor. Thus $K_V+B_V$ is a limit of movable$/Z$ $\R$-divisors. 
We claim that $K_V+B_V\sim_\Q 0/Z$.  
Since $K_V+B_V\sim_\Q 0$ over $U$, we have $K_V+B_V\sim_\Q P_V/Z$ for some vertical$/Z$ divisor $P_V$. 
Replacing $P_V$ with $P_V+\sum h_iF_i$ where $F_i$ are the fibres of $V\to Z$ having a common component with $P_V$ 
and $h_i$ are appropriate rational numbers, we can assume that $P_V\le 0$ and that 
$\Supp P_V$ does not contain the support of any fibre of $V\to Z$. 
Since $P_V$ is a limit of movable$/Z$ $\Q$-divisors, $P_V|_S$ is pseudo-effective 
for any component $S$ of any fibre of $V\to Z$. This is possible only if $P_V=0$ otherwise we can take 
$S$ to be a component of a fibre intersecting $P_V$ but not a component of $P_V$ and get a 
contradiction. Therefore, $K_V+B_V\sim_\Q 0/Z$ as claimed. In particular, $A_V$ is semi-ample and nef and big 
over $Z$.\\

\emph{Step 4.}
Let $\phi \colon N\to X$ and $\psi\colon N\to V$ be a common resolution.
By construction, 
$$
L:=\psi^*(K_V+B_V)-\phi^*(K_X+B)
$$
is zero over $U$. Then since $K_V+B_V\sim_\Q 0/Z$ and $K_X+B\sim_\Q 0/Z$, we see that 
$L$ is the pullback of a $\Q$-divisor $R_Z$ supported in $Z\setminus U$. 
Thus we have the adjunction formula 
$$
K_V+B_V=\psi_*(\phi^*(K_X+B)+L)\sim_q g^*(K_Z+{B}_Z+R_Z+M_Z)
$$ 
where $g$ denotes $V\to Z$ and ${B}_Z+R_Z,M_Z$ are the discriminant and moduli divisors 
of $(V,B_V)\to Z$, 
respectively, and $B_Z,M_Z$ are as in $(*)$. Therefore, replacing $(X,B),A$ with $(V,B_V),A_V$ 
we can assume that $(X,B+tA)$ is $\Q$-factorial dlt for some $t>0$, 
$A$ is semi-ample over $Z$, that $qB$ is integral, and 
that $\rddown{B}$ contains $f^{-1}(Z\setminus U)$.
In particular, this means the coefficient $\mu_zB_Z=1$ for any $z\in Z\setminus U$.\\

\emph{Step 5.}
Let $X\to T/Z$ be the birational contraction defined by $A$ over $Z$. Assume $S$ is a vertical$/Z$ component of $\rddown{B}$  
that is not contracted over $T$. Then $A_S:=A|_S$ is nef and big and $\vol(A_S)\le \vol(A|_F)\le u$ where 
$F$ is a general fibre of $f$.
By divisorial adjunction, 
$$
K_S+B_S:=(K_X+B)|_S\sim_\Q 0,
$$ 
so $(S,B_S)$ is a log Calabi-Yau pair and the coefficients of $B_S$ are in a DCC 
set depending only on $q$. In fact, by [\ref{HMX2}], they belong to a fixed finite set. 

We claim that the Cartier index of $A$ at the generic point of any codimension two subvariety $I$ 
of $X$ contained in $S$ is bounded from above. We may assume $I\subset\operatorname{Supp}A$, otherwise $A$ is Cartier near the generic point of $I$. Since $A$ does not contain any 
non-klt centre of $(X,B)$, $I$ is not a non-klt centre of $(X,B)$. So, $(X,B)$ 
is plt near the generic point of $I$. Then the Cartier index of $K_X+S$ along 
$I$ is bounded otherwise the coefficient of $I$ in $B_S$ would be $1$ [\ref{Shokurov-log-flips}, Proposition 3.9]
which would imply that $I$ is a non-klt centre of $(X,B)$, a contracdiction. 
But then the Cartier index along $I$ of any Weil divisor on $X$ would also be bounded by the same 
reference.   

Replacing $A$ with a bounded multiple, 
we can assume that $A_S$ is integral. Moreover, $(S,B_S+tA_S)$ is lc for some $t>0$. Therefore, 
by [\ref{B-pol-var}, Theorem 1.7], $(S,B_S+\lambda A_S)$ is lc 
for some rational number $\lambda>0$ depending only on $d,q,u$. In particular, 
$\vol(K_S+B_S+\lambda A_S)$ is bounded from below 
away from $0$ by [\ref{HMX2}, Theorem 1.3].\\

\emph{Step 6.}
Pick $z\in Z\setminus U$. 
Write the fibre of $f$ over $z$ as $\sum m_iF_i$. Then 
$$
\sum m_i\vol((K_X+B+\lambda A)|_{F_i})=\vol((K_X+B+\lambda A)|_F)\le \lambda^{d-1}u
$$
where $F$ is a general fibre of $f$. If $F_i$ is contracted over $T$, then $A|_{F_i}$ is not big, so
$$
\vol((K_X+B+\lambda A)|_{F_i})=0.
$$
But if $F_i$ is not contracted over $T$, then 
$\vol((K_X+B+\lambda A)|_{F_i})$ is bounded from below away from $0$, by the previous step, thus 
the corresponding $m_i$ is bounded from above. 
Therefore, replacing $X$ with $T$, we can assume that the fibre $\sum m_iF_i$ over $z$ 
has bounded coefficients (the dlt property of $(X,B)$ may be lost but we do not need it anymore).

Now from the formula $(*)$ and the facts that $q(K_X+B)$ is integral and 
$K_Z+B_Z$ is Cartier near $z$, we see that $qf^*M_Z$ is integral over $z$. That is, 
$q(\mu_zM_Z)(\sum m_iF_i)$ is integral, so $q(\mu_zM_Z)m_i$ is integral for each $i$. 
But then since the $m_i$ are bounded, we see that $p\mu_zM_Z$ is integral   
for some $p$ depending only on $q$ and the $m_i$, hence depending only on $d,q,u$.
\end{proof}

\begin{lem}\label{l-adjunction-bnd-fibs}
Let $d\in\N$, $\Phi\subset \Q^{\ge 0}$ be a DCC set, and $u\in \Q^{>0}$. 
Then there exist $p,q\in \N$ depending only on $d,\Phi,u$ satisfying the following. 
Assume that 
\begin{itemize}
\item $(X,B)$ is a projective lc pair of dimension $d$,
\item the coefficients of $B$ are in $\Phi$,
\item $f\colon X\to Z$ is a contraction with $K_X+B\sim_\Q 0/Z$,
\item $A\ge 0$ is an integral divisor  on $X$ such that over the generic point of $Z$: 
$(X,B+tA)$ is lc for some $t>0$ and $A$ is relatively ample,
\item $\vol(A|_F)=u$ for the general fibres $F$ of $f$.\
\end{itemize}
Then we can write an adjunction formula 
$$
K_X+B\sim_q f^*(K_Z+B_Z+M_Z)
$$
where $pM_{Z'}$ is Cartier on some high resolution $Z'\to Z$.
\end{lem}
\begin{proof}
Let $(F,B_F)$ be a general log fibre of $(X,B)$ over $Z$ and $A_F=A|_F$. 
Then $K_F+B_F\sim_\Q 0$, the coefficients of $B_F$ 
belong to $\Phi$, $A_F$ is ample with $\vol(A_F)=u$, and $(F,B_F+tA_F)$ is lc for some $t>0$. 
Then $(F,B_F),A_F$ is a stable log Calabi-Yau pair, that is, it is a polarised 
log Calabi-Yau pair in the language of [\ref{B-pol-var}]. Therefore, $(F,\Supp(B_F+A_F))$ belongs to 
a bounded family of couples, by [\ref{B-pol-var}, Corollary 1.8]. 

Thus by Lemma \ref{l-bnd-torsion-index}, there exists $q\in \N$ depending only on $d,\Phi,u$ such that 
$q(K_F+B_F)\sim 0$. Thus the same holds if $F$ is the generic fibre of $f$.
This implies that we can find a rational function $\alpha$ on $X$ so that 
$q(K_X+B)+\Div(\alpha)$ is vertical over $Z$. So since 
$$
q(K_X+B)+\Div(\alpha)\sim_\Q 0/Z,
$$ 
we deduce that $q(K_X+B)+\Div(\alpha)$ is the pullback of a $\Q$-Cartier $\Q$-divisor $qL_Z$ on $Z$, hence 
$
K_X+B\sim_q f^*L_Z.
$ 
Therefore, we get an adjunction formula 
$$
(*)~~~~ K_X+B\sim_q f^*(K_Z+B_Z+M_Z)
$$
where $B_Z$ is the discriminat divisor and $M_Z:=L_Z-(K_Z+B_Z)$ is the moduli divisor. 
The moduli divisor $M_{Z'}$ is nef and commutes with taking higher resolutions, for some 
resolution $Z'\to Z$. 

On the other hand, the moduli divisor $M_Z$ depends only on the generic log fibre of $(X,B)$ over $Z$ 
(in a birational sense)  
and the chosen rational function $\alpha$, assuming we have fixed $K_X,K_Z$ as Weil divisors 
(cf. [\ref{PSh-II}, \S 7][\ref{B-compl}, 3.4]). 

Let $\phi\colon X'\to X$ be a log resolution so that $X'\bir Z'$ is a morphism. Let $\Delta'$ be 
the horizontal$/Z$ part of the reduced exceptional divisor of $\phi$ plus 
the birational transform of the horizontal$/Z$ part of $B$. So every non-klt centre of 
$(X',\Delta')$ is horizontal over $Z'$. Run an MMP on $K_{X'}+\Delta'$ 
over $X\times_Z Z'$ with scaling of some ample divisor. It terminates over the generic point of $Z'$ 
[\ref{B-lc-flips}, Theorem 1.8] because 
over this generic point: the MMP is an MMP over $X$ and $K_{X'}+\Delta'$ is the sum of $\phi^*(K_X+B)$ plus an effective exceptional divisor. 
So we get a model $X''$ on which $K_{X''}+\Delta''\sim_\Q 0$ over the generic point of $Z'$. 
 Then applying [\ref{B-lc-flips}, Theorem 1.5][\ref{HMX2}, Theorem 1.1] or applying 
 [\ref{HX-closure}] shows that we can run a further MMP on $K_{X''}+\Delta''$ over $Z'$ ending with a 
 good minimal model. Replacing $X''$ with the minimal model we can assume $K_{X''}+\Delta''$ 
 is semi-ample over $Z'$ defining a contraction $X''\to Z''/Z'$. 
 
 The moduli divisor of $(X'',\Delta'')\to Z''$ 
 coincides with the moduli divisor of $(X,B)\to Z$ on $Z''$ because on a common resolution of 
 $X,X''$ the pullbacks of $K_X+B$ and $K_{X''}+\Delta''$ are equal over the generic point of $Z''$. 
Also the adjunction formula $(*)$ induces an adjunction formula
$$
(**)~~~~ K_{X''}+\Delta''\sim_q f''^*(K_{Z''}+B_{Z''}+M_{Z''})
$$
where $f''$ denotes $X''\to Z''$.

It is enough to find $p\in \N$ depending only on $d,\Phi,u$ so that $pM_{Z''}$ is an integral divisor
because then $pM_{Z'}$ would be integral, hence Cartier. 
Cutting $Z''$ by hyperplane sections, we can find a curve $C\subset Z''$ and an lc pair $(S,\Gamma)$ 
over $C$ with $S\subset X''$ such that $(S,\Gamma)$ is lc, $K_S+\Gamma\sim_\Q 0/C$, 
and the coefficients of 
$\Gamma$ are in $\Phi$ (perhaps after adding $1$ to $\Phi$). Moreover, if $A''$ is the birational transform of the horizontal$/Z$ part of  
$A$ and $H=A''|_S$, then over the generic point of $C$: $(S,\Gamma+tH)$ is lc for some $t>0$ and 
$H$ is big and semi-ample with $\vol(H|_G)=u$ for the general fibres $G$ of $g\colon S\to C$. 
In addition, the adjunction formula $(**)$ induces an adjunction formula 
$$
(***)~~~~ K_S+\Gamma\sim_q g^*(K_C+B_C+M_C)
$$
where $M_C={M_{Z''}}|_C$ (cf. proof of [\ref{B-sing-fano-fib}, Lemma 3.2]). 

Now by Lemma \ref{l-cartier-index-moduli-div-over-curves}, 
$pM_C$ is integral for some $p\in \N$ depending only on $d,q,u$ hence depending only on $d,\Phi,u$. Then $pM_{Z''}$ is also integral as required 
because $C$ is a general curve on $Z''$.
\end{proof}

\subsection{DCC of Iitaka volumes}
\begin{proof}(of Theorem \ref{t-dcc-iitaka-volumes})
First we treat the lc case. Pick  
$$
(X,B)\in \mathcal{I}_{lc}(d,\Phi,u)
$$
and let $f\colon X\to Z$ be the given contraction, and let $(F,B_F)$ be a log 
general fibre of $(X,B)$ over $Z$. By assumption, there is an integral divisor 
$A\ge 0$ on $X$ such that $\vol(A|_F)=u$ and over some non-empty 
open subset of $Z$: $(X,B+tA)$ is lc for some $t>0$ and $A$ is ample.

By Lemma \ref{l-adjunction-bnd-fibs}, there exist $p,q\in \N$ depending only on 
$d,\Phi,u$ such that we can write an adjunction formula 
$$
K_X+B\sim_q f^*(K_Z+B_Z+M_Z)
$$
where $pM_{Z'}$ is Cartier on some high resolution $Z'\to Z$. Moreover, viewing $(Z,B_Z+M_Z)$ as a generalised pair with data $Z'\to Z$ and $M_{Z'}$, $(Z,B_Z+M_Z)$ is generalised lc and 
$$
\Ivol(K_X+B)=\vol(K_Z+B_Z+M_Z).
$$
In addition, by the ACC for lc thresholds [\ref{HMX2}], the coefficients of $B_Z$ 
belong to a DCC set $\Psi$ depending only on $d,\Phi$. We can assume $\frac{1}{p}\in \Psi$. 
Thus 
$$
(Z,B_Z+M_Z)\in \mathcal{G}_{glc}(\dim Z,\Psi),
$$
hence $\vol(K_Z+B_Z+M_Z)$ belongs to a DCC set depending only on $\dim Z,\Psi$, 
by Theorem \ref{t-dcc-vol-gen-pairs}. Therefore, $\Ivol(K_X+B)$ 
belongs to a DCC set depending only on $d,\Phi,u$.

Now we treat the klt case. Pick  
$$
(X,B)\in \mathcal{I}_{klt}(d,\Phi,<\!\!u),
$$
let $f\colon X\to Z$ be the given contraction, and let $(F,B_F)$ be a log 
general fibre of $(X,B)$ over $Z$. Then $(F,B_F)$ is a klt log Calabi-Yau pair, and the coefficients of $B_F$ belog to a fixed finite set [\ref{HMX2}].
By assumption, there is an integral divisor 
$A\ge 0$ on $X$ such that $A_F=A|_F$ is ample with $\vol(A|_F)<u$.

Now $(F,B_F+\lambda A_F)$ is klt for some fixed $\lambda>0$, by [\ref{B-pol-var}, Theorem 1.7], in particular, the coefficients of $B_F+\lambda A_F$ belong to a fixed finite set. 
Moreover, $(F,\Supp(B_F+\lambda A_F))$ belongs to a bounded family of couples, by 
[\ref{B-pol-var}, Corollary 1.6]. 
This implies that the Cartier index of $K_F+B_F+\lambda A_F$ is bounded, by [\ref{B-lcyf}, Lemma 3.16].
Thus 
$$
\vol(A_F)=\frac{\vol(K_F+B_F+\lambda A_F)}{\lambda^{\dim F}}
$$ 
takes finitely many possible values $u_1,\dots,u_r$. Therefore, 
$$
(X,B)\in \bigcup_j \mathcal{I}_{lc}(d,\Phi,u_j),
$$
hence the DCC follows from the DCC in the lc case.
\end{proof}


\section{\bf Boundedness of stable minimal models}

In this section we prove our main result on boundedness of strongly stable minimal models, 
that is, Theorem \ref{t-bnd-stable-mmodels-lc}. We also prove a boundedness result for klt pairs (Theorem \ref{t-bnd-stable-mmodels-klt}).

\subsection{Notation}
We introduce some notation that will be used in this section. 
Let $d\in \N$, $\Phi\subset \Q^{\ge 0}$, and $u,v\in \Q^{>0}$. 
An \emph{lc $(d,\Phi,u,v)$-stable minimal model} is an lc stable minimal model $(X,B),A$ such that    
\begin{itemize}
\item $\dim X=d$,
\item the coefficients of $B$ and $A$ are in $\Phi$,
\item $\vol(A|_F)=u$ where $F$ is a general fibre of the contraction $f\colon X\to Z$ 
defined by $K_X+B$, and 
\item 
$
\Ivol(K_X+B)=v.
$
\end{itemize}
Let $\mathcal{S}_{lc}(d,\Phi,u,v)$ consist of all the lc $(d,\Phi,u,v)$-stable minimal models.
Similarly define $\mathcal{S}_{klt}(d,\Phi,u,v)$ by replacing the lc condition of $(X,B)$ with klt.
The subsets consisting of the strongly stable minimal models are denoted $\mathcal{S}^s_{lc}(d,\Phi,u,v)$ and $\mathcal{S}^s_{klt}(d,\Phi,u,v)$, respectively.

When $0$ is not an accummulation point of $\Phi$, e.g. when it is a DCC set, 
a subset $\mathcal{E}\subset \mathcal{S}_{lc}(d,\Phi,u,v)$ is said to be a bounded family 
if there is a fixed $r\in \N$ such that 
for any $(X,B),A$ in the family we can find a very ample divisor $H$ on $X$ so that $H^d\le r$  
and 
$$
(K_X+B+A)\cdot H^{d-1}\le r.
$$
Boundedness of a subset $\mathcal{E}\subset \mathcal{S}_{lc}(d,\Phi,u,\sigma)$ is similarly defined.

\subsection{Log discrepancies on stable minimal models}
We prove a statement similar to Theorem \ref{t-disc-for-F-glc} but for stable minimal models.

\begin{lem}\label{l-disc-stable-mmodels}
Let $d\in \N$, $\Phi\subset \Q^{\ge 0}$ be a DCC set, and $u,v\in \Q^{>0}$. 
Then the set 
$$
\{a(D,X,B)\le 1 \mid (X,B),A\in \mathcal{S}_{lc}(d,\Phi,u,v), \ \mbox{$D$ prime divisor over $X$}\}
$$
is finite. 
\end{lem}
\begin{proof}
Let $f\colon X\to Z$ be the given contraction defined by the semi-ample divisor $K_X+B$. 
If $f$ is birational, then $K_Z+B_Z=f_*(K_X+B)$ is ample with volume $v$ and $(Z,B_Z)$ is lc. 
So $(Z,B_Z)$, considered as a generalised pair with trivial nef part, belongs to $\mathcal{F}_{glc}(d,\Phi,v)$. 
Thus the lemma follows from Theorem \ref{t-disc-for-F-glc} in this case because $(X,B)$ and 
$(Z,B_Z)$ have the same log discrepancies. 

From now on we assume $f$ is not birational.
Let $(F,B_F)$ be a general log fibre of $f$ and $A_F=A|_F$. By assumption, $A_F$ is 
ample with $\vol(A_F)=u$ and $(F,B_F+tA_F)$ is lc for some $t>0$. 
Then there is a rational number $\lambda>0$ depending only on $d,\Phi,u$ such that 
$(F,B_F+\lambda A_F)$ is lc, by [\ref{B-pol-var}, Theorem 6.4].
Since 
$$
\vol(K_F+B_F+\lambda A_F)=\vol(\lambda A_F)=\lambda^{\dim F}u,
$$
we see that the pair $(F,B_F+\lambda A_F)$, considered as a generalised pair with trivial nef part, belongs to  
$$
\mathcal{F}_{glc}(\dim F,\Phi+\lambda\Phi,\lambda^{\dim F}u),
$$
so the coefficients of $B_F+\lambda A_F$ belong to a fixed finite set depending only on $d,\Phi,u$, by 
Theorem \ref{t-disc-for-F-glc}. This in turn implies that the coefficients of $A_F$ belong to a fixed finite set.
Thus there is a bounded $l\in\N$ such that 
the horizontal$/Z$ part $lA^h$ is integral. Therefore, 
$$
(X,B),lA^h\in \mathcal{I}_{lc}(d,\Phi,l^{\dim F}u).
$$

By Lemma \ref{l-adjunction-bnd-fibs}, there exist $p,q\in \N$ depending only on $d,\Phi,u,l$, 
hence on $d,\Phi,u$ such that we can write an adjunction formula 
$$
K_X+B\sim_q f^*(K_Z+B_Z+M_Z).
$$
If $\dim Z=0$, the assertion follows from $q(K_X+B)\sim0$. Thus we may assume $\dim Z>0$ in which case 
$pM_{Z'}$ is Cartier on some high resolution $Z'\to Z$. Moreover, 
$(Z,B_Z+M_Z)$ is generalised lc with nef part $M_{Z'}$, and 
$$
\Ivol(K_X+B)=\vol(K_Z+B_Z+M_Z)=v.
$$
In addition, by the ACC for lc thresholds [\ref{HMX2}], the coefficients of $B_Z$ 
belong to a DCC set $\Psi$ depending only on $d,\Phi$. Adding $\frac{1}{p}$ we can assume 
$\frac{1}{p}\in \Psi$, hence  
$$
(Z,B_Z+M_Z)\in \mathcal{F}_{glc}(\dim Z,\Psi,v)
$$
where now $\Psi$ depends only on $d,\Phi,u$.

By Theorem \ref{t-disc-for-F-glc}, all the generalised log discrepancies 
of $(Z,B_Z+M_Z)$ in the interval $[0,1]$ belong to a fixed finite set $\Pi$ depending only on $\dim Z,\Psi,v$,  
hence depending only on $d,\Phi,u,v$. 

Assume $D$ is a prime divisor over $X$ with $a(D,X,B)\le 1$. If $D$ is horizontal over 
$Z$, then it determines a prime divisor $S$ over a general fibre $F$ of $f$ so that 
$a(S,F,B_F)\le 1$. Then $a(S,F,B_F)$ belongs to a finite set because by the adjunction formula above, 
we have $q(K_F+B_F)\sim 0$ which in particular means that the Cartier index of $K_F+B_F$ is bounded. 

We can then assume that $D$ is vertical over $Z$.
Perhaps after replacing $Z'$ with a higher resolution, there exists a high log resolution $X'\bir X$ such that $f'\colon X'\bir Z'$ is a morphism, 
$D$ is a divisor on $X'$, and its image on $Z'$ is a prime divisor, say $E$. 
Letting $K_{X'}+B'$ 
be the pullback of $K_X+B$ we have an adjunction formula 
$$
K_{X'}+B'\sim_q f'^*(K_{Z'}+B_{Z'}+M_{Z'}).
$$

By definition of adjunction, $\mu_EB_{Z'}=1-t$ where $t$ is the lc threshold of $f'^*E$ with 
respect to $(X',B')$ over the generic point of $E$. Since 
$$
\mu_EB_{Z'}=1-a(E,Z,B_Z+M_Z), 
$$
we see that 
$$
t=a(E,Z,B_Z+M_Z).
$$ 
On the other hand, since $f'^*E$ has integral coefficients, $t\le 1-\mu_DB'$, so we have  
$$
a(E,Z,B_Z+M_Z)=t\le 1-\mu_DB'=a(D,X,B)\le 1.
$$
Therefore, 
$$
a(E,Z,B_Z+M_Z)\in \Pi
$$ 
which means $\mu_EB_{Z'}$ belongs to a fixed finite set of rational numbers. 
In particular, perhaps after replacing $p$ with a bounded multiple, 
we can assume that $p(K_{Z'}+B_{Z'}+M_{Z'})$ is Cartier near the 
generic point of $E$. But then from 
$$
pq(K_{X'}+B')\sim pqf'^*(K_{Z'}+B_{Z'}+M_{Z'})
$$  
we see that $pqB'$ is Cartier over the generic point of $E$, so it is Cartier near the 
generic point of $D$. Therefore, 
$$
1-a(D,X,B)=\mu_DB'\in\frac{1}{pq}\Z
$$ 
which implies that 
$$
a(D,X,B)\in \frac{1}{pq}\Z\cap [0,1].
$$ 
\end{proof}

\subsection{Lc thresholds on stable minimal models}

\begin{lem}\label{l-lct-stable-lc-mmodels}
Let $d\in \N$, $\Phi\subset \Q^{\ge 0}$ be a DCC set, and $u,v,w\in \Q^{>0}$. 
Then there is $\lambda\in \Q^{>0}$ depending only on  $d,\Phi,u,v,w$ such that for any 
$$
(X,B),A\in \mathcal{S}^s_{lc}(d,\Phi,u,v)
$$
with $\vol(K_X+B+A)<w$, the pair $(X,B+\lambda A)$ is lc. 
\end{lem}
\begin{proof}
Pick $(X,B),A$ as in the lemma
and let $f\colon X\to Z$ be the given contraction defined by the semi-ample divisor $K_X+B$.
By Lemma \ref{l-disc-stable-mmodels}, 
the log discrepancies of $(X,B)$ in the interval $[0,1]$ belong to a fixed finite set 
depending only on $d,\Phi,u,v$. In particular, the coefficients of $B$ belong to a fixed finite 
set $\Phi_0\subset \Phi$. Moreover, fixing a sufficiently small $\epsilon>0$, if $D$ is a 
prime divisor over $X$ with $a(D,X,B)<\epsilon$, then $a(D,X,B)=0$.

Let $(Y,B_Y)$ be a $\Q$-factorial dlt model of $(X,B)$ and let $A_Y$ be the pullback of $A$. 
Since $(X,B+tA)$ is lc for some $t>0$, $A$ does not contain any non-klt centre of $(X,B)$, 
hence $A_Y$ is just the birational transform of $A$. In particular, the coefficients 
of $A_Y$ are in $\Phi$. Since $(Y,0)$ is klt, there is a birational contraction $Y'\to Y$ 
which extracts exactly the prime divisors $D$ over $Y$ with $a(D,Y,0)<\epsilon$.  
For such $D$ we have 
$$
a(D,X,B)=a(D,Y,B_Y)\le a(D,Y,0)< \epsilon,
$$ 
hence by the previous paragraph, 
$a(D,Y,B_Y)=0$. Therefore, if $K_{Y'}+B_{Y'}$ is the pullback of $K_Y+B_Y$, 
then each exceptional divisor of $Y'\to Y$ is a component of $\rddown{B_{Y'}}$.  
Since $(Y,B_Y+tA_Y)$ is lc for some $t>0$, we deduce that no exceptional divisor of $Y'\to Y$ 
maps into $A_Y$, hence if $A_{Y'}$ is the pullback of $A_Y$, then $A_{Y'}$ is the 
birational transform of $A_Y$. 

Now replace $(Y,B_Y),A_Y$ with $(Y',B_{Y'}),A_{Y'}$ (the dlt property maybe lost but we will not need it anymore). 
Then we have a crepant model 
$(Y,B_Y)$ of $(X,B)$ such that $(Y,0)$ is $\epsilon$-lc, the exceptional divisors 
of $Y\to X$ are components of $\rddown{B_Y}$, and $A_Y$ is the pullback of $A$ which coincides with 
the birational transform of $A$. 

By construction, 
$K_Y+B_Y+A_Y$ is nef and big, $K_Y+B_Y$ is nef, the coefficients of $B_Y$ are in a 
finite set of rational numbers, and the coefficients of $A_Y$ are in $\Phi$. 
Thus applying [\ref{B-pol-var}, Theorem 4.2] to an appropriate 
bounded multiple of $K_Y+B_Y+A_Y$, we see 
that $|m(K_Y+B_Y+A_Y)|$ defines a birational map for some $m\in\N$ depending only on 
$d,\Phi,\epsilon$, hence depending only on 
$d,\Phi,u,v$. In particular, $|m(K_X+B+A)|$ also defines a birational map.

Pick a member 
$$
M\in |m(K_X+B+A)|.
$$
Then 
$$
\vol(M)=\vol(m(K_X+B+A))<m^dw.
$$
We can assume the non-zero coefficients of $mA$ are at least $1$ and that $mB$ is integral. 
So for any component $D$ of $M$ we have $\mu_D(M+B+mA)\ge 1$ because any component of $M$ whose 
coefficient is not an integer is a component of $A$. Moreover, by construction, 
$$
M-(K_X+B+mA)\sim (m-1)(K_X+B)
$$ 
is pseudo-effective. Therefore, 
applying [\ref{B-compl}, Proposition 4.4] to $(X,B+mA),M$ 
we deduce that $(X,B+mA+M)$ is log birationally bounded, that is, 
 there is a bounded set of couples $\mathcal{P}$ and a number $c\in\Q^{>0}$ depending only on  
 $d,\Phi,m^dw$, hence only on 
$d,\Phi,u,v,w$, such that there is a log smooth couple $(\overline{X},\overline{\Sigma})\in \mathcal{P}$ 
and a birational map $\overline{X}\bir X$ such that 
\begin{itemize}
\item $\overline{\Sigma}$ contains the reduced exceptional divisor of $\overline{X}\bir X$ and the support of the birational 
transform of $B+A+M$, and 
\item if $\overline{M}$ is the pullback of $M$ to $\overline{X}$, then each coefficient of $\overline{M}$ is at most $c$. 
\end{itemize}
Here by pullback of $\overline{M}$ to $\overline{X}$ (and similarly for other divisors) 
we mean pulling back $\overline{M}$ 
to a common resolution of $X,\overline{X}$ and then pushing it down to $\overline{X}$. 

Replacing $c$ we can assume that there exists a very ample divisor $\overline{H}$ on $\overline{X}$ so that we have $\overline{M}\cdot \overline{H}^{d-1}\le c$ because $(\overline{X},\overline{\Sigma})$ is bounded, $\Supp \overline{M}\subset \overline{\Sigma}$ and the coefficients of $\overline{M}$ are bounded from above. 
Since $M-A$ is pseudo-effective, if $\overline{A}$ is the pullback of $A$ to $\overline{X}$, 
then $\overline{A}\cdot \overline{H}^{d-1}\le c$. In particular, the coefficients of $\overline{A}$ 
are bounded from above. 

Let $K_{\overline{X}}+\overline{B}$ be the pullback of 
$K_X+B$ to $\overline{X}$. Since no log discrepancy of $(X,B)$ belongs to $(0,\epsilon)$, 
no coefficient of $\overline{B}$ is in $(b,1)$ where
$b=1-\epsilon$. On the other hand, since $A$ does not contain any non-klt centre of $(X,B)$, 
we see that 
$\overline{A}$ does not contain any component of $\overline{B}$ with coefficient $1$.
Therefore, there is $\lambda\in(0,1)$ depending only on $d,\Phi,u,v,w$, such that 
no coefficient of $\overline{B}+\lambda\overline{A}$ exceeds $1$, that is, 
$(\overline{X,}\overline{B}+\lambda\overline{A})$ is sub-lc as it is log smooth. But then since 
$$
K_X+B+\lambda A=(1-\lambda)(K_X+B)+\lambda (K_X+B+A)
$$ 
is ample, we deduce that $(X,B+\lambda A)$ is lc because the pullback of $K_X+B+\lambda A$ 
to a common resolution of $X,\overline{X}$ is less than or equal to the pullback of 
$K_{\overline{X}}+\overline{B}+\lambda \overline{A}$.
\end{proof}


\subsection{Boundedness of $\mathcal{S}^s_{lc}(d,\Phi,u,\sigma)$}

\begin{proof}(of Theorem \ref{t-bnd-stable-mmodels-lc})
Pick 
$$
(X,B),A\in \mathcal{S}^s_{lc}(d,\Phi,u,\sigma)
$$
and let $f\colon X\to Z$ be the given contraction defined by the semi-ample divisor $K_X+B$.
By assumption,  
$$
\vol(K_X+B+t A)=\sigma(t)
$$ 
for any sufficiently small $t>0$. Since $K_X+B$ is nef and $K_X+B+A$ is ample, 
the left hand side of the above equality is a polynomial in $t$ for $t\in[0,1]$ as 
$$
\vol(K_X+B+t A)=(K_X+B+t A)^d
$$
for such $t$. Therefore, this volume is equal to $\sigma(t)$ for any 
 $t\in[0,1]$ as $\sigma(t)$ is also a polynomial. In particular, $\vol(K_X+B+A)=\sigma(1)$. 

In fact, $\sigma$ uniquely determines the intersection numbers $(K_X+B)^i\cdot A^{d-i}$. When $i>\dim Z$, this intersection number is $0$. But when $i=\dim Z$, we see that 
$$
(K_X+B)^i\cdot A^{d-i}=\Ivol(K_X+B)\vol(A|_F)=\Ivol(K_X+B)u>0
$$
where $F$ is a general fibre of $f$. Thus $v:=\Ivol(K_X+B)$ is determined by the data $d,u,\sigma$.

Now applying Lemma \ref{l-lct-stable-lc-mmodels}, 
there is a rational number $\lambda\in (0,1)$ depending only on $d,\Phi,u,v,\sigma(1)$, such that 
$(X,B+\lambda A)$ is lc. 

The coefficients of $B+\lambda A$ belong to the DCC set $\Phi+\lambda \Phi$ and the volume 
$$
\vol(K_X+B+\lambda A)=\sigma(\lambda)
$$ 
is fixed. Therefore, by [\ref{HMX3}], $(X,B+\lambda A)$ belongs to a 
bounded family of pairs. So there is a very ample divisor $L$ on $X$ with 
$$
L^d ~~\mbox{and}~~ (K_X+B+\lambda A)\cdot L^{d-1}
$$
bounded from above. Then 
$$
(K_X+B+A)\cdot L^{d-1}
$$
is bounded from above, so the $(X,B),A$ form a bounded family.
\end{proof}

\subsection{Boundedness of klt strongly stable minimal models}
The main point of the following result is that we can prove boundeness of certain klt strongly stable minimal models without using a polynomial volume function.

\begin{thm}\label{t-bnd-stable-mmodels-klt}
Let $d\in \N$, $\Phi\subset \Q^{\ge 0}$ be a DCC set, and $u,v,w\in \Q^{>0}$. Then the set of those 
$(X,B),A$ in $\mathcal{S}^s_{lc}(d,\Phi,u,v)$ such that $(X,B)$ is klt and $\vol(K_X+B+A)<w$ forms a 
 bounded family.
\end{thm}
\begin{proof}
Pick such a 
$
(X,B),A
$
and let $f\colon X\to Z$ be the given contraction defined by the semi-ample divisor $K_X+B$. 
By Lemma \ref{l-disc-stable-mmodels}, 
the log discrepancies of $(X,B)$ in the interval $[0,1]$ belong to a fixed finite set 
depending only on $d,\Phi,u,v$. Therefore, $(X,B)$ is $\epsilon$-lc and 
$l(K_X+B)$ is integral for some $\epsilon>0$ and $l\in\N$ 
depending only on $d,\Phi,u,v$. 

By assumption, $K_X+B+A$ is ample with volume $<w$. 
By Lemma \ref{l-lct-stable-lc-mmodels}, 
there is a rational number $\lambda\in (0,1)$ depending only on $d,\Phi,u,v,w$ such that 
$(X,B+\lambda A)$ is klt. Replacing $\lambda$ with $\frac{\lambda}{2}$, we can assume 
$(X,B+\lambda A)$ is $\frac{\epsilon}{2}$-lc.
Let 
$$
N:=l(K_X+B+\lambda A).
$$ 
Then 
$$
\vol(K_X+B+\lambda A+N)\le \vol((l+1)(K_X+B+A))<(l+1)^dw.
$$
Therefore, applying [\ref{B-pol-var}, Theorem 6.2] to $(X,B+\lambda A), N$ shows that 
$(X,B+\lambda A)$ belongs to a bounded family. This implies $(X,B),A$ belongs to a bounded family.
\end{proof}





\vspace{2cm}
\small 
\textsc{Yau Mathematical Sciences Center, JingZhai Building, Tsinghua University, Hai Dian District, Beijing, China 100084  } \endgraf
\email{birkar@tsinghua.edu.cn\\}

\end{document}